\documentclass[12pt,reqno]{amsart}

\usepackage[ngerman, italian, english]{babel}
\usepackage{amssymb}
\usepackage{hyperref}
\usepackage{mathptmx}
\usepackage{verbatim}
\usepackage[utf8]{inputenc}

\theoremstyle{plain}
\newtheorem{thm}{\bf Theorem}[section]

\newtheorem{prop}[thm]{\bf Proposition}
\newtheorem{lemma}[thm]{\bf Lemma}
\newtheorem{corollary}[thm]{\bf Corollary}
\newtheorem{conjecture}[thm]{\bf Conjecture}

\theoremstyle{definition}
\newtheorem{definition}[thm]{\bf Definition}
\theoremstyle{remark}
\newtheorem{remark}[thm]{\bf Remark}
\newtheorem{example}[thm]{\bf Example}
\theoremstyle{example}
\numberwithin{equation}{section}

\def \Hf{{\operatorname{HF}}}

\def \init{{\operatorname{in}_{\prec}}}

\def \Sym{{\operatorname{Sym}}}
\def \Ker{{\operatorname{Ker}}}
\def \Sym{{\operatorname{Sym}}}

\def \GL{{\operatorname{GL}}}

\def \reg{{\operatorname{reg}}}
\def \diag{{\operatorname{diag}}}

\def \xx{{\mathfrak{X}}}
\def \ff{{\bf f}}
\def \gg{{\bf g}}
\def \hh{{\bf h}}
\def \aa{{\alpha}}
\def \kk{\Bbbk}

\def \NN{\mathbb N}
\def \ZZ{\mathbb Z}

\def \R{\mathcal R}
\def \C{\mathcal C}

\def \Wl{L_{\lambda}W}

\def \Wc{L_{\gamma}V}
\def \Vl{L_{\lambda}V}

\def \UU{\operatorname{U}}
\def \Hom{\operatorname{Hom}}

\def\TTT{T_{\lambda,\pi,\gamma}}
\def \tl{{}^{\textup{t}\negthinspace}}
\def\YY{{\Bbb Y}}

\def\Bw{\bigwedge}
\def\Bl{\biggl}
\def\Br{\biggr}
\def\tensor{\otimes}
\def\Tensor{\bigotimes}
\def\Dirsum{\bigoplus}
\def\dirsum{\oplus}
\def\mult{\operatorname{mult}}

\def\st{{\,t}}
\def\su{{\,u}}
\def\sv{{\,v}}

\begin{document}
\title{Relations between the minors of a generic matrix}
\author{Winfried Bruns}
\address{Universit\"at Osnabr\"uck, Institut f\"ur Mathematik, 49069 Osnabr\"uck, Germany}
\email{wbruns@uos.de}
\author{Aldo Conca}
\address{Dipartimento di Matematica,
Universit\`a degli Studi di Genova, Italy} \email{conca@dima.unige.it}
\author{Matteo Varbaro}
\address{Dipartimento di Matematica,
Universit\`a degli Studi di Genova, Italy}
\email{varbaro@dima.unige.it}
\subjclass[2000]{13A50, 14M12, 14L30}
\keywords{Relations of minors; Determinantal varieties; Plethysms}
\dedicatory{To David Eisenbud on his 65th birthday}
\maketitle

\begin{abstract}
It is well-known that the Plücker relations generate the ideal of
relations of the maximal minors of a generic $m\times n$ matrix. In
this paper we discuss the relations of $t$-minors for $t<\min(m,n)$.
We will exhibit minimal relations in degrees $2$ (non-Plücker in
general) and $3$, and give some evidence for our conjecture that we
have found the generating system of the ideal of relations. The
approach is through the representation theory of the general linear group.
\end{abstract}

\section*{Introduction}
In algebra, in algebraic geometry and in representation theory the
polynomial relations between the minors of a matrix  are interesting
objects for many reasons. Surprisingly  they are still unknown in
almost all cases. While it is a classical theorem that the Pl\"ucker
relations (of maximal minors of a generic matrix) generate the
defining ideal of the Grassmannian, only a few other cases have been
treated, for example, the principal minors of a (symmetric) matrix,
see Holtz and Sturmfels \cite{HSt}, Lin and Sturmfels \cite{LSt} and
Oeding \cite{Oed}. For arbitrary $t$, the relations between the
$t$-minors of a generic matrix are certainly not understood, and in
this paper we try to investigate them.

We refer the reader to Fulton and Harris \cite{FH}, Procesi
\cite{Pr},
and Weyman \cite{We} for
background in representation theory, to Bruns and Vetter
\cite{BV} for the theory of determinantal rings, and to
\cite{B}, \cite{BC1}, \cite{BC2} and \cite{BC5} for structural
results of algebras generated by minors.

Let us consider the matrix
\begin{displaymath}
X=\begin{pmatrix}
x_{11} & x_{12} & x_{13} & x_{14} \\
x_{21} & x_{22} & x_{23} & x_{24}
\end{pmatrix}
\end{displaymath}
where the $x_{ij}$'s are indeterminates over a field $\kk$.
With $[ij]=x_{1i}x_{2j}-x_{1j}x_{2i}$,
one has
\[
[12][34]-[13][24]+[14][23]=0.
\]
This is the Pl\"ucker relation, and it is the only minimal
relation in the sense that it generates the ideal of relations.
In fact, the case $t=\min\{m,n\}$ is well understood in
general, even if anything but trivial: If $t=\min \{m,n\}$ the
Plücker relations generate the ideal of relations between the
$t$-minors of $X$. In particular, there are only quadratic
minimal relations. Similarly, other classical algebras
generated by minors, like the coordinate ring of the flag
variety, are defined by quadrics, for instance see \cite[Chap.
14]{MS}.

This changes already for $2$-minors of a $3\times 4$-matrix. To
identify a minor we have now to specify rows and columns indices.
Denote by  $[ij|pq]$ the minor of $X$ with row indices $i,j$ and
column indices $p,q$. Of course,  the Pl\"ucker relations are still
present, but they are no more sufficient.  Cubics appear among the
minimal relations, for example
\begin{equation}\label{typical cubic relation}
\det\begin{pmatrix}
[12|12]&[12|13]&[12|14]\\
[13|12]&[13|13]&[13|14]\\
[23|12]&[23|13]&[23|14]
\end{pmatrix}=0;
\end{equation}
see \cite{B}.

One reason why the case of maximal minors is easier than
the general case emerges from a representation-theoretic point
of view. Let $\kk$ be a field of characteristic $0$, $A_t$
denote the subalgebra of the polynomial ring
$\kk[X]=\kk[x_{ij}]$ generated by the $t$-minors of $X$. When
$t=m\le n$
the ring $A_t$ is the coordinate ring of the Grassmannian $G(m,n)$ of
all $m$-dimensional subspaces of a vector space $W$ of
dimension $n$. In the general case, $A_t$
is the coordinate ring of the Zariski closure of the image of
the following morphism of affine spaces:
\[
\Lambda_t \ : \ \Hom_{\kk}(W,V) \to \Hom_{\kk}(\Bw^t W,\Bw^t V),
 \ \ \Lambda_t(\phi)=\wedge^t \phi,
 \]
where $V$ is a vector space of dimension $m$. Notice that
the group $G=\GL(V)\times \GL(W)$ acts on each graded
component $(A_t)_d$ of $A_t$. If $t=\min\{m,n\}$, then each
$(A_t)_d$ is actually an \emph{irreducible} $G$-representation.
This is far from being true in the general case, and this
complicates the situation tremendously.

\vspace{2mm}

In this paper we will exhibit quadratic and cubic minimal relations between
$t$-minors, that naturally appear in a $m\times n$-matrix for $t\ge2$.
The action of $G$ on $A_t$ induces a
$G$-action also on the ideal of relations $J_t$.  Therefore it suffices to
describe the highest weight vectors of the $G$-irreducible
subrepresentations of $J_t$.

Each relation $f$  between minors  gives rise   to a  mirror
relation denoted  by $f'$, namely the one  obtained  by switching columns and rows.

The quadratic relations will be completely described in
Subsection \ref{subdeg2} in terms of the irreducible
$G$-representations associated to them and their highest weight
vectors: we call the latter $\ff_{u,v}$ where $u$ and $v$ vary
in $\{0,\dots,t\}$ and are such that $u+v$ is even and $u\neq
v$, see \eqref{expldeg2}. These correspond to Pl\"ucker
relations if and only if $u=0$ or $v=0$. So, if $t\geq 3$,
Pl\"ucker relations are not the only quadratic relations.
By construction one has $\ff_{u,v}'=\ff_{v,u}$.

As \eqref{typical cubic relation} shows, minimal cubic relations
exist already for $t=2$.
We will see that, every time $t$ increases by $1$, a new type of minimal cubic relation comes
up. We give the corresponding irreducible $G$-representations
and highest weight vectors in Subsections \ref{sec_cubic_sh}
and \ref{sec_cubic_weight}.  For a given $t$ the cubic relations we describe are of two kinds (up to mirror),  even and odd. We denote their highest weight vector by  $\gg_u$, see \eqref{expldeg3}, with $1\leq u\leq \lfloor t/2 \rfloor$ for the even relations  and  by $\hh_u$, see  \eqref{expldeg3odd}, with $2\leq u\leq  \lceil t/2 \rceil$ for the odd.
 In Subsection \ref{DetREl} we will describe how one can find
the especially appealing determinantal relations, not necessarily minimal, like
\eqref{typical cubic relation}.

We can prove that there are no further minimal cubic relations only for $t=2$ and $t=3$
(Subsections \ref{sec_min_cub} and \ref{nogeg4t=2}).
Nevertheless we conjecture that the highest weight relations
 we have identified generate the ideal of relations for all
$t,m,n$ (Conjecture \ref{conjrel}).

In Section \ref{Upper} we have collected the evidence
supporting our conjecture. To a large extent it is based on
computer calculations involving various tools like
\textsf{Singular} \cite{singular} and \textsf{Lie} \cite{lie}
and algorithms developed by the authors. Using the toric
deformation of \cite{BC1}, we first determine the
Castelnuovo-Mumford regularity of $A_t$ in Theorem
\ref{regularity} for all $t,m,n$. In conjunction with a priori
information on the Hilbert function of $A_t$, it provides
degree bounds for Gröbner basis calculations by which we have
verified the conjecture in case $t=2$ for $m,n\le 5$ and $m=4$,
$n$ arbitrary, as documented in Subsection \ref{sec4xn}. (A
duality argument, see Proposition \ref{dual}, then implies it
for $t=3$, $m=n=5$.) The result for $4\times n$ matrices is based on (the easy)
Theorem \ref{independencefromn}
 by which minimal relations of $t$-minors of a $m\times n$ matrix have already to
``live'' in an $m\times(m+t)$ matrix.

By computations based on Young symmetrizers we can exclude that
there exist degree $4$ minimal relations for $t=2$, and this
may be the strongest argument for the conjecture. (With more
effort, these computations could be pushed until  degree $6$.)
In the last two subsections \ref{T-shape_rel} and
\ref{nomoredeg3sh} we show that we have found all relations
that exist for ``very strong'' combinatorial reasons. At least,
they make it very unlikely that our relations are incomplete in
degree $3$.

To indicate our main method of proof we have to specify some
technical details. In representation theoretic terms, $A_t$ is
the subalgebra of $\kk[X]$ generated by the unique copy of the
irreducible $G$-representation $\Bw^t V\tensor \Bw^t W^*$ in
$\kk[X]$. By the universal property of the symmetric algebra
one has a presentation
\[
A_t=\Sym(E\tensor F^*)/J_t,\qquad E=\Bw^t V,\ F=\Bw^t W.
\]
 The problem we discuss is to describe a (minimal) system of
generators of $J_t$ as a ($G$-)ideal in $S_t=\Sym(E\tensor
F^*)$. It is one of the two main obstructions to the solution
of the problem that the decomposition of $S_t$ into
$G$-irreducibles is not known. (In fact, to know it is
equivalent to knowing the $\GL(V)$-decomposition of
$L_\mu(\Bw^tV)$ for all partitions $\mu$, a completely open
plethysm problem.) Fortunately, by the work of De Concini,
Eisenbud and Procesi \cite{DEP1}, from the decomposition of
$A_t$ one can link the decompositions of $S_t$ and $J_t$
easily.

In order to describe  minimal relations  we  develop combinatorial techniques to identify irreducible representations in $J_t$ and to decide whether  they are in the span of lower degree representations.

At this point it is inevitable to work simultaneously with the larger
group $H=\GL(E)\times \GL(F)$, despite the fact that $J_t$ is
not an $H$-ideal. After the introduction of some notation and
of our objects in Subsections \ref{sec_not} and \ref{sec_at},
we develop the representation theoretic structure of $S_t$ in
Subsections \ref{sec_tensalg} and \ref{Estructure}.

The intermediate subsection \ref{sec_succ} is devoted to a
formula that will allow us to derive relations with prescribed
$G$-type from lower degree relations. Lemma \ref{lemmaformula},
which may be of interest beyond our application, helps us in
specific cases to overcome the second main obstruction, namely
the lack of understanding the relationship between the algebra
structure of $S_t$ and its $G$-structure. In contrast, the
$H$-structure is well understood by \cite{DEP1}, and we can
combine it with Pieri's formula in order to (dis)prove that
certain representations in $J_t$ are minimal.

It turns out that  all  the minimal relations we have found exist for ``shape reasons''
encoded in the $G$-decompositions of the modules $L_\lambda E\tensor
L_\lambda F^*$ and Pieri's formula.  Indeed, it is our feeling, mainly based on
computational experience, that these are, roughly speaking,  the only reasons
for a irreducible $G$-representation to give a minimal relation.
The feeling  just  expressed is made more precise in Conjecture  \ref{StrongAlg}.

In view of the representation theoretic approach we will assume
throughout that the base field $\kk$ has characteristic $0$.

\section{The representation theoretic structure}\label{sec_rep}

Representation theory will guide us in our search for relations
between the $t$-minors, in proving existence and proving
non-existence. Before starting, we need to introduce some
notation.

\subsection{Notation}\label{sec_not}

Let $\kk$ be a field of characteristic $0$, $V$ a $\kk$-vector
space of dimension $n$ and $E$ a finite dimensional rational
$\GL(V)$-representation (or $\GL(V)$-module). Then $E$ can be
decomposed in irreducible $\GL(V)$-modules, which are
parametrized by partitions $\lambda =
(\lambda_1,\dots,\lambda_k)$ with $\lambda_1\geq \dots \geq
\lambda_k\geq 1$ and $\lambda_1\leq n$. More precisely, $E$ can
be written as a direct sum of \emph{Schur modules} $\Vl$ and of
their duals. Since there is a $\GL(V)$-equivariant isomorphism
$(\Vl)^*\cong L_{\lambda}(V^*)$, there is no danger in writing
$\Vl^*$ for $(\Vl)^*$, and from now on we will do it. We follow
the notation of Weyman \cite{We}, so $L_{(1,1,\dots,1)}V\cong
\Sym^d V$ and $L_{(d)}V\cong \Bw^dV$. (Fulton and Harris
\cite{FH}
use the dual convention). We will
write $\lambda\vdash d$ if $\lambda_1+\dots +\lambda_k=d$. It
might be that we will write a partition grouping the equal
terms together: For example we may write $(7^3,2,1^2)$ for
$(7,7,7,2,1,1)$. We can view a partition $\lambda$ as a
\emph{(Young) diagram} (sometimes  we will refer to it also as
a \emph{shape}), that we will still denote by $\lambda$,
namely:
\[
\lambda = \{(i,j)\in \NN \setminus \{0\}\times \NN \setminus \{0\} \ :
\ i\leq k \mbox{ and }j\leq \lambda_i\}.
\]
It is convenient to think of a diagram as  a sequence of rows
of boxes, for instance the diagram associated to the partition
$\lambda=(6,5,5,3,1)$ is
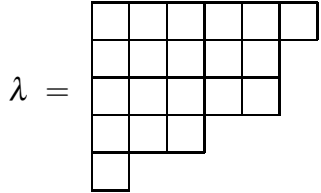
\[
{\setlength{\unitlength}{1mm}
\begin{picture}(30,25)(-5,0)

\put(-11,12){$\lambda \ =$}

\put(0,25){\line(1,0){30}}
\put(0,20){\line(1,0){30}}
\put(0,15){\line(1,0){25}}
\put(0,10){\line(1,0){25}}
\put(0,5){\line(1,0){15}}
\put(0,0){\line(1,0){5}}

\put(0,25){\line(0,-1){25}}
\put(5,25){\line(0,-1){25}}
\put(10,25){\line(0,-1){20}}
\put(15,25){\line(0,-1){20}}
\put(20,25){\line(0,-1){15}}
\put(25,25){\line(0,-1){15}}
\put(30,25){\line(0,-1){5}}

\end{picture}}
\]
Given a diagram $\lambda$, a {\it (Young) tableau} $\Lambda$ of
shape $\lambda$ on $\{1,\dots,r\}$ is a filling of the boxes of
$\lambda$ by letters in the alphabet $\{1,\dots,r\}$. For
instance, the following is a tableau of shape $(6,5,5,3,1)$ on
$\{1,\dots,7\}$:
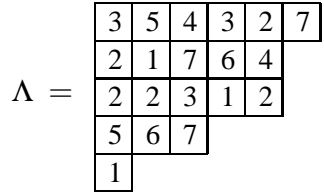
\[
{\setlength{\unitlength}{1mm}
\begin{picture}(30,25)(-5,0)

\put(-11,12){$\Lambda \ =$}

\put(0,25){\line(1,0){30}}
\put(0,20){\line(1,0){30}}
\put(0,15){\line(1,0){25}}
\put(0,10){\line(1,0){25}}
\put(0,5){\line(1,0){15}}
\put(0,0){\line(1,0){5}}

\put(0,25){\line(0,-1){25}}
\put(5,25){\line(0,-1){25}}
\put(10,25){\line(0,-1){20}}
\put(15,25){\line(0,-1){20}}
\put(20,25){\line(0,-1){15}}
\put(25,25){\line(0,-1){15}}
\put(30,25){\line(0,-1){5}}

\put(1.8,21.3){\small{$3$}}
\put(6.8,21.3){\small{$5$}}
\put(11.8,21.3){\small{$4$}}
\put(16.8,21.3){\small{$3$}}
\put(21.8,21.3){\small{$2$}}
\put(26.8,21.3){\small{$7$}}
\put(1.8,16.3){\small{$2$}}
\put(6.8,16.3){\small{$1$}}
\put(11.8,16.3){\small{$7$}}
\put(16.8,16.3){\small{$6$}}
\put(21.8,16.3){\small{$4$}}
\put(1.8,11.3){\small{$2$}}
\put(6.8,11.3){\small{$2$}}
\put(11.8,11.3){\small{$3$}}
\put(16.8,11.3){\small{$1$}}
\put(21.8,11.3){\small{$2$}}
\put(1.8,6.3){\small{$5$}}
\put(6.8,6.3){\small{$6$}}
\put(11.8,6.3){\small{$7$}}
\put(1.8,1.3){\small{$1$}}

\end{picture}}
\]
Formally, a tableau $\Lambda$ of shape $\lambda$ on
$\{1,\dots,r\}$ is a map $\Lambda:\lambda \rightarrow
\{1,\dots,r\}$. The {\it content} of $\Lambda$ is the vector
$c(\Lambda)=(c(\Lambda)_1,\dots,c(\Lambda)_r)\in \NN^r$ such
that $c(\Lambda)_p=|\{(i,j) \ : \ \Lambda(i,j)=p\}|$. A tableau
is \emph{standard} if the numbers in each row form a strictly
increasing sequence and the numbers in each column form a
weakly increasing sequence. It turns out that, once a basis of
$V$ has been fixed, let us say $e_1,\dots,e_n$, the set of
standard tableaux of shape $\lambda$ on $\{1,\dots,n\}$ is in
one-to-one correspondence with a basis of $L_\lambda V$.
Moreover, we can identify $\GL(V)$ with the group of invertible
$n\times n$-matrices with entries in $\kk$: A matrix $A\in
\GL(V)$ acts on $V$ by multiplication on the left of the column
vectors.

Let us recall the  following explicit construction of a Schur
module. Let $\lambda\vdash d$ be a diagram and $\Lambda$ be a tableau of
shape $\lambda$ such that $c(\Lambda)=(1,1,\dots,1)\in \NN^d$.
Let $\Sigma_d$ be the symmetric group on $d$ elements, and let
us define the following subsets of it:
\begin{gather*}
\C_{\Lambda}=\{\sigma \in \Sigma_d \ : \ \sigma
\mbox{ preserves each column of }\Lambda\},\\
\R_{\Lambda}=\{\tau \in \Sigma_d \ : \ \tau \mbox{ preserves
each row of }\Lambda\}.
\end{gather*}
The symmetric group $\Sigma_d$ acts on $\bigotimes^dV$
by
\[
\sigma(v_1\otimes \dots \otimes v_d)=v_{\sigma^{-1}(1)}\otimes \dots
\otimes v_{\sigma^{-1}(d)}, \ \ \ \sigma \in \Sigma_d, \ v_i\in V.
\]
and extending  $\kk$-linearly.  With these notation, the {\it Young symmetrizer} (with respect
to $\Lambda$) is the following map:
\begin{align*}
\YY_\Lambda : \bigotimes^d V & \rightarrow  \displaystyle \bigotimes^d V\\
v_1\otimes \dots \otimes v_d & \mapsto  \sum_{\sigma \in \C_{\Lambda}}
\sum_{\tau \in \R_{\Lambda}}(-1)^\tau
\sigma\tau (v_1\otimes\dots\otimes v_d).
\end{align*}
It turns out that there is a $\GL(V)$-isomorphism
$\YY_\Lambda(\bigotimes^d V)\cong L_\lambda V$. For a tableau
$\Gamma$ of shape $\lambda$ on $\{1,\dots,n\}$ we set
\[
\YY_\Lambda(\Gamma)=\YY_\Lambda(e_{\Gamma(1,1)}\otimes \cdots \otimes e_{\Gamma(1,\lambda_1)}\otimes
\dots \otimes e_{\Gamma(k,1)}\otimes \cdots \otimes e_{\Gamma(k,\lambda_k)}).
\]
Notice that $\YY_\Lambda$ is alternating in the rows of
$\lambda$: if $\Gamma'$ arises from $\Gamma$ by the exchange of
two entries in the same row, then
\[
\YY_\Lambda(\Gamma)=-\YY_\Lambda(\Gamma').
\]
In literature, the Young symmetrizers are often defined by
letting first act  the column-preserving permutations and
then the row-preserving ones. Such a definition does not yield
an alternating map. However, the two definitions lead to the
same theory, as explained in the book of Procesi \cite[Section 9.2]{Pr}.

We recall that an irreducible rational $\GL(V)$-representation $F\subseteq E$
can be identified by its highest weight. We fix a basis of $V$ so
that we can speak of diagonal or triangular matrices in $\GL(V)$. A
\emph{weight vector} of $E$ of \emph{weight}
$\alpha=(\alpha_1,\dots,\allowbreak \alpha_n)\in \ZZ^n$ is a vector
$v\in E$ such that $\diag({\bf a})v=a_1^{\alpha_1}\cdots
a_n^{\alpha_n}v$ , where $\diag({\bf a})$ is an arbitrary diagonal
matrix in $\GL(V)$ with diagonal $a_1, \dots, a_n\in \kk$. The
\emph{highest weight} of $F$ is the lexicographically largest
weight of a weight vector of $F$, and the corresponding weight vector $v$, unique up to
scalar, is called a \emph{highest weight vector}. The highest weight
is independent of the basis chosen in $V$ and represents the
irreducible representation up to isomorphism. If $E$
is polynomial, then $F\cong \Vl$ if and only if $\tl\lambda$ is the
weight of $v$. (We remind the reader that $\tl \lambda$ is the
transpose partition of $\lambda$, given by
$\tl\lambda_i=|\{j:\lambda_j\geq i\}|$.)

Let $\UU_-(V)\subseteq \GL(V)$ be the subgroup of lower triangular
matrices with $1$'s on the diagonal. Then a $\UU_-(V)$-invariant
vector $v$ of a rational representation $E$ is
the highest weight vector of an irreducible $\GL(V)$-module
$F\subseteq E$.

Given the $\GL(V)$-module $E$, we define
$$
E_\lambda
$$
to be the sum of all its irreducible $\GL(V)$-submodules that
are isomorphic to $L_\lambda V$. Then $E_\lambda\cong
(L_\lambda V)^m$ for some integer $m\ge 0$. We denote the
\emph{multiplicity $m$ of $\lambda$ in $E$} by
$$
\mult_\lambda(E).
$$
If $\mult_\lambda(E)\le 1$ for all $\lambda$, then $E$ is called
\emph{multiplicity free}. If $\mult_\lambda(E)>0$, we will say that
$\lambda$ \emph{occurs in} $E$.

We will mainly be concerned with representations of the group
$G=\GL(V)\times \GL(W)$ for vector spaces $V$ and $W$. Up to
isomorphism its irreducible polynomial representations are the
modules $L_\gamma V\tensor L_\lambda W$. Actually, we will deal
especially with the rational irreducible $G$-modules $L_\gamma
V\tensor L_\lambda W^*$. The notation just introduced will be
applied analogously to pairs $(\gamma|\lambda)$. So we will speak of
\emph{bi-diagrams} $(\gamma |\lambda)$, \emph{bi-tableaux} etc. We
have also to speak about \emph{bi-weights} and \emph{bi-weight
vectors}. The \emph{highest bi-weight vector} of $\Wc\tensor \Wl^*$
is the (unique up to scalar) $U$-invariant element of $\Wc\tensor
\Wl^*$, where $U=\UU_-(V)\times \UU_+(W)$: equivalently, it is the
element of bi-weight
$((\tl\,\gamma_1,\dots,\tl\,\gamma_h)|(-\tl\lambda_k,\dots,-\tl\lambda_1))$.

\subsection{The algebras $A_t$ and their defining
ideals}\label{sec_at}

First of all, let us introduce our objects. Let $\kk$ be a
field of characteristic $0$, $m$ and $n$ two positive integers
such that $m\leq n$ and
\[
X = \begin{pmatrix} x_{11} & x_{12} & \cdots & \cdots &  x_{1n} \\
x_{21} & x_{22} & \cdots & \cdots & x_{2n} \\
\vdots & \vdots & \ddots & \ddots & \vdots \\
x_{m1} & x_{m2} & \cdots & \cdots & x_{mn}
\end{pmatrix}
\]
a $m\times n$ matrix of indeterminates over $\kk$. Moreover let
\[R(m,n)=\kk[x_{ij} \ : \ i=1,\dots,m, \ j=1,\dots,n ]\]
be the polynomial ring in $mn$ variables over $\kk$. We are
interested in the $\kk$-subalgebra $A_t(m,n)\subseteq R(m,n)$
generated by the $t$-minors of the matrix $X$. We will use the
standard notation for a $t$-minor, namely, given two sequences
$1\leq i_1, \dots ,i_t\leq m$ and$1\leq j_1, \dots ,j_t\leq n$,
we write
\[
[i_1,\dots,i_t|j_1,\dots,j_t]
\]
for the determinant of the $t\times t$-submatrix of $X$  with  row indices  $i_1,\dots,i_t$ and the column indices  $j_1,\dots,j_t$. So we have
\[
A_t(m,n)=\kk[[i_1,\dots,i_t|j_1,\dots,j_t]:1\leq i_1<\dots <i_t\leq m,
1\leq j_1<\dots <j_t\leq n] \ \subseteq \ R(m,n).
\]
When there is no danger of confusion, we will simply write $R$
and $A_t$ instead of, respectively, $R(m,n)$ and $A_t(m,n)$.
Now let $V$ and $W$ be $\kk$-vector spaces of dimension,
respectively, $m$ and $n$. Let us fix a basis
$\{e_1,\dots,e_m\}$ of $V$ and one of $W$, say
$\{f_1,\dots,f_n\}$. We have a natural action of
$G=\GL(V)\times \GL(W)$ on $R$, namely the one induced by
\[
(A,B)\cdot X=AXB^{-1} \ \ \forall \ A\in \GL(V), B\in \GL(W).
\]
For $1\leq t\leq m$ the $\kk$-algebra $A_t$ is a $G$-invariant
subspace of $R$. Moreover this action respects the
$\NN$-grading of $R$, so, actually, any degree component $R_d$
is a finite rational $G$-representation. Moreover, the decomposition of
$R$ into irreducible $G$-modules is available, known as the
\emph{Cauchy formula}: It is easy to show that the natural
isomorphism $\Sym(V\tensor W^*)\cong R$ is $G$-equivariant, and
the Cauchy formula gives the decomposition
\begin{equation}\label{cauchy}
R_d\cong \Sym^d(V\tensor W^*)\cong
\bigoplus_{\lambda\vdash d}L_{\lambda}V\tensor L_{\lambda}W^*.
\end{equation}
where the direct sum is extended over all the partitions
$\lambda$ of $d$ such that $\lambda_1\leq m$. The decomposition
of the subrepresentation $A_t\subseteq R$ in irreducible
$G$-modules can be deduced from the work of De Concini, Eisenbud and Procesi \cite{DEP1}. Before describing it, we want
to point out that we will consider the graded structure on $A_t$ such that  the  $t$-minors have degree $1$, so that $(A_t)_d\subseteq
R_{td}$.

\begin{definition}
A partition $\lambda=(\lambda_1,\dots,\lambda_k) \vdash e$ is
called \emph{$(t,d)$-admissible} if $e=td$ and $k\leq d$
\end{definition}

We have the decomposition

\begin{equation}\label{decat}
(A_t)_d \cong \bigoplus_{\lambda \vdash td}L_{\lambda}V\tensor L_{\lambda}W^*
\end{equation}
where the direct sum runs over the  $(t,d)$-admissible partitions.  See \cite[3.3]{BC5} for this
compact description of $A_t$.

To a pair of standard tableaux of shape $\lambda$ on
$\{1,\dots,m\}$ and $\{1,\dots,n\}$, respectively, we can
associate a product of minors $\Delta\in R$ of shape $\lambda$,
namely $\Delta=\delta_1 \cdots \delta_k$ where $\delta_i$ is a
$\lambda_i$-minor. For example:
\[
{\setlength{\unitlength}{0.8mm}
\begin{picture}(30,15)(-5,5)

\put(-47,13.0){$\bigg($}

\put(-40,20){\line(1,0){15}} \put(-40,15){\line(1,0){15}}
\put(-40,10){\line(1,0){5}}

\put(-25,20){\line(0,-1){5}} \put(-30,20){\line(0,-1){5}}
\put(-35,20){\line(0,-1){10}} \put(-40,20){\line(0,-1){10}}

\put(-21,15){,}

\put(-15,20){\line(1,0){15}} \put(-15,15){\line(1,0){15}}
\put(-15,10){\line(1,0){5}}

\put(-15,20){\line(0,-1){10}} \put(-10,20){\line(0,-1){10}}
\put(-5,20){\line(0,-1){5}} \put(0,20){\line(0,-1){5}}

\put(3,13.0){$\bigg)$}

\put(-28.5,16){$4$} \put(-33.5,16){$3$} \put(-38.5,16){$1$}
\put(-38.5,11){$3$} \put(-13.5,16){$2$} \put(-8.5,16){$3$}
\put(-3.5,16){$5$} \put(-13.5,11){$2$}

\put(10,14){$\leadsto$}

\put(20,13.5){$[1,3,4|2,3,5] \cdot [3|2]$}

\end{picture}}
\]

As said in the introduction we want to understand the relations
of the $t$-minors of $X$. Therefore we have to investigate the
kernel $J_t(m,n)$ of the natural graded homomorphism
\[
\pi \ : \ S_t(m,n)= \Sym \Bl(\Bw^t V \tensor \Bw^t W^*\Br)
\to A_t(m,n).
\]
When there is no ambiguity we will just write $S_t$ and $J_t$
instead of $S_t(m,n)$ and $J_t(m,n)$.

\begin{remark}\label{partcases}
Consider the following numerical situations:
\begin{align}
\text{(a)} \ \ & t =1 \nonumber\\
\text{(b)} \ \ & n \leq t+1 \label{easycases}\\
\text{(c)} \ \ & t =m \nonumber
\end{align}
In the cases (a) and (b) the algebra $A_t$ is a polynomial ring, so
that $J_t=0$. In case (a) this is trivial, and in case (b) it
follows from the fact that the Krull dimension of $A_t$ is equal to
$mn$ if $t<m$ (see the book of Bruns and Vetter \cite[Prop.
10.16(b)]{BV}). In the case (c) $A_t$ is the coordinate ring of the
Grassmannian $G(m,n)$. In this case the ideal
$J_t$ is generated by the Pl\"ucker relations. In particular it is
generated in degree~$2$.
\end{remark}

%
Notice that the group $G$ acts in an obvious way on the
polynomial ring $S_t$. Furthermore the map $\pi$ is
$G$-equivariant. This implies that $J_t$ is a
$G$-subrepresentation of $S_t$, so that it has a decomposition
as a direct sum of irreducible representations. Moreover, if
$\Wc \tensor \Wl^*$ is an irreducible representation of $S_t$,
then it collapses to zero or it is mapped isomorphically to
itself. So \eqref{decat} implies that $\Wc \tensor \Wl^*
\subseteq J_t$ whenever $\gamma \neq \lambda$. However it is
difficult to say anything more at this point. In fact, a
decomposition of $S_t$ as direct sum of irreducible
representations is unknown, falling into the category of
\emph{plethysm problems}.

Let us note a useful duality that does not depend on
representation theory.

\begin{prop}\label{dual}
The graded algebras $A_t(n,n)$ and $A_{n-t}(n,n)$ are
isomorphic.
\end{prop}

\begin{proof}
We use the notation of \cite[Section 4]{BV}. In the coordinate
ring $G(Y)$ of the Grassmannian $G(n,2n)$ we consider the
subalgebra $P$ generated by all $n$-minors with exactly $t$
columns in the first $n$ columns of the $n\times 2n$ matrix
$Y$. The standard homomorphism $\phi$ that maps $G(Y)$ to
$\kk[Z]$ where $Z$ is an $n\times n$-matrix of indeterminates,
maps $P$ surjectively onto $A_t(n,n)$. However, $\phi\mid_P$ is
an isomorphism since the kernel of $\phi$ is generated by
$\Delta\pm1$ where $\Delta$ is the minor $[n+1,\dots,2n]$ of
$Y$. As $\phi\mid_P$ is a homomorphism of graded algebras, its
kernel is generated by homogeneous elements, but $\Delta\pm1$
has no homogeneous nonzero multiples.

If we consider dehomogenization with respect to the minor
$[1,\dots,n]$ we obtain an isomorphism of $P$ and
$A_{n-t}(n,n)$.
\end{proof}

A special case of the proposition is the isomorphism of
$A_{n-1}(n,n)$ and $A_1(n,n)$ observed above.

In the following we will often speak about ``minimal
generators" or even ``minimal subspaces'' of $J_t$. Let us make
this terminology precise. An element $x$ in $J_t$ is a
\emph{minimal generator} if its image under the natural map
$J_t\to J_t/(S_t)_1\cdot J_t$ is non-zero, and $x_1,\dots,x_n$
are said to be \emph{minimal generators} if their images in
$J_t/(S_t)_1\cdot J_t$ are $\kk$-linearly independent, in other
words, if $x_1,\dots,x_n$ can be extended to a minimal system
of generators. A $\kk$-subspace $Q$ is \emph{minimal} if the
natural map $Q\to J_t/(S_t)_1\cdot J_t$ is injective.

It should be noted that minimal relations of $t$-minors stay
minimal if the matrix is increased and can be extended to
minimal relations of $t'$-minors for $t'\ge t$. In fact, in
\cite[5.2]{BC5} the following has been proved:

\begin{prop}\label{full_retract}
$A_t(m,n)$ is a graded $\kk$-algebra retract of $A_{t'}(m',n')$
if $n'-n,m'-m\ge t'-t$.
\end{prop}

\subsection{The passage to the tensor algebra}\label{sec_tensalg}

In order to avoid the difficulties just described, we go ``one
more step to the left", in a way that we are going to outline.

Consider the Segre product  $T_t(m,n)$  of the tensor algebras $T(\Bw^t V)$ and $T(\Bw^t W^*)$ which is ($G$-equivariantly
isomorphic to) the tensor algebra $T(\Bw^t V\tensor \Bw^t W^*)$.
We have the projection from the tensor algebra to the
symmetric algebra
 \[
\phi : T_t(m,n) \to S_t(m,n).
\]
 whose kernel is a two-sided ideal generated in degree $2$.  When it does not raise confusion, we
simply write $T_t$ for $T_t(m,n)$. Finally, we have a
$G$-equivariant surjective graded homomorphism
\[
\psi = \pi \circ \phi : T_t(m,n) \to A_t(m,n).
\]
Its kernel is denoted by  $K_t(m,n)$ or simply $K_t$. Since
$\Ker(\phi)$ is generated in degree two and $J_t$ is generated
in degree at least two, in order to find the maximum degree of
a minimal generator of $J_t$ we can study the maximum degree of
a minimal generator of the two-sided ideal $K_t$. Actually we
can say more: If some element $x$ of an irreducible
subrepresentation $Q$ of $S_t$ is a minimal generator of $J_t$,
then the whole $\kk$-basis of such an irreducible
representation consists of minimal generators of $J_t$. In
fact, if $x\notin (S_t)_1\cdot (J_t)_{d-1}$, then the $G$-equivariant
map $Q\to J_t/\bigl((S_t)_1\cdot (J_t)_{d-1}\bigr)$ has to be injective. The same holds for $T_t$ and $K_t$. Therefore we
are allowed to speak about ``minimal irreducible
representations'' or ``minimal bi-shapes'' in the kernel.

\begin{lemma}\label{redtotensor}
Let $d\ge 3$ be an integer. An irreducible representation of
$(T_t)_d$ is minimal in $K_t$ if and only if is minimal in
$J_t$.
\end{lemma}

The advantages of passing to $T_t$ are that it ``separates rows
and columns'' (of the minors) and that its decomposition in
irreducible $G$-representations is available, see Proposition
\ref{tensordec}. The disadvantage is that we have to work in a
noncommutative setting. Before describing the decomposition of
$T_t$ it is convenient to introduce a definition.

\begin{definition}
We say that a diagram $\aa$ is a \emph{$t$-predecessor} (or
simply \emph{predecessor}) of a $(t,d)$-admissible diagram
$\lambda$ if $\alpha$ is $(t,d-1)$-admissible, $\alpha_1\leq
\lambda_1\leq \alpha_1+t$ and $\alpha_i\leq \lambda_i\leq
\alpha_{i-1}$ for all $i\geq 2$.

If $\aa$ is a predecessor of $\lambda$, then $\lambda$ is a
\emph{successor} of $\aa$.
\end{definition}

The notion of predecessor (or successor) reflects \emph{Pieri's
formula} (for example, see \cite[Corollary 2.3.5]{We}):
\begin{equation}
L_\aa V\tensor \Bw^t V\cong \Dirsum L_\lambda V\label{Pieri}
\end{equation}
where $\lambda$ runs through the successors of $\aa$.

\begin{prop}\label{tensordec}
As a $G$-representation, $(T_t)_d$ decomposes as
\[
(T_t)_d \cong \bigoplus_{\gamma ,\lambda} (\Wc \tensor \Wl^*)^{n(\gamma , \lambda)},
\]
where the sum runs over the $(t,d)$-admissible diagrams
$\gamma$ and
$\lambda$ with
$\gamma_1\leq m$, $\lambda_1\leq n$; the multiplicity $n(\gamma
,\lambda)=\mult_{(\gamma|\lambda)}(T_t)$ is a positive integer,
described recursively as follows:
\begin{enumerate}
\item If $\gamma  = \lambda = (t)$, then $n(\gamma,
    \lambda)=1$;
\item If $\gamma$ and $\lambda$ are $(t,d)$-admissible
    partitions with $d>1$, then $n(\gamma , \lambda) = \sum
    n(\alpha ,\beta)$ where the sum runs over all
    $t$-bi-predecessors $(\alpha|\beta)$ of $(\gamma|\lambda)$.
\end{enumerate}
\end{prop}
\begin{proof}
It is enough to find a decomposition of $\Tensor^d \Bw^t V$ as
a $\GL(V)$-represen\-ta\-tion and of $\Tensor^d \Bw^t W^*$ as a
$\GL(W)$-representation. (As mentioned above, the irreducible
$G$-repre\-sen\-tations in $(T_t)_d$  are all of type $\Wc
\tensor \Wl^*$ where $\Wc$ is an irreducible
$\GL(V)$-represen\-ta\-tion in $\Tensor^d \Bw^t V$ and $\Wl^*$
is an irreducible $\GL(W)$-representation in $\Tensor^d \Bw^t
W^*$). Now Pieri's formula and an induction easily yield the
conclusion.
\end{proof}

While the decompositions described in \eqref{cauchy} and in
\eqref{decat} are multiplicity free, the numbers
$n(\gamma,\lambda)$ may be, and in fact usually are, bigger
than $1$. As the reader will realize in the course of the
paper, this is a major obstacle to saying something about the
relations between minors.

Since the decomposition of $A_t$ is known, we can easily
compare the decompositions of $S_t$ and $J_t$. In Section
\ref{QuadCub} the comparison will allow us to identify certain
minimal relations. The next proposition follows immediately
from \eqref{decat}.

\begin{prop}\label{multJ_t}
Let $\gamma$ and $\lambda$ be $(t,d)$-admissible partitions for
some $d\ge 0$. Then
$$
\mult_{(\gamma|\lambda)}(J_t)=\begin{cases}
\mult_{(\gamma|\lambda)}(S_t)&\text{if }\gamma\neq\lambda,\\
\mult_{(\gamma|\lambda)}(S_t)-1&\text{if }\gamma=\lambda.
\end{cases}
$$
\end{prop}

\begin{remark}\label{idealA_t}
It is worth noting that Pieri's formula completely governs the
structure of the $G$-stable ideals in $A_t$.

(a) Let us first discuss the case $t=1$. Let $R=R(m,n)$ and
consider the ideal $I_\sigma$ generated by
$R_{(\sigma|\sigma)}$. By a theorem of \cite{DEP1} (also see
\cite[11.15]{BV}) one has
\begin{equation}
I_\sigma=\Dirsum_\tau R_{(\tau|\tau)}\label{Itau}
\end{equation}
where the sum is extended over all diagrams $\tau\supseteq
\sigma$.

(b) Now let $\lambda$ be $(t,d)$-admissible, and let
$B_\sigma$ be the ideal in $A_t$ generated by
$(A_t)_\sigma =(A_t)_{(\sigma|\sigma)}$. Then
\begin{equation}
B_\sigma=A_t\cap I_\sigma=\Dirsum (A_t)_\tau \label{Whitehead}
\end{equation}
where the sum is taken over all partitions $\tau$ that arise as
iterated $t$-successors of $\sigma$.

The inclusion $\subseteq$ is a direct consequence of Pieri's
formula whereas the opposite inclusion follows from a theorem
of Whitehead \cite[Theorem 7.2]{Whh} who determined the
(necessarily multiplicity free) decomposition of
$R_{(\sigma|\sigma)}\cdot R_{(\tau|\tau)}$ for arbitrary
$\sigma$ and $\tau$, showing that the irreducibles appearing in
it are exactly those that come up in the Littlewood-Richardson
formula for $L_\sigma V\tensor L_\tau V$. For $\tau=(d)$ the
Littlewood-Richardson formula specializes to Pieri's formula.
Then \eqref{Whitehead} follows by induction.
\end{remark}

\subsection{A formula for successors of a Schur
module}\label{sec_succ}

In order to exclude a bi-diagram $(\gamma|\lambda)$ from being
minimal in $J_t$ we must find a bi-diagram $(\gamma'|\lambda')$
such that $(\gamma|\lambda)$ occurs in $(S_t)_1\cdot
(J_t)_{(\gamma'|\lambda')}$. In this subsection we will derive
a formula which allows us to explicitly build a highest weight
vector of shape $\gamma$ from a highest weight vector of shape
$\gamma'$. The formula will be crucial for concrete
computations in Section \ref{Upper}.

More precisely, let $\lambda =
(\lambda_1,\dots,\lambda_k)\vdash N$ and
$\gamma=(\gamma_1,\dots,\gamma_h)\vdash N+t$ be two diagrams.
Furthermore, let $\Lambda$ and $\Gamma$ be tableaux of shapes
$\lambda$ and $\gamma$ on $\{1,\dots,N\}$ and
$\{1,\dots,N+t\}$, of contents $(1,\dots,1)\in \NN^N$ and
$(1,\dots,1)\in \NN^{N+t}$, respectively. We know that an
isomorphic copy of $\YY_\Gamma(\bigotimes^{N+t}V)$ is a direct
summand of $\YY_\Lambda(\bigotimes^NV)\otimes (\bigotimes^tV)$
if and only if $\lambda \subseteq \gamma$. However, in general
$\YY_\Gamma(\bigotimes^{N+t}V)$ is not contained in
$\YY_\Lambda(\bigotimes^NV)\otimes (\bigotimes^tV)$, regardless
of the choice of $\Gamma$. Below, we will discuss how to
produce an element in $\bigotimes^{N+t}V$ which is the highest
weight vector of one of the isomorphic copies of $L_{\gamma}V$
contained in $\YY_\Lambda(\bigotimes^NV)\otimes
(\bigotimes^tV)$ under the condition that $\gamma$ is built
from $\lambda$ by adding $t$ boxes in different columns, or, by
Pieri's formula, shows up in $L_\lambda\tensor \Bw^t V$, and
this is the case in which we are interested.

More precisely, let $\gamma$ be obtained by adding the $t$
boxes
\[
(i_1,j_{1,1}),\dots,(i_1,j_{1,s_1}),(i_2,j_{2,1}),\dots,(i_2,j_{2,s_2}),\dots,(i_p,j_{p,1}), \dots,(i_p,j_{p,s_p})
\]
to $\lambda$ such that \smallskip

\begin{itemize}
\item[(i)] $1\leq i_1<i_2<\dots < i_p\leq h$;
\item[(ii)] if $q>r$, then $j_{q,a}<j_{r,b}$ for all $a$
    and $b$. Moreover $j_{r,a+1}=j_{r,a}+1$ whenever $1\leq a<s_r$;

\end{itemize}
\smallskip

Let us define a tableau $T_{\lambda,\pi,\gamma}$ of shape
$\lambda$ on $\{1,\dots,n\}$ for a permutation $\pi\in
\Sigma_{\gamma_{i_1}}$ as follows:
\[
T_{\lambda,\pi,\gamma}(i,j) =
\begin{cases}
\pi(j) & \mbox{if } i=i_{\ell} \mbox{ and } j>j_{\ell -1,s_{\ell -1}}=
\gamma_{i_{\ell -1}}, \\
j & \mbox{otherwise},
\end{cases}
\]
for all $i=1,\dots,k$ and $j=1,\dots,\lambda_i$ (with the
convention that $j_{0,s_0}=0$).

\begin{example}\label{exTableu}
Suppose we want to pass from $\lambda=(7,4,1)\vdash 12$ to
$\gamma=(8,6,2)\vdash 16$. Given $\pi\in \Sigma_8$, the above
tablaeu is:

\[
T_{\lambda ,\pi ,\gamma}=
\begin{tabular}{| c | c | c | c | c | c | c |}
\hline
1 & \ \ 2 \ \ & \ \ 3 \ \ & 4 & \ \ 5 \ \ & \ \ 6 \ \ &$\pi(7)$ \\
\hline
1 & 2 &$\pi(3)$&$\pi(4)$\\
\cline{1-4}
$\pi(1)$\\
\cline{1-1}
\end{tabular}
\]

\vspace{1mm}

\end{example}

\begin{lemma}\label{lemmaformula}
The following element of $\bigotimes^{N+t}V$ is the highest
weight vector of one of the copies of $L_{\gamma}V$ that appear
in the decomposition of $\YY_\Lambda(\bigotimes^N V)\otimes
(\bigotimes^tV)\subseteq \bigotimes^{N+t}V$:
\begin{multline}\label{fromltog} g_{\lambda \leadsto
\gamma}=\\ \sum_{\pi\in
\Sigma_{\gamma_{i_1}}}(-1)^{\pi}\YY_\Lambda(\TTT)\otimes
\bigl(e_{\pi(j_{p,1})}\otimes \cdot\cdot \otimes
e_{\pi(j_{p,s_p})}\otimes \cdots\cdots\otimes
e_{\pi(j_{1,1})}\otimes \cdot\cdot \otimes
e_{\pi(j_{1,s_1})}\bigr)
\end{multline}
\end{lemma}
\begin{proof}
The element $g_{\lambda \leadsto \gamma}\in \bigotimes^{N+t}V$
belongs to $\YY_\Lambda(\bigotimes^N V)\otimes (\bigotimes^tV)$
by construction. Furthermore its weight is $\tl \gamma$.
Therefore, we need just to show that $g_{\lambda \leadsto
\gamma}$ is $\UU_-(V)$-invariant. Notice that a system of
generators of the group $\UU_-(V)$ is provided by the
elementary transformations $E_{ij}^x$ with $n\geq i>j\geq 1$
and $x\in \kk$, acting on $V$ via
\[
E_{ij}^x(e_k)= e_k+\delta_{ik}xe_j,\qquad k=1,\dots,n
\]
($\delta_{ik}$ is Kronecker's delta). Therefore, we need to
show that $E_{ij}^xg_{\lambda \leadsto \gamma}=g_{\lambda
\leadsto \gamma}$ for all $n\geq i>j\geq 1$ and $x\in \kk$.
Because $\YY_\Lambda$ is alternating on the rows, we have
\[
E_{ij}^xg_{\lambda \leadsto \gamma}=g_{\lambda \leadsto \gamma}+
x\sum_{\pi\in \Sigma_{\gamma_{i_1}}}(-1)^{\pi}g_{\lambda,\pi,\gamma}(i\mapsto j)
\]
where $g_{\lambda,\pi,\gamma}(i\mapsto j)$ means
$\YY_\Lambda(\TTT)\otimes (e_{j_{p,1}}\otimes \cdot\cdot
\otimes e_{j_{p,s_p}}\otimes \cdots\cdot \otimes
e_{j_{1,1}}\otimes \cdot\cdot \otimes e_{j_{1,s_1}})$ with the
unique permuted $e_i$ replaced by $e_j$. Now, for all $\pi\in
\Sigma_{\gamma_{i_1}}$, set $\pi'=(i \ j)\cdot \pi$. Clearly
we have $g_{\lambda,\pi,\gamma}(i\mapsto
j)=g_{\lambda,\pi',\gamma}(i\mapsto j)$. Moreover
$(-1)^{\pi'}=-(-1)^{\pi}$. This implies that
\[
\sum_{\pi\in \Sigma_{\gamma_{i_1}}}(-1)^{\pi}g_{\lambda,\pi,\gamma}(i\mapsto j)=0,
\]
so $E_{ij}^xg_{\lambda \leadsto \gamma}=g_{\lambda \leadsto
\gamma}$.

It just remains to be shown that $g_{\lambda\leadsto\gamma}\neq
0$. For this we rewrite $g_{\lambda\leadsto\gamma}$ as
\[
\sum_{T,\underline{i}}a_{T,\underline{i}}\YY_\Lambda(T)\otimes
e_{i_1}\otimes \cdots \otimes e_{i_t},
\]
where $T$ varies among the standard tableaux of shape $\lambda$
in $\{1,\dots,n\}$, $\underline{i}$ varies in $\{1,\dots,n\}^t$
and the $a_{T,\underline{i}}\in \kk$ are the coefficients.
Since the above representation is a linear combination of
elements of a basis of $\YY_\Lambda(\bigotimes^N V)\otimes
(\bigotimes^tV)$, it is enough to show that at least one of the
$a_{T,\underline{i}}$ is not $0$. This follows immediately from
the fact that $\YY_\Lambda$ is alternating on the rows: let
$\underline{i_0}=(j_{p,1},\dots,j_{p,s_p},\dots,j_{1,1},\dots,j_{1,s_1})$.
The only possibly nonzero coefficient $a_{T,\underline{i_0}}$
corresponds to the tableau $T_0$ of shape $\lambda$ such that
$T(i,j)=j$ for all $(i,j)\in \lambda$. We have that
\[
a_{T_0,\underline{i_0}}=\sum_{\pi\in A}(-1)^{\pi}(-1)^{\pi}=|A|,
\]
where $A\subseteq \Sigma_{\gamma_{i_1}}$ consists in the
permutations $\pi$ such that $\pi(j_{h,k})=j_{h,k}$ for all
$h=1,\dots,p$ and $k=1,\dots,s_h$ and $\pi$ preserves the rows
of $T_0$.
\end{proof}

\begin{example}
In the situation of Example \ref{exTableu}, we have
\[
g_{\lambda\leadsto \gamma}=\sum_{\pi\in \Sigma_8}\mathbb{Y}_{\Lambda}(T_{\lambda,\pi ,\gamma})
\otimes (e_{\pi(2)}\otimes e_{\pi(5)}\otimes e_{\pi(6)}\otimes e_{\pi(8)}).
\]
\end{example}

\begin{remark}\label{remformula}
In view of the application of Lemma \ref{lemmaformula} that we
have in mind let us consider the natural $\GL(V)$-equivariant
surjective map:
\[
f_d:\bigotimes^{dt}V\to \bigotimes^{d}\Bl(\Bw^tV\Br).
\]
If $N=dt$ and the starting shape $\lambda =
(\lambda_1,\dots,\lambda_k)\vdash dt$ in Lemma
\ref{lemmaformula} is $(t,d)$-admissible, then there exists a
tableau $\Lambda$ of shape $\lambda$ on $\{1,\dots,dt\}$ such
that $c(\Lambda)=(1,\dots,1)\in\NN^{dt}$ and
\[
f_d\left(\YY_\Lambda\Bl(\bigotimes^{dt}V\Br)\right)\cong L_{\lambda}V.
\]
In this situation, one can show that $f_d(g_{\lambda \leadsto
\gamma})\neq 0$ by the same method used in the proof of Lemma
\ref{lemmaformula}. In particular, $f_d(g_{\lambda \leadsto
\gamma})$ is the highest weight vector of the unique copy of
$L_{\gamma}V$ which is a direct summand of
$f_d(\YY_\Lambda(\bigotimes^{dt}V))\otimes (\Bw^tV)$.
\end{remark}

\subsection{The coarse decomposition}\label{Estructure}
Set
\[
E=\Bw^tV\qquad\text{and}\qquad  F=\Bw^tW.
\]
Instead of the group $G=\GL(V)\times\GL(W)$ one can also
consider the action of the larger group $H=\GL(E)\times \GL(F)$
on $T_t$ and $S_t$. The main advantage is that the
$H$-structure of $S_t$ is well-understood by the Cauchy
formula:
\begin{equation}
S_t=\Dirsum_{\mu} (S_t)_\mu,\qquad (S_t)_\mu=L_\mu E \tensor L_\mu F^*,
\label{CauchyH}
\end{equation}
with the restrictions imposed on $\mu$ by the dimensions of the
involved vector spaces. However, $H$ does not act on $A_t$, and
the ideal $J_t$ is not an $H$-submodule of $S_t$ (apart from
trivial exceptions). Therefore, in order to make full use of
\eqref{CauchyH} one would have to understand the
$\GL(V)$-decomposition of $L_\mu E$. For example, a
bi-shape $(\gamma|\lambda)$ of partitions $\gamma ,\lambda\vdash
dt$ has multiplicity $\ge 1$ in $S_t$ if and only if there
exists a partition $\mu\vdash d$ such that $L_\gamma V$
occurs in the decomposition of $L_\mu E$, and the same holds
for $L_\lambda W^*$ in $L_\mu F^*$.

In general, the $\GL(V)$-decomposition of $L_\lambda E$ is an
unsolved plethysm. The difficulty of the problem is illustrated
by the fact that copies of  $\Wc$ may appear in $L_\mu E$ for
several $\mu$, and that there is no equivalence relation on
partitions $\gamma ,\lambda\vdash dt$ by which one could decide
whether $(\gamma |\lambda)$ has multiplicity $\ge 1$ in $S_t$.
In order to illustrate the problem and for the discussion of
concrete examples we include plethysms for $t=2$. The tables
have been computed by \textsf{Lie} \cite{lie}. (Despite of the below tables, even for
$t=2$ the $\GL(V)$-modules are not multiplicity free in
general.)
\begin{table}[hbt]
\begin{center}
\begin{tabular}{|l|l|l|}
\hline
$\mu=(1,1,1)$&$\mu=(2,1)$&$\mu=(3)$\\
\hline
(6)&(5,1)&(4,1,1)\\
(4,2)&(4,2)&(3,3)\\
(2,2,2)&(3,2,1)&\\
\hline
\end{tabular}
\end{center}
\caption{Plethysms for $L_\mu (\Bw^2 V)$, $\mu\vdash 3$}
\label{Plt23}
\end{table}
\begin{table}[hbt]
\begin{center}
\begin{tabular}{|l|l|l|l|l|}
 \hline
$\mu=(1,1,1,1)$&$\mu=(3,1)$&$\mu=(2,2)$&$\mu=(2,1,1)$&$\mu=(4)$\\
\hline
(8)&         (7,1)&        (6,2)&       (6,1,1)&      (5,1,1,1)\\
(6,2)&       (6,2)&        (5,2,1)&     (5,3)&        (4,3,1)\\
(4,4)&       (5,3)&        (4,4)&       (5,2,1)&      \\
(4,2,2)&     (5,2,1)&      (4,2,2)&     (4,3,1)&      \\
(2,2,2,2)&   (4,3,1)&      (3,3,1,1)&   (4,2,1,1)&    \\
&            (4,2,2)&      &            (3,3,2)&      \\
&            (3,2,2,1)&    &            &             \\
\hline
\end{tabular}
\end{center}
\caption{Plethysms for $L_\mu (\Bw^2 V)$, $\mu\vdash 4$}
\label{Plt24}
\end{table}

\begin{remark}\label{non_dir_sum}
Despite of the fact that $J_t$ is not an $H$-ideal in $S_t$ one
could hope for the next best structure with respect to the
$H$-action, namely that $J_t$ is the direct sum of its
intersections with the $H$-irreducibles $(S_t)_\mu$. Clearly,
if a bi-diagram $(\gamma|\lambda)$ occurs with multiplicity $1$
in $S_t$, then the corresponding $G$-irreducible must be
contained in (exactly) one of the $(S_t)_\mu$. However, as soon
as $\mult_{(\gamma|\lambda)}(S_t)\ge 2$, the inclusion
$(J_t)_{(\gamma|\lambda)}\subseteq\Dirsum J_t\cap
(S_t)_\mu$
may fail.
In fact, it fails already in the smallest possible case, namely
$(4,2|4,2)$, which has multiplicity $2$ in $S_2$ (see Table
\ref{Plt23}) and multiplicity $1$ in $J_2$. We will discuss the
computation in Subsection \ref{nogeg4t=2}.
\end{remark}

One of the few classical known plethysms is
\begin{equation}
\Sym^d\Bl(\Bw^2 V\Br)=\Dirsum_{\lambda\text{ even}\atop \lambda_1\le m}
L_\lambda V\label{plethsym}
\end{equation}
where $\lambda$ is \emph{even} if all its parts $\lambda_i$ are
even; see \cite[p.~63]{We}. The plethysm \eqref{plethsym} has a
companion for exterior powers that we will encounter later on.

The plethysm \eqref{plethsym} can be used in a ring-theoretic
way in connection with the following proposition. (The Segre
product of graded algebras $A=\Dirsum_i A_i$ and $B=\Dirsum_i
B_i$ is the algebra $\Dirsum_i A_i\tensor B_i$.)

\begin{prop}\label{Segre}
There are natural $G$-equivariant projections
\begin{align*}
\alpha:S_t(m,n)&\to \Sym(E)\ \sharp\  \Sym(F^*),\\
\beta: S_t(m,n)&\to \Bw (E)\ \sharp\  \Bw(F^*),
\end{align*}
where $\sharp$ denotes the Segre product.
\end{prop}

\begin{proof}
By the universal property of the symmetric algebra, the natural
homomorphisms
\begin{align*}
\Tensor(E\tensor F^*)&=\Tensor E\ \sharp\ \Tensor F^*\to
\Sym E\ \sharp\  \Sym F^*,\\
\Tensor(E\tensor F^*)&=\Tensor E\ \sharp\ \Tensor F^*\to
\Bw E\ \sharp\  \Bw F^*
\end{align*}
are   $G$-equivariant and factor through $S_t$. (Note that the Segre product of the
exterior algebras is commutative.)
\end{proof}

Now we formulate a very useful rule that simplifies many
discussions. It is the represen\-ta\-tion-theoretic analogue of
Proposition \ref{full_retract}.

\begin{prop}\label{propretract}
Let $\mu$ be a partition of $d$ and consider partitions
$\lambda =(\lambda_1,\dots,\lambda_k)\vdash td$ with $k\leq d$
and
$\tilde{\lambda}=(\lambda_1+1,\dots,\lambda_k+1,1,\dots,1)\vdash
dt+d$. If $\dim_{\kk}V\geq \lambda_1+1$, then

$$\mult_\lambda (L_\mu(\Bw^tV)) =  \mult_{\tilde{\lambda}}  (L_\mu(\Bw^{t+1}V))$$
\end{prop}

\begin{proof}
Let us consider the map
\[
\xi:\Tensor^d\Bl(\Bw^tV\Br)\to \Tensor^d\Bl(\Bw^{t+1}V\Br)
\]
that extends the assignment
\begin{multline*}
(e_{a_{1,1}}\wedge \cdots \wedge e_{a_{1,t}})\otimes \cdots
\otimes (e_{a_{d,1}}\wedge \cdots \wedge e_{a_{d,t}})\\
\mapsto
 (e_1\wedge e_{a_{1,1}+1}\wedge \cdots \wedge
e_{a_{1,t}+1})\otimes \cdots \otimes (e_1\wedge
e_{a_{d,1}+1}\wedge \cdots \wedge e_{a_{d,t}+1})
\end{multline*}
$\kk$-linearly; here $a_{i,j}\in \{1,\dots,\dim_{\kk}V\}$, and we
use the convention that $e_q=0$ if $q>\dim_{\kk}V$.

Since $\dim_{\kk}V\geq \lambda_1+1$ the vector space $Q$ of the
$\UU_-(V)$-invariants of weight $\tl\lambda$ in $\Tensor^d\Bw^t V$
is contained in the subspace $\Tensor^d\Bw^t V'$ where $V'$ is
generated by $e_1,\dots,e_{n-1}$, $n=\dim_{\kk} V$. On this subspace
$\xi$ is injective. On the other hand, the subspace of the
$\UU_-(V)$-invariants of weight $\tl\tilde\lambda$ in
$\Tensor^d\Bw^{t+1} V$ is contained in $\xi(Q)$ since
each tensor factor of each summand in the representation of such a
$\UU_-(V)$-invariant in the natural basis starts with $e_1$.
\end{proof}

\begin{definition}
If a partition $\tilde \lambda$ arises from $\lambda$ by
prefixing $\lambda$ with columns of length $d$, then
$\tilde\lambda$ is called a \emph{trivial extension} of
$\lambda$.
\end{definition}

Iterated application of Proposition \ref{propretract} shows
that it holds for trivial extensions in general.

For the analysis of degree $3$ relations the following
proposition will turn out useful.

\begin{prop}\label{propinthemiddle}
Let $\lambda$ be a $(t,3)$-admissible diagram with more than
one predecessor. If $\dim_{\kk}V\geq \lambda_1$, then
$L_\lambda V$ is a direct summand of $L_{(2,1)}E$.
\end{prop}
\begin{proof}
By Proposition \ref{propretract} we can assume
$\lambda=(\lambda_1,\lambda_2)$. Then $\lambda$ has more than
one predecessor if and only if $\lambda_1>\lambda_2>0$.

If $\lambda_2\leq t$, then $(2t)$ is a predecessor of
$\lambda$. Using Lemma \ref{lemmaformula} and Remark
\ref{remformula}, we know that the element
$g=g_{(2t)\leadsto\lambda}$ is the $\UU_-(V)$-invariant of the
unique copy of $L_\lambda V$ contained in $L_{(2t)}V\otimes E$.
Now, $L_{(2t)}V\otimes E$ is contained in $\Sym^3E\oplus
L_{(2,1)}E$ or in $\Bw^3E\oplus L_{(2,1)}E$, depending on
the parity of $t$. In any case, the element $g$ is neither
symmetric nor alternating. To see this, we need to
consider the $\ell$ monomials in the support of $g$:
\[
(e_{a_1^i}\wedge \cdots \wedge e_{a_t^i})\otimes (e_{b_1^i}\wedge \cdots
\wedge e_{b_t^i})\otimes (e_{c_1^i}\wedge \cdots
\wedge e_{c_t^i}), \ \ \ i=1,\dots,\ell .
\]
Then, for all $i\in\{1,\dots,\ell\}$, we
have $1\in \{c_1^i,\dots,c_t^i\}$, whereas $1$ does not belong
to the intersection $\{a_1^i,\dots,a_t^i\}\cap
\{b_1^i,\dots,b_t^i\}$. So $g=f+h$ with $h\in L_{(2,1)}E$
different from $0$ and $f\in\Sym^3E$ or $f\in\Bw^3E$,
depending on the parity of $p$. In any case, $h$ is a
$\UU_-(V)$-invariant of weight $\tl\lambda$, thus the
$\GL(V)$-space generated by it, which obviously is contained in
$L_{(2,1)}E$, is isomorphic to $L_\lambda V$.

If $\lambda_2>t$, then we consider the predecessor
$(\lambda_1,\lambda_2-t)$ of $\lambda$. The proof of
this case is analog to the previous one, so we do not repeat
it. Let us just say that this time we show that
$g_{(\lambda_1,\lambda_2-t)\leadsto\lambda}$ is neither symmetric nor
alternating by using that, for all $i$, $\lambda_1\notin
\{c_1^i,\dots,c_t^i\}$ and $\lambda_1\in
\{a_1^i,\dots,a_t^i\}\cup \{b_1^i,\dots,b_t^i\}$.
\end{proof}

We introduce a class of partitions that seem to be crucial for
the analysis of $J_t$.

\begin{definition}
We say that a partition $\lambda\vdash dt$ is of \emph{single
$\Bw^t$-type $\mu$} if $\mu\vdash d$ is the only partition such
that the $\GL(V)$-irreducible $L_\lambda V$ occurs in the
$\GL(E)$-irreducible $L_\mu E$ and, moreover, has multiplicity
$1$ in it.

A bi-diagram $(\gamma|\lambda)$ is of \emph{single $\Bw^t$-type}
if both $\gamma$ and $\lambda$ are of single $\Bw^t$-type.
\end{definition}

Clearly, bi-diagrams of single $\Bw^t$-type have multiplicity $1$
in $S_t$ (if they occur at all), but the converse does not
hold, as shown by $(4,3,1|6,2)$ for $t=2$, $d=4$.

\begin{remark}
For every partition $\mu\vdash d$ there exists at least one
partition $\lambda\vdash dt$ of single $\Bw^t$-type $\mu$: just
take $\lambda$ to be the trivial extension of $\mu$ by
prefixing it with $t-1$ columns of length $d$. One can use
$\lambda$ as an indicator for $\mu$: a partition $\gamma$
appears in $\mu$ if and only $(\gamma|\lambda)$ occurs in $S_t$
(with the same multiplicity). Therefore the
$\GL(V)$-decomposition of $L_\mu E$ can be reconstructed for
all $\mu$ from the decomposition of $S_t$.

In general there exist more than one partition of single
$\Bw^t$-type $\mu$. The reader may check that the following
$(t,d)$-admissible diagrams $\lambda$ are of single
$\Bw^t$-type: (i) $\lambda_1\le t+1$, (ii) $\lambda$ is a hook,
i.e.~ $\lambda_2\le 1$. By trivial extension one can construct
further singe $\Bw^t$-type diagrams from (ii). Two other types
will be encountered in Theorem \ref{shapeink3odd} and Corollary
\ref{mult1}.
\end{remark}

\begin{prop}
$\lambda\vdash dt$ is of single $\Bw^t$-type if and only if the
bi-shape $(\lambda|\lambda)$ has multiplicity $1$ in $S_t$ or,
equivalently, does not occur in $J_t$.
\end{prop}

This follows immediately from \eqref{cauchy}. Single
$\Bw^t$-type can be characterized recursively:

\begin{prop}\label{single_char}
Let $\lambda\vdash dt$ and $\mu\vdash d$ be partitions such
that $\lambda$ occurs in $L_\mu E$. Then the following are
equivalent:
\begin{itemize}
\item[(i)]$\lambda$ is of single $\Bw^t$-type;
\item[(ii)] the multiplicities of $\lambda$ and of $\mu$ in
    $\Tensor^d(\Bw^t V)$
    coincide;
\item[(iii)] every $t$-predecessor $\lambda'$ of $\lambda$ is
    of single
    $\Bw^t$-type $\mu'$ where $\mu'$ is a $1$-predecessor of $\mu$,
    and no two distinct $t$-predecessors of $\lambda$ share the same
    $1$-predecessor $\mu'$ of $\mu$.
\end{itemize}
\end{prop}

The proof uses only the recursive formula for multiplicities in
Proposition \ref{tensordec}.

In the next theorem we exploit Pieri's formula \eqref{Pieri}
for $G$ and $H$  and the Cauchy formula \eqref{CauchyH}
simultaneously.

\begin{thm}\label{GandH}
\begin{itemize}
\item[(i)] Let $\mu\vdash d$ be a partition, and let $M$ be the
    set of $1$-successors of $\mu$. Then the linear map
$$
(S_t)_1\tensor (S_t)_\mu\to
\Dirsum_{\nu\in M} (S_t)_\nu
$$
induced by multiplication in $S_t$ is surjective.
\item[(ii)] Let $\gamma$ and $\lambda$ be $(t,d)$-admissible
    partitions. If $(\gamma|\lambda)$ occurs in $(S_t)_\mu$, but there
    exists a $1$-predecessor $\mu'$ of $\mu$ such that all
    bi-predecessors of $(\gamma|\lambda)$ that occur in
    $(S_t)_{\mu'}$ are
    asymmetric, then $(\gamma|\lambda)$ is not minimal
    in $J_t$.
\item[(iii)] With the same notation, suppose that
    $\gamma\neq\lambda$ and that all bi-predecessors of
    $(\gamma|\lambda)$ that occur in $(S_t)_{\mu'}$ for any
    $1$-predecessor $\mu'$ of $\mu$ are symmetric of single
    $\Bw^t$-type. Then $(\gamma|\lambda)$ is minimal in
    $J_t$.
\item[(iv)] Let $(\gamma|\lambda)$ be asymmetric of single
    $\Bw^t$-type $\mu$.
    Then either (a) $(\gamma|\lambda)$ is not minimal in
    $(J_t)_\mu$ or (b) $\gamma$ and $\lambda$ have the same
    predecessors (of single $\Bw^t$-type).
\end{itemize}
\end{thm}

\begin{proof} (i) It has already been mentioned in Remark
\ref{idealA_t}(a) that the ideal in $S_t$ generated by
$(S_t)_\mu$ is the sum of all $(S_t)_\nu$ where $\nu$ arises
from $\mu$ by the addition of boxes. This implies claim (i)
(and is equivalent to it by induction).

(ii) By hypothesis all bi-predecessors of $(\gamma|\lambda)$  in
$(S_t)_{\mu'}$ lie in $J_t$ since they are asymmetric. So (i)
implies that $(\gamma|\lambda)$ lies in $(S_t)_1\cdot J_t$.

(iii) Let $U$ be the $G$-submodule of $(J_t)_{d-1}$ generated by
all irreducibles whose shapes are bi-predecessors of $(\gamma|\lambda)$.
We must show that $(\gamma|\lambda)$ does not occur in
$(S_t)_1 U\cap (J_t)_\mu$.

We split $U$ into the sum of three
$G$-submodules, namely the sum $U_1$ of all $G$-irreducibles
whose shape occurs only in $(S_t)_{\mu'}$ for some $1$-predecessor
$\mu'$ of $\mu$, the sum $U_2$ of all submodules whose shape
occurs only in $(S_t)_{\mu'}$ for some non-$1$-predecessor $\mu'$ of
$\mu$ and a complementary summand $U_3$ of $U_1\dirsum U_2$
(which exists by linear reductivity of $G$). In general $U_3$
may be non-zero (see Remark \ref{non_dir_sum}), however all
bi-shapes $(\gamma'|\lambda')$ in $U_3$ must appear in a
$1$-predecessor of $\mu$ as well as in a non-$1$-predecessor.
This is impossible for single $\Bw^t$-type, and so $U_3=0$.
Since $((S_t)_1\cdot U_2)\cap (S_t)_\mu=0$ and $U_1=0$
by hypothesis, $(\mu|\lambda)$ must indeed be minimal in $J_t$.

(iv) It follows from (i) that $(\gamma|\lambda)$ has a
bi-predecessor in $(S_t)_{\mu'}$ for every predecessor $\mu'$
of $\mu$, and because of single $\Bw^t$-type there exists
exactly one such predecessor in every $(S_t)_{\mu'}$.

Suppose first that $(\gamma$ and $\lambda)$ have different
predecessors. Then there must exist a $\mu'$ in which the
single predecessor is asymmetric, and so $(\gamma|\lambda)$ is
not minimal in $J_t$. Otherwise all predecessors are symmetric
of single $\Bw^t$-type and we can apply (iii) in order to
conclude that $(\gamma|\lambda)$ is minimal in $J_t$.
\end{proof}

In particular, $(\gamma|\lambda)$ is minimal in $J_t$ if all
its bi-predecessors (with $\mult_{(\gamma'|\lambda')}(S_t)>0$)
are symmetric of multiplicity $1$. Conversely, if all
bi-predecessors are asymmetric, then $(\gamma|\lambda)$ is not
minimal. However, Theorem \ref{GandH}(ii) is more precise as
the following example shows: for $t=2$ the bi-diagram $(5, 3|7,
1)$ belongs with multiplicity $1$ only to $(S_2)_\mu$ for
$\mu=(2,1,1)$. However, in $(S_2)_\nu$, $\nu=(1,1,1)$ it has no
symmetric bi-predecessor, and therefore it is not minimal in
$J_2$. (But it has the symmetric bi-predecessor $(5,1|5,1)$ of
multiplicity $1$ in $(S_2)_{(2,1)}$.)

On the other hand, Theorem \ref{GandH} does not allow us to
exclude that $(6,2|7,1)$ is minimal in $J_2$, although all
relevant plethysms are known. That it is not minimal will be
documented in Subsection \ref{nogeg4t=2}.

\begin{definition}
The minimal relations $(\gamma|\lambda)$ identified in Theorem
\ref{GandH}(iii) are called \emph{shape relations}:
$(\gamma|\lambda)$ is asymmetric and occurs in $(S_t)_\mu$, but
all bi-predecessors of $(\gamma|\lambda)$ that occur in
$(S_t)_{\mu'}$ for any $1$-predecessor $\mu'$ of $\mu$ are
symmetric of single    $\Bw^t$-type.
\end{definition}

We do not know whether all minimal relations are shape
relations. Raising this question is a main point of the paper.
It is useful to introduce shape relations also in the tensor
algebra:

\begin{definition}
Let $\gamma,\lambda\vdash dt$ be $(t,d)$-admissible,
$\gamma\neq\lambda$. If all bi-predecessors of
$(\gamma|\lambda)$ are symmetric of multiplicity $1$ in $T_t$,
then $(\gamma|\lambda)$ is called a \emph{$T$-shape relation}.
\end{definition}

\begin{prop}\label{Tshape}
$T$-shape relations are minimal in $K_t$, and a $T$-shape
relation that appears in $S_t$ is a shape relation. In
particular, all $T$-shape relations of degree $\ge 3$ are
shape relations.
\end{prop}

\begin{proof}
The first statement follows by the same (and even simpler)
arguments as for shape relations. The second is obvious, and
for the third we apply Lemma \ref{redtotensor}.
\end{proof}

We will classify the $T$-shape relations in Subsection
\ref{T-shape_rel}. However, not all shape relations are
$T$-shape relations, as will become apparent in Subection
\ref{secoddrelations}.

\section{Quadratic and cubic relations}\label{QuadCub}

In order to write down explicit polynomials representing the relations (and not just shapes or
tableaux) we must introduce some notation. Let $A\subseteq \NN$ be a
set of cardinality $N<\infty$. Let us write $A=\{a_1,\dots,a_N\}$ in
ascending order. Let $A_1,\dots,A_k$ be a $k$-partition of $A$: that
is, $A_1\cup \dots \cup A_k=A$ and $A_i\cap A_j=\emptyset$ for all
$i\neq j$. Set $r_i=|A_i|$ and let us write
$A_i=\{a_{i,1},\dots,a_{i,r_i}\}$ in ascending order. With the
symbol
\[
(-1)^{A_1,\dots,A_k}
\]
we mean the sign of the unique permutation of $A$ taking the
sequence $a_1,\dots,a_N$ to the sequence
$a_{1,1},\dots,a_{1,r_1},a_{2,1},\dots,a_{2,r_2},\dots,a_{k,1},\dots,a_{k,r_k}$.
If some $A_i$ consists of one element, so that
$A_i=\{a_{i,1}\}$, we may simply write this sign as
$(-1)^{A_1,\dots,A_{i-1},a_{i,1},A_{i+1},\dots,A_k}$. Given
another finite set $B$, we will say that $A$ is
lexicographically smaller than $B$ if $|A|<|B|$ or $|A|=|B|$
and the vector $(a_1,\dots,a_N)$ is lexicographically smaller
than $(b_1,\dots,b_N)$ with $b_i\in B$ taken in ascending
order. With $e_A$ we mean $e_{a_1}\wedge e_{a_2}\wedge \dots
\wedge e_{a_N}$. Similarly for $e_A^*$, $f_A$ and $f_A^*$.
Eventually, if $B_i=\{b_{i,1},\dots,b_{i,s_i}\}\subseteq \NN$,
with the $b_{i,j}$'s taken in ascending order, are disjoint
subsets for $i=1,\dots,h$ such that $s_1+\dots +s_h=N$, we
define the $N$-minor
\[[
A_1,\dots,A_k|B_1,\dots,B_h]=[a_{1,1},\dots,a_{1,r_1}, \dots,a_{k,1},
\dots,a_{k,r_k}|b_{1,1},\dots,b_{1,s_1}, \dots,b_{h,1},\dots,b_{h,s_h}]
\]

In order to keep the notation transparent, we set
\[
E=\Bw^tV\qquad\text{and}\qquad  F=\Bw^tW
\]
as in Subsection \ref{Estructure}.

\subsection{Quadratic relations}\label{subdeg2}

The only degree $2$ (minimal) relations between $2$-minors of
an $m\times n$-matrix are Pl\"ucker relations, as we will see.
However this is not true anymore for $t$-minors with $t\geq 3$.
In this subsection we want to describe all the degree $2$
relations between $t$-minors. In order to do this we need a
decomposition of
\[
\Sym^2(E\tensor F^*)
\]
into irreducible $G$-modules. Since
\[\
\Tensor^2E=\Sym^2E\oplus \Bw^2E,
\]
one can show (or \eqref{CauchyH} implies) that:
\[
\Sym^2(E\otimes F^*)=
\Bl(\Sym^2E\otimes \Sym^2F^*\Br) \bigoplus
\Bl(\Bw^2E\otimes \Bw^2F^*\Br).
\]
By Pieri's formula, we know that $\displaystyle
\Tensor^2E\cong \bigoplus_{u=0}^t L_{\tau_u}V$, where
\begin{equation}\label{ker2}
\tau_u=(t+u,t-u).
\end{equation}
So the matter is just to decide whether $L_{\tau_u}V$ is in
$\Sym^2E$ or in $\Bw^2E$:

\begin{lemma}\label{decs2at}
If $\dim_{\kk}V\geq 2t$, for $u\in \{0,\dots t\}$, we have:
\[L_{\tau_u}V\subseteq \Sym^2E \iff u \mbox{ is even}.\]
\end{lemma}
\begin{proof}
It is straightforward to check that the element
\[
\sum_{{{I\cup J = \{t-u+1,\dots,t+u\}}\atop{|I|=|J|=u}}}(-1)^{I,J}(e_1\wedge e_2 \wedge \dots
\wedge e_{t-u}\wedge e_I)\tensor(e_1\wedge e_2 \wedge \dots \wedge e_{t-u}\wedge e_J)
\]
is a nonzero $\UU_-(V)$-invariant. Therefore it is a highest weight
vector of weight $\tl\tau_u = (2^{t-u},1^{2u})$. Thus it
generates the irreducible $\GL(V)$-module $L_{\tau_u}V$.
Furthermore, it is clear that $(-1)^{I,J}=(-1)^u(-1)^{J,I}$, so
the claim follows.
\end{proof}

The same discussion holds for $W^*$, so Lemma \ref{decs2at} yields the desired decomposition:
\[
\Sym^2\Bl(E\tensor F^*\Br)\cong
\bigoplus_{{u,v\in\{0,\ldots ,t\}}\atop{u+v \text{ even}}}L_{\tau_u}V\tensor
L_{\tau_v}W^*.
\]
Since the above decomposition is multiplicity free, exactly the
asymmetric shapes belong to $(J_t)_2$:
\[
(J_t)_2\cong\bigoplus_{{u,v\in\{0,\ldots ,t\}}\atop{{u+v \text{ even}}\atop{u\neq v}}}
L_{\tau_u}V\tensor L_{\tau_v}W^*.
\]
So, the highest bi-weight vector of the bi-diagram
$(\tau_u|\tau_v)$, with $u+v$ even and $u\neq v$, is the
following element:
\begin{equation}\label{expldeg2}
\ff_{u,v}=\sum_{{I,J}\atop{H,K}}(-1)^{I,J}(-1)^{H,K}[1,\dots,t-u,I|1,\dots,t-v,H]
[1, \dots,t-u,J|1,\dots,t-v,K]
\end{equation}
where the sum runs over the $2$-partitions $I,J$ of
$\{t-u+1,\dots,t+u\}$ and  $H,K$ of $\{t-v+1,\dots,t+v\}$ such that
$|I|=|J|=u$ and $|H|=|K|=v$. Furthermore one can assume that $I$ is
lexicographically smaller than $J$, so that the relation is the
original one divided by $2$. When we need to emphasize the size of
minors, we will write $\ff_{u,v}^\st$.

\begin{remark}\label{remnotonlypluecker}
Notice that $\ff_{u,v}$ is a Pl\"ucker relation if and only if $u=0$
or $v=0$. Moreover, if $t>\max\{u,v\}$, then $\ff_{u,v}^\st$ is
obtained by trivial extension from $\ff_{u,v}^\su$ or
$\ff_{u,v}^\sv$, according to whether $u>v$ or $v>u$ (Proposition
\ref{propretract}).
\end{remark}

\subsection{Cubic shape relations}\label{sec_cubic_sh}

We will determine relations of degree $3$ that are minimal
generators of $J_t$. We will see that they are shape relations,
and in Subsection \ref{nomoredeg3sh} it will be shown that
there are no other shape relations in degree $3$.

A minimal relation between $t$-minors is said to be
\emph{really new} if it does not come from a relation
between $(t-1)$-minors by trivial extension. Every time that
$t$ increases by one a really new type of minimal cubic
relation shows up (provided that $m\geq \lceil t/2 \rceil$ and
$n\geq 2t$).   Such really new cubic minimal relations exist
for slightly different reasons according to whether $t$ is even
or odd, therefore we will divide this subsection in two parts.

\subsubsection{Even minimal cubics}\label{evencubics}

Despite of the title, in this first part we will construct
minimal cubic relations between $t$-minors for any $t$ (also
for odd $t$). However, they will be really new only if $t$ is
even. To this purpose we define some special bi-diagrams
$(\gamma_u |\lambda_u)$ for any $\displaystyle u=1,\dots,
\lfloor t/2 \rfloor$, for which both $\gamma_u$ and $\lambda_u$
are partitions of $3t$. In Theorem \ref{shapeink3} we will
prove that some of these bi-diagrams (actually each of them if
the size of the matrix is big enough) are minimal irreducible
representations of degree $3$ in $J_t$.

For all $u=1,\dots,\lfloor t/2\rfloor$, we define the bi-diagram
$(\gamma_u|\lambda_u)$ ($=(\gamma_u^\st|\lambda_u^t)$ if we need to
emphasize the size of the minors) by
\begin{equation}
\begin{aligned}
 \gamma_u & =(t+u, \ t+u, \ t-2u), \\
\lambda_u  & =(t+2u, \ t-u, \ t-u).
\end{aligned}\label{kerdiage}
\end{equation}
Notice that $\gamma_u$ and $\lambda_u$ are both partitions of
$3t$. Furthermore, provided that $m\geq t+u$ and $n\geq t+2u$,
the irreducible $G$-representation $L_{\gamma_u}V\tensor
L_{\lambda_u}W^*$ occurs in $(T_t)_3$.

\begin{remark}
Notice that, if $t$ is odd, the bi-diagram
$(\gamma_u^\st|\lambda_u^t)$ is a trivial extension of
$(\gamma_u^{\,t-1}|\lambda_u^{t-1})$ by Proposition
\ref{propretract}. Therefore $(\gamma_u^\st|\lambda_u^t)$ is
really new if and only if $t$ is even and $u=t/2$.
\end{remark}

\begin{thm}\label{shapeink3}
The bi-diagram $(\gamma_u|\lambda_u)$ is a $T$-shape relation
of degree $3$ and therefore a minimal irreducible
representation of $J_t(m,n)$ (provided that $u\leq m-t$ and
$2u\leq n-t$).
\end{thm}

\begin{proof}
The only bi-predecessor of $(\gamma_u|\lambda_u)$ is the
bi-diagram $(\tau|\tau)$ with $\tau=(t+u,t-u)$. Since $\tau$
has degree $2$, it has multiplicity $1$. This shows that
$(\gamma_u|\lambda_u)$ is a $T$-shape relation, and we can
apply Proposition \ref{Tshape}
\end{proof}

\begin{corollary}\label{cubics}
The ideal $J_t$ has some minimal generators of degree $3$ apart
from the cases discussed in Remark \ref{partcases}.
\end{corollary}

\begin{proof}
In this situation the bi-diagram $(\gamma_1|\lambda_1)$ always
satisfies the side condition of Theorem \ref{shapeink3}.
\end{proof}

\subsubsection{Odd minimal cubics}\label{secoddrelations}

Once again despite of the title, in this second part we will
construct other minimal cubic relations between $t$-minors for
any $t$ (also for even $t$). However, they will be really new
only if $t$ is odd. Here the proof is more tricky than the one
for the even cubics since the odd ones are not $T$-shape
relations.

For all $u=2,\dots,\lceil t/2\rceil$, we define the bi-diagram
$=(\rho_{u}|\sigma_{u})$ ($(\rho_{u}^t|\sigma_{u}^t)$ if we
want to emphasize the size of the minors) by
\begin{equation}
\begin{aligned}
\rho_{u} & =(t+u, \ t+u-1, \ t-2u+1), \\
\sigma_{u}  & =(t+2u-1, \ t-u+1, \ t-u).
\end{aligned}\label{kerodd}
\end{equation}
Notice that both $\rho_{u}$ and $\sigma_{u}$ are partitions of $3t$.

\begin{remark}
If $t$ is even, the bi-diagram $(\rho_u^t|\sigma_u^\st)$  is a
trivial extension of $(\rho_u^{t-1}|\sigma_u^{\,t-1})$ by
Proposition \ref{propretract}. So minimal relations we are
going to describe now are really new only if $t$ is odd and
$u=\lceil t/2\rceil$.
\end{remark}

\begin{thm}\label{shapeink3odd}
The bi-diagram $(\rho_u|\sigma_u)$ is a shape relation (of
single $\Bw^t$-type) and therefore a minimal irreducible
representation of $J_t(m,n)$ of degree $3$ (provided that
$u\leq m-t$ and $2u\leq n-t+1$).
\end{thm}
\begin{proof}
Notice that $\rho_u$ has two predecessors, namely $(t+u,t-u)$
and $(t+u-1,t-u+1)$. Also $\sigma_u$ has two predecessors,
namely $(t+u,t-u)$ and $(t+u-1,t-u+1)$. Therefore Proposition
\ref{propinthemiddle} implies that
\[
L_{\rho_u}V\subseteq L_{(2,1)}E
\]
and
\[
L_{\sigma_u}W^*\subseteq L_{(2,1)}F^*.
\]
So, exploiting \eqref{CauchyH}, we get that $(\rho_u|\sigma_u)$
is a $G$-subrepresentation of $S_t(m,n)_{(2,1)}$.
Moreover, Lemma
\ref{decs2at} implies that the only two pairs (of the
predecessors of $\rho_u$ and $\sigma_u$) living in $S_t(m,n)$
are $((t+u,t-u)|(t+u,t-u))$ and
$((t+u-1,t+u-1)|(t+u-1,t-u+1))$. Both of these are symmetric
bi-diagrams in degree $2$, and it follows that
$(\rho_u|\sigma_u)$ is a shape relation.
\end{proof}

Since $(\rho_u|\sigma_u)$ has an asymmetric bi-predecessor in
$T_t$ it is not a $T$-shape relation.

\subsection{A second look at the minimal relations}

The goal of this subsection is to augment the information on the
minimal relations we found in this section. In Figure \ref{bidi2}
below we will feature the bi-shapes $(\tau_u|\tau_v)$ corresponding
to quadratic minimal relations when $u+v$ is even and $u<v$. Of
course, one has to keep in mind that there are also the quadratic
minimal relations corresponding to the mirrored bi-shapes, namely
$(\tau_u|\tau_v)$ for $u>v$.

\begin{figure}[htb]
{\setlength{\unitlength}{1.7mm}
\begin{picture}(70,60)(-4,-15)

\put(-10,38.5){\line(0,-1){55}} \put(-2,38.5){\line(0,-1){55}}
\put(10,38.5){\line(0,-1){55}} \put(26,38.5){\line(0,-1){55}}
\put(46,38.5){\line(0,-1){55}} \put(70,38.5){\line(0,-1){55}}

\put(-16,35){\line(1,0){93}} \put(-16,30){\line(1,0){93}}
\put(-16,25){\line(1,0){93}} \put(-16,20){\line(1,0){93}}
\put(-16,15){\line(1,0){93}} \put(-16,10){\line(1,0){93}}
\put(-16,5){\line(1,0){93}}\put(-16,0){\line(1,0){93}}
\put(-16,-5){\line(1,0){93}}\put(-16,-10){\line(1,0){93}}
\put(-16,-15){\line(1,0){93}}

\put(-7.5,36.5){{\scriptsize $t=2$}} \put(2.3,36.5){{\scriptsize
$t=3$}} \put(16.3,36.5){{\scriptsize $t=4$}}
\put(34,36.5){{\scriptsize $t=5$}} \put(55.5,36.5){{\scriptsize
$t=6$}}

\put(-15.5,32){{\scriptsize $(\tau_0|\tau_2)$}}

\put(-15.5,27){{\scriptsize $(\tau_1|\tau_3)$}}

\put(-15.5,22){{\scriptsize $(\tau_0|\tau_4)$}}

\put(-15.5,17){{\scriptsize $(\tau_2|\tau_4)$}}

\put(-15.5,12){{\scriptsize $(\tau_1|\tau_5)$}}

\put(-15.5,7){{\scriptsize $(\tau_3|\tau_5)$}}

\put(-15.5,2){{\scriptsize $(\tau_0|\tau_6)$}}

\put(-15.5,-3){{\scriptsize $(\tau_2|\tau_6)$}}

\put(-15.5,-8){{\scriptsize $(\tau_4|\tau_6)$}}

\put(-13,-13.5){$\vdots$}

\put(73,36.5){$\dots$}

\put(73,32.5){$\dots$}

\put(73,27.5){$\dots$}

\put(73,22.5){$\dots$}

\put(73,17.5){$\dots$}

\put(73,12.5){$\dots$}

\put(73,7.5){$\dots$}

\put(73,2.5){$\dots$}

\put(73,-2.5){$\dots$}

\put(73,-7.5){$\dots$}

\put(73,-13.5){$\ddots$}


\put(-9.5,33.5){\line(1,0){2}} \put(-9.5,32.5){\line(1,0){2}}
\put(-9.5,31.5){\line(1,0){2}} \put(-9.5,33.5){\line(0,-1){2}}
\put(-8.5,33.5){\line(0,-1){2}} \put(-7.5,33.5){\line(0,-1){2}}

\put(-7,34){\line(0,-1){3}}

\put(-6.5,33.5){\line(1,0){4}} \put(-6.5,32.5){\line(1,0){4}}
\put(-6.5,33.5){\line(0,-1){1}} \put(-5.5,33.5){\line(0,-1){1}}
\put(-4.5,33.5){\line(0,-1){1}} \put(-3.5,33.5){\line(0,-1){1}}
\put(-2.5,33.5){\line(0,-1){1}}


\put(-0.5,33.5){\line(1,0){3}} \put(-0.5,32.5){\line(1,0){3}}
\put(-0.5,31.5){\line(1,0){3}} \put(-0.5,33.5){\line(0,-1){2}}
\put(0.5,33.5){\line(0,-1){2}} \put(1.5,33.5){\line(0,-1){2}}
\put(2.5,33.5){\line(0,-1){2}}

\put(3,34){\line(0,-1){3}}

\put(3.5,33.5){\line(1,0){5}} \put(3.5,32.5){\line(1,0){5}}
\put(3.5,31.5){\line(1,0){1}} \put(3.5,33.5){\line(0,-1){2}}
\put(4.5,33.5){\line(0,-1){2}} \put(5.5,33.5){\line(0,-1){1}}
\put(6.5,33.5){\line(0,-1){1}}\put(7.5,33.5){\line(0,-1){1}}
\put(8.5,33.5){\line(0,-1){1}}


\put(-1.5,28.5){\line(1,0){4}} \put(-1.5,27.5){\line(1,0){4}}
\put(-1.5,26.5){\line(1,0){2}} \put(-1.5,28.5){\line(0,-1){2}}
\put(-0.5,28.5){\line(0,-1){2}} \put(0.5,28.5){\line(0,-1){2}}
\put(1.5,28.5){\line(0,-1){1}}\put(2.5,28.5){\line(0,-1){1}}

\put(3,29){\line(0,-1){3}}

\put(3.5,28.5){\line(1,0){6}} \put(3.5,27.5){\line(1,0){6}}
\put(3.5,28.5){\line(0,-1){1}}
\put(4.5,28.5){\line(0,-1){1}} \put(5.5,28.5){\line(0,-1){1}}
\put(6.5,28.5){\line(0,-1){1}}\put(7.5,28.5){\line(0,-1){1}}
\put(8.5,28.5){\line(0,-1){1}}\put(9.5,28.5){\line(0,-1){1}}


\put(12.5,33.5){\line(1,0){4}} \put(12.5,32.5){\line(1,0){4}}
\put(12.5,31.5){\line(1,0){4}} \put(12.5,33.5){\line(0,-1){2}}
\put(13.5,33.5){\line(0,-1){2}} \put(14.5,33.5){\line(0,-1){2}}
\put(15.5,33.5){\line(0,-1){2}}\put(16.5,33.5){\line(0,-1){2}}

\put(17,34){\line(0,-1){3}}

\put(17.5,33.5){\line(1,0){6}}\put(17.5,32.5){\line(1,0){6}}
\put(17.5,31.5){\line(1,0){2}} \put(17.5,33.5){\line(0,-1){2}}
\put(18.5,33.5){\line(0,-1){2}} \put(19.5,33.5){\line(0,-1){2}}
\put(20.5,33.5){\line(0,-1){1}}\put(21.5,33.5){\line(0,-1){1}}
\put(22.5,33.5){\line(0,-1){1}}\put(23.5,33.5){\line(0,-1){1}}


\put(11.5,28.5){\line(1,0){5}} \put(11.5,27.5){\line(1,0){5}}
\put(11.5,26.5){\line(1,0){3}} \put(11.5,28.5){\line(0,-1){2}}
\put(12.5,28.5){\line(0,-1){2}} \put(13.5,28.5){\line(0,-1){2}}
\put(14.5,28.5){\line(0,-1){2}}\put(15.5,28.5){\line(0,-1){1}}
\put(16.5,28.5){\line(0,-1){1}}\put(17.5,28.5){\line(0,-1){1}}

\put(17,29){\line(0,-1){3}}

\put(17.5,28.5){\line(1,0){7}} \put(17.5,27.5){\line(1,0){7}}
\put(17.5,26.5){\line(1,0){1}}
\put(17.5,28.5){\line(0,-1){2}}\put(18.5,28.5){\line(0,-1){2}}
\put(19.5,28.5){\line(0,-1){1}}
\put(20.5,28.5){\line(0,-1){1}} \put(21.5,28.5){\line(0,-1){1}}
\put(22.5,28.5){\line(0,-1){1}}\put(23.5,28.5){\line(0,-1){1}}
\put(24.5,28.5){\line(0,-1){1}}


\put(11.5,23.5){\line(1,0){4}} \put(11.5,22.5){\line(1,0){4}}
\put(11.5,21.5){\line(1,0){4}} \put(11.5,23.5){\line(0,-1){2}}
\put(12.5,23.5){\line(0,-1){2}} \put(13.5,23.5){\line(0,-1){2}}
\put(14.5,23.5){\line(0,-1){2}}\put(15.5,23.5){\line(0,-1){2}}

\put(16,24){\line(0,-1){3}}

\put(16.5,23.5){\line(1,0){8}} \put(16.5,22.5){\line(1,0){8}}
\put(16.5,23.5){\line(0,-1){1}}\put(17.5,23.5){\line(0,-1){1}}
\put(18.5,23.5){\line(0,-1){1}}
\put(19.5,23.5){\line(0,-1){1}} \put(20.5,23.5){\line(0,-1){1}}
\put(21.5,23.5){\line(0,-1){1}}\put(22.5,23.5){\line(0,-1){1}}
\put(23.5,23.5){\line(0,-1){1}}\put(24.5,23.5){\line(0,-1){1}}


\put(10.5,18.5){\line(1,0){6}} \put(10.5,17.5){\line(1,0){6}}
\put(10.5,16.5){\line(1,0){2}} \put(10.5,18.5){\line(0,-1){2}}
\put(11.5,18.5){\line(0,-1){2}} \put(12.5,18.5){\line(0,-1){2}}
\put(13.5,18.5){\line(0,-1){1}}\put(14.5,18.5){\line(0,-1){1}}
\put(15.5,18.5){\line(0,-1){1}}\put(16.5,18.5){\line(0,-1){1}}

\put(17,19){\line(0,-1){3}}

\put(17.5,18.5){\line(1,0){8}} \put(17.5,17.5){\line(1,0){8}}
\put(17.5,18.5){\line(0,-1){1}}\put(18.5,18.5){\line(0,-1){1}}
\put(19.5,18.5){\line(0,-1){1}}
\put(20.5,18.5){\line(0,-1){1}} \put(21.5,18.5){\line(0,-1){1}}
\put(22.5,18.5){\line(0,-1){1}}\put(23.5,18.5){\line(0,-1){1}}
\put(24.5,18.5){\line(0,-1){1}}\put(25.5,18.5){\line(0,-1){1}}


\put(29.5,33.5){\line(1,0){5}}\put(29.5,32.5){\line(1,0){5}}
\put(29.5,31.5){\line(1,0){5}} \put(29.5,33.5){\line(0,-1){2}}
\put(30.5,33.5){\line(0,-1){2}} \put(31.5,33.5){\line(0,-1){2}}
\put(32.5,33.5){\line(0,-1){2}}\put(33.5,33.5){\line(0,-1){2}}
\put(34.5,33.5){\line(0,-1){2}}

\put(35,34){\line(0,-1){3}}

\put(35.5,33.5){\line(1,0){7}}\put(35.5,32.5){\line(1,0){7}}
\put(35.5,31.5){\line(1,0){3}} \put(35.5,33.5){\line(0,-1){2}}
\put(36.5,33.5){\line(0,-1){2}} \put(37.5,33.5){\line(0,-1){2}}
\put(38.5,33.5){\line(0,-1){2}}\put(39.5,33.5){\line(0,-1){1}}
\put(40.5,33.5){\line(0,-1){1}}\put(41.5,33.5){\line(0,-1){1}}
\put(42.5,33.5){\line(0,-1){1}}


\put(28.5,28.5){\line(1,0){6}} \put(28.5,27.5){\line(1,0){6}}
\put(28.5,26.5){\line(1,0){4}} \put(28.5,28.5){\line(0,-1){2}}
\put(29.5,28.5){\line(0,-1){2}} \put(30.5,28.5){\line(0,-1){2}}
\put(31.5,28.5){\line(0,-1){2}}\put(32.5,28.5){\line(0,-1){2}}
\put(33.5,28.5){\line(0,-1){1}}\put(34.5,28.5){\line(0,-1){1}}

\put(35,29){\line(0,-1){3}}

\put(35.5,28.5){\line(1,0){8}} \put(35.5,27.5){\line(1,0){8}}
\put(35.5,26.5){\line(1,0){2}}
\put(35.5,28.5){\line(0,-1){2}}\put(36.5,28.5){\line(0,-1){2}}
\put(37.5,28.5){\line(0,-1){2}}
\put(38.5,28.5){\line(0,-1){1}} \put(39.5,28.5){\line(0,-1){1}}
\put(40.5,28.5){\line(0,-1){1}}\put(41.5,28.5){\line(0,-1){1}}
\put(42.5,28.5){\line(0,-1){1}}\put(43.5,28.5){\line(0,-1){1}}


\put(28.5,23.5){\line(1,0){5}} \put(28.5,22.5){\line(1,0){5}}
\put(28.5,21.5){\line(1,0){5}} \put(28.5,23.5){\line(0,-1){2}}
\put(29.5,23.5){\line(0,-1){2}} \put(30.5,23.5){\line(0,-1){2}}
\put(31.5,23.5){\line(0,-1){2}}\put(32.5,23.5){\line(0,-1){2}}
\put(33.5,23.5){\line(0,-1){2}}

\put(34,24){\line(0,-1){3}}

\put(34.5,23.5){\line(1,0){9}} \put(34.5,22.5){\line(1,0){9}}
 \put(34.5,21.5){\line(1,0){1}}\put(35.5,23.5){\line(0,-1){2}}\put(36.5,23.5){\line(0,-1){1}}
\put(37.5,23.5){\line(0,-1){1}}
\put(38.5,23.5){\line(0,-1){1}} \put(39.5,23.5){\line(0,-1){1}}
\put(40.5,23.5){\line(0,-1){1}}\put(41.5,23.5){\line(0,-1){1}}
\put(42.5,23.5){\line(0,-1){1}}\put(43.5,23.5){\line(0,-1){1}}
\put(34.5,23.5){\line(0,-1){2}}


\put(27.5,18.5){\line(1,0){7}} \put(27.5,17.5){\line(1,0){7}}
\put(27.5,16.5){\line(1,0){3}} \put(27.5,18.5){\line(0,-1){2}}
\put(28.5,18.5){\line(0,-1){2}} \put(29.5,18.5){\line(0,-1){2}}
\put(30.5,18.5){\line(0,-1){2}}\put(31.5,18.5){\line(0,-1){1}}
\put(32.5,18.5){\line(0,-1){1}}\put(33.5,18.5){\line(0,-1){1}}
\put(34.5,18.5){\line(0,-1){1}}

\put(35,19){\line(0,-1){3}}

\put(35.5,18.5){\line(1,0){9}} \put(35.5,17.5){\line(1,0){9}}
\put(35.5,16.5){\line(1,0){1}}
\put(35.5,18.5){\line(0,-1){2}}\put(36.5,18.5){\line(0,-1){2}}
\put(37.5,18.5){\line(0,-1){1}}\put(38.5,18.5){\line(0,-1){1}}
\put(39.5,18.5){\line(0,-1){1}} \put(40.5,18.5){\line(0,-1){1}}
\put(41.5,18.5){\line(0,-1){1}}\put(42.5,18.5){\line(0,-1){1}}
\put(43.5,18.5){\line(0,-1){1}}\put(44.5,18.5){\line(0,-1){1}}


\put(27.5,13.5){\line(1,0){6}} \put(27.5,12.5){\line(1,0){6}}
\put(27.5,11.5){\line(1,0){4}} \put(27.5,13.5){\line(0,-1){2}}
\put(28.5,13.5){\line(0,-1){2}} \put(29.5,13.5){\line(0,-1){2}}
\put(30.5,13.5){\line(0,-1){2}}\put(31.5,13.5){\line(0,-1){2}}
\put(32.5,13.5){\line(0,-1){1}}\put(33.5,13.5){\line(0,-1){1}}

\put(34,14){\line(0,-1){3}}

\put(34.5,13.5){\line(1,0){10}} \put(34.5,12.5){\line(1,0){10}}
\put(34.5,13.5){\line(0,-1){1}}
\put(35.5,13.5){\line(0,-1){1}}\put(36.5,13.5){\line(0,-1){1}}
\put(37.5,13.5){\line(0,-1){1}}\put(38.5,13.5){\line(0,-1){1}}
\put(39.5,13.5){\line(0,-1){1}} \put(40.5,13.5){\line(0,-1){1}}
\put(41.5,13.5){\line(0,-1){1}}\put(42.5,13.5){\line(0,-1){1}}
\put(43.5,13.5){\line(0,-1){1}}\put(44.5,13.5){\line(0,-1){1}}


\put(26.5,8.5){\line(1,0){8}} \put(26.5,7.5){\line(1,0){8}}
\put(26.5,6.5){\line(1,0){2}} \put(26.5,8.5){\line(0,-1){2}}
\put(27.5,8.5){\line(0,-1){2}} \put(28.5,8.5){\line(0,-1){2}}
\put(29.5,8.5){\line(0,-1){1}}\put(30.5,8.5){\line(0,-1){1}}
\put(31.5,8.5){\line(0,-1){1}}\put(32.5,8.5){\line(0,-1){1}}
\put(33.5,8.5){\line(0,-1){1}}\put(34.5,8.5){\line(0,-1){1}}

\put(35,9){\line(0,-1){3}}

\put(35.5,8.5){\line(1,0){10}} \put(35.5,7.5){\line(1,0){10}}
\put(35.5,8.5){\line(0,-1){1}}\put(36.5,8.5){\line(0,-1){1}}
\put(37.5,8.5){\line(0,-1){1}}
\put(38.5,8.5){\line(0,-1){1}} \put(39.5,8.5){\line(0,-1){1}}
\put(40.5,8.5){\line(0,-1){1}}\put(41.5,8.5){\line(0,-1){1}}
\put(42.5,8.5){\line(0,-1){1}}\put(43.5,8.5){\line(0,-1){1}}
\put(44.5,8.5){\line(0,-1){1}}\put(45.5,8.5){\line(0,-1){1}}


\put(50.5,33.5){\line(1,0){6}}\put(50.5,32.5){\line(1,0){6}}
\put(50.5,31.5){\line(1,0){6}} \put(50.5,33.5){\line(0,-1){2}}
\put(51.5,33.5){\line(0,-1){2}} \put(52.5,33.5){\line(0,-1){2}}
\put(53.5,33.5){\line(0,-1){2}}\put(54.5,33.5){\line(0,-1){2}}
\put(55.5,33.5){\line(0,-1){2}}\put(56.5,33.5){\line(0,-1){2}}

\put(57,34){\line(0,-1){3}}

\put(57.5,33.5){\line(1,0){8}}\put(57.5,32.5){\line(1,0){8}}
\put(57.5,31.5){\line(1,0){4}} \put(57.5,33.5){\line(0,-1){2}}
\put(58.5,33.5){\line(0,-1){2}} \put(59.5,33.5){\line(0,-1){2}}
\put(60.5,33.5){\line(0,-1){2}}\put(61.5,33.5){\line(0,-1){2}}
\put(62.5,33.5){\line(0,-1){1}}\put(63.5,33.5){\line(0,-1){1}}
\put(64.5,33.5){\line(0,-1){1}}\put(65.5,33.5){\line(0,-1){1}}


\put(49.5,28.5){\line(1,0){7}} \put(49.5,27.5){\line(1,0){7}}
\put(49.5,26.5){\line(1,0){5}} \put(49.5,28.5){\line(0,-1){2}}
\put(50.5,28.5){\line(0,-1){2}} \put(51.5,28.5){\line(0,-1){2}}
\put(52.5,28.5){\line(0,-1){2}}\put(53.5,28.5){\line(0,-1){2}}
\put(54.5,28.5){\line(0,-1){2}}\put(55.5,28.5){\line(0,-1){1}}
\put(56.5,28.5){\line(0,-1){1}}

\put(57,29){\line(0,-1){3}}

\put(57.5,28.5){\line(1,0){9}} \put(57.5,27.5){\line(1,0){9}}
\put(57.5,26.5){\line(1,0){3}}
\put(57.5,28.5){\line(0,-1){2}}\put(58.5,28.5){\line(0,-1){2}}
\put(59.5,28.5){\line(0,-1){2}}
\put(60.5,28.5){\line(0,-1){2}} \put(61.5,28.5){\line(0,-1){1}}
\put(62.5,28.5){\line(0,-1){1}}\put(63.5,28.5){\line(0,-1){1}}
\put(64.5,28.5){\line(0,-1){1}}\put(65.5,28.5){\line(0,-1){1}}
\put(66.5,28.5){\line(0,-1){1}}


\put(49.5,23.5){\line(1,0){6}} \put(49.5,22.5){\line(1,0){6}}
\put(49.5,21.5){\line(1,0){6}} \put(49.5,23.5){\line(0,-1){2}}
\put(50.5,23.5){\line(0,-1){2}} \put(51.5,23.5){\line(0,-1){2}}
\put(52.5,23.5){\line(0,-1){2}}\put(53.5,23.5){\line(0,-1){2}}
\put(54.5,23.5){\line(0,-1){2}}\put(55.5,23.5){\line(0,-1){2}}

\put(56,24){\line(0,-1){3}}

\put(56.5,23.5){\line(1,0){10}} \put(56.5,22.5){\line(1,0){10}}
 \put(56.5,21.5){\line(1,0){2}}\put(57.5,23.5){\line(0,-1){2}}\put(58.5,23.5){\line(0,-1){2}}
\put(59.5,23.5){\line(0,-1){1}}
\put(60.5,23.5){\line(0,-1){1}} \put(61.5,23.5){\line(0,-1){1}}
\put(62.5,23.5){\line(0,-1){1}}\put(63.5,23.5){\line(0,-1){1}}
\put(64.5,23.5){\line(0,-1){1}}\put(65.5,23.5){\line(0,-1){1}}
\put(56.5,23.5){\line(0,-1){2}}\put(66.5,23.5){\line(0,-1){1}}


\put(48.5,18.5){\line(1,0){8}} \put(48.5,17.5){\line(1,0){8}}
\put(48.5,16.5){\line(1,0){4}} \put(48.5,18.5){\line(0,-1){2}}
\put(49.5,18.5){\line(0,-1){2}} \put(50.5,18.5){\line(0,-1){2}}
\put(51.5,18.5){\line(0,-1){2}}\put(52.5,18.5){\line(0,-1){2}}
\put(53.5,18.5){\line(0,-1){1}}\put(54.5,18.5){\line(0,-1){1}}
\put(55.5,18.5){\line(0,-1){1}}\put(56.5,18.5){\line(0,-1){1}}

\put(57,19){\line(0,-1){3}}

\put(57.5,18.5){\line(1,0){10}} \put(57.5,17.5){\line(1,0){10}}
\put(57.5,16.5){\line(1,0){2}}
\put(57.5,18.5){\line(0,-1){2}}\put(58.5,18.5){\line(0,-1){2}}
\put(59.5,18.5){\line(0,-1){2}}\put(60.5,18.5){\line(0,-1){1}}
\put(61.5,18.5){\line(0,-1){1}} \put(62.5,18.5){\line(0,-1){1}}
\put(63.5,18.5){\line(0,-1){1}}\put(64.5,18.5){\line(0,-1){1}}
\put(65.5,18.5){\line(0,-1){1}}\put(66.5,18.5){\line(0,-1){1}}
\put(67.5,18.5){\line(0,-1){1}}


\put(48.5,13.5){\line(1,0){7}} \put(48.5,12.5){\line(1,0){7}}
\put(48.5,11.5){\line(1,0){5}} \put(48.5,13.5){\line(0,-1){2}}
\put(49.5,13.5){\line(0,-1){2}} \put(50.5,13.5){\line(0,-1){2}}
\put(51.5,13.5){\line(0,-1){2}}\put(52.5,13.5){\line(0,-1){2}}
\put(53.5,13.5){\line(0,-1){2}}\put(54.5,13.5){\line(0,-1){1}}
\put(55.5,13.5){\line(0,-1){1}}

\put(56,14){\line(0,-1){3}}

\put(56.5,13.5){\line(1,0){11}} \put(56.5,12.5){\line(1,0){11}}
\put(56.5,11.5){\line(1,0){1}}\put(56.5,13.5){\line(0,-1){2}}
\put(57.5,13.5){\line(0,-1){2}}\put(58.5,13.5){\line(0,-1){1}}
\put(59.5,13.5){\line(0,-1){1}}\put(60.5,13.5){\line(0,-1){1}}
\put(61.5,13.5){\line(0,-1){1}} \put(62.5,13.5){\line(0,-1){1}}
\put(63.5,13.5){\line(0,-1){1}}\put(64.5,13.5){\line(0,-1){1}}
\put(65.5,13.5){\line(0,-1){1}}\put(66.5,13.5){\line(0,-1){1}}
\put(67.5,13.5){\line(0,-1){1}}


\put(47.5,8.5){\line(1,0){9}} \put(47.5,7.5){\line(1,0){9}}
\put(47.5,6.5){\line(1,0){3}} \put(47.5,8.5){\line(0,-1){2}}
\put(48.5,8.5){\line(0,-1){2}} \put(49.5,8.5){\line(0,-1){2}}
\put(50.5,8.5){\line(0,-1){2}}\put(51.5,8.5){\line(0,-1){1}}
\put(53.5,8.5){\line(0,-1){1}}\put(52.5,8.5){\line(0,-1){1}}
\put(55.5,8.5){\line(0,-1){1}}\put(54.5,8.5){\line(0,-1){1}}
\put(56.5,8.5){\line(0,-1){1}}

\put(57,9){\line(0,-1){3}}

\put(57.5,8.5){\line(1,0){11}} \put(57.5,7.5){\line(1,0){11}}
\put(57.5,6.5){\line(1,0){1}}
\put(57.5,8.5){\line(0,-1){2}}\put(58.5,8.5){\line(0,-1){2}}
\put(59.5,8.5){\line(0,-1){1}}
\put(60.5,8.5){\line(0,-1){1}} \put(61.5,8.5){\line(0,-1){1}}
\put(62.5,8.5){\line(0,-1){1}}\put(63.5,8.5){\line(0,-1){1}}
\put(64.5,8.5){\line(0,-1){1}}\put(65.5,8.5){\line(0,-1){1}}
\put(66.5,8.5){\line(0,-1){1}}\put(67.5,8.5){\line(0,-1){1}}
\put(68.5,8.5){\line(0,-1){1}}


\put(48.5,3.5){\line(1,0){6}} \put(48.5,2.5){\line(1,0){6}}
\put(48.5,1.5){\line(1,0){6}} \put(48.5,3.5){\line(0,-1){2}}
\put(49.5,3.5){\line(0,-1){2}} \put(50.5,3.5){\line(0,-1){2}}
\put(51.5,3.5){\line(0,-1){2}}\put(52.5,3.5){\line(0,-1){2}}
\put(53.5,3.5){\line(0,-1){2}}\put(54.5,3.5){\line(0,-1){2}}

\put(55,4){\line(0,-1){3}}

\put(55.5,3.5){\line(1,0){12}} \put(55.5,2.5){\line(1,0){12}}
\put(55.5,3.5){\line(0,-1){1}}\put(56.5,3.5){\line(0,-1){1}}
\put(57.5,3.5){\line(0,-1){1}}
\put(58.5,3.5){\line(0,-1){1}} \put(59.5,3.5){\line(0,-1){1}}
\put(60.5,3.5){\line(0,-1){1}}\put(61.5,3.5){\line(0,-1){1}}
\put(62.5,3.5){\line(0,-1){1}}\put(63.5,3.5){\line(0,-1){1}}
\put(64.5,3.5){\line(0,-1){1}}\put(65.5,3.5){\line(0,-1){1}}
\put(66.5,3.5){\line(0,-1){1}}\put(67.5,3.5){\line(0,-1){1}}


\put(47.5,-1.5){\line(1,0){8}} \put(47.5,-2.5){\line(1,0){8}}
\put(47.5,-3.5){\line(1,0){4}} \put(47.5,-1.5){\line(0,-1){2}}
\put(48.5,-1.5){\line(0,-1){2}} \put(49.5,-1.5){\line(0,-1){2}}
\put(50.5,-1.5){\line(0,-1){2}}\put(51.5,-1.5){\line(0,-1){2}}
\put(52.5,-1.5){\line(0,-1){1}}\put(53.5,-1.5){\line(0,-1){1}}
\put(54.5,-1.5){\line(0,-1){1}}\put(55.5,-1.5){\line(0,-1){1}}

\put(56,-1){\line(0,-1){3}}

\put(56.5,-1.5){\line(1,0){12}} \put(56.5,-2.5){\line(1,0){12}}
\put(56.5,-1.5){\line(0,-1){1}}\put(57.5,-1.5){\line(0,-1){1}}
\put(58.5,-1.5){\line(0,-1){1}}
\put(59.5,-1.5){\line(0,-1){1}} \put(60.5,-1.5){\line(0,-1){1}}
\put(61.5,-1.5){\line(0,-1){1}}\put(62.5,-1.5){\line(0,-1){1}}
\put(63.5,-1.5){\line(0,-1){1}}\put(64.5,-1.5){\line(0,-1){1}}
\put(65.5,-1.5){\line(0,-1){1}}\put(66.5,-1.5){\line(0,-1){1}}
\put(67.5,-1.5){\line(0,-1){1}}\put(68.5,-1.5){\line(0,-1){1}}


\put(46.5,-6.5){\line(1,0){10}} \put(46.5,-7.5){\line(1,0){10}}
\put(46.5,-8.5){\line(1,0){2}} \put(46.5,-6.5){\line(0,-1){2}}
\put(47.5,-6.5){\line(0,-1){2}} \put(48.5,-6.5){\line(0,-1){2}}
\put(49.5,-6.5){\line(0,-1){1}}\put(50.5,-6.5){\line(0,-1){1}}
\put(51.5,-6.5){\line(0,-1){1}}\put(52.5,-6.5){\line(0,-1){1}}
\put(53.5,-6.5){\line(0,-1){1}}\put(54.5,-6.5){\line(0,-1){1}}
\put(55.5,-6.5){\line(0,-1){1}}\put(56.5,-6.5){\line(0,-1){1}}

\put(57,-6){\line(0,-1){3}}

\put(57.5,-6.5){\line(1,0){12}} \put(57.5,-7.5){\line(1,0){12}}
\put(57.5,-6.5){\line(0,-1){1}}\put(58.5,-6.5){\line(0,-1){1}}
\put(59.5,-6.5){\line(0,-1){1}}
\put(60.5,-6.5){\line(0,-1){1}} \put(61.5,-6.5){\line(0,-1){1}}
\put(62.5,-6.5){\line(0,-1){1}}\put(63.5,-6.5){\line(0,-1){1}}
\put(64.5,-6.5){\line(0,-1){1}}\put(65.5,-6.5){\line(0,-1){1}}
\put(66.5,-6.5){\line(0,-1){1}}\put(67.5,-6.5){\line(0,-1){1}}
\put(68.5,-6.5){\line(0,-1){1}}\put(69.5,-6.5){\line(0,-1){1}}

\end{picture}}
\caption{Bi-diagrams of degree $2$ minimal relations}\label{bidi2}
\end{figure}

As we already noticed in Remark \ref{remnotonlypluecker}, Figure
\ref{bidi2} once more shows that if $v<t$, then the relation
$(\tau_u^t|\tau_v^t)$ is a trivial extension of
$(\tau_u^v|\tau_v^v)$. On the other hand, the relations
$(\tau_u^v|\tau_v^v)$ are really new. Therefore, whenever $t$
increases by one, exactly $\lfloor t/2 \rfloor$ really new minimal
quadratic relations appear. Furthermore, notice that the Pl\"ucker
relations between $t$-minors are those with a $2\times t$ rectangle
on one side.

\begin{remark}\label{place of quadratics}
We have already used the coarse decomposition
\[\Sym^2(E\otimes F^*)=\Bl(\Sym^2 E\otimes
\Sym^2F^*\Br) \oplus \Bl(\Bw^2E
\otimes \Bw^2F^*\Br),
\]
so one may wonder where the bi-diagram $(\tau_u|\tau_v)$
is placed. The answer is already clear from Subsection
\ref{subdeg2}, namely:
\begin{itemize}
\item[(i)] $(\tau_u|\tau_v)$ is in $\Sym^2E\otimes
    \Sym^2F^*$ if and only if $u$ and $v$ are
    even;
\item[(ii)] $(\tau_u|\tau_v)$ is in $\Bw^2E\otimes
    \Bw^2F^*$ if and only if $u$ and $v$ are odd.
\end{itemize}
\end{remark}

Now we want to look at the shape of the found minimal cubic
relations. Once again, in Figure \ref{bidi3} we omit the mirrored
relations.

\begin{figure}[htb]
{\setlength{\unitlength}{1.7mm}
\begin{picture}(70,45)(-4,0)

\put(-10,38.5){\line(0,-1){40}} \put(-1,38.5){\line(0,-1){40}}
\put(12,38.5){\line(0,-1){40}} \put(28,38.5){\line(0,-1){40}}
\put(48,38.5){\line(0,-1){40}} \put(71,38.5){\line(0,-1){40}}

\put(-16,35){\line(1,0){93}} \put(-16,29){\line(1,0){93}}
\put(-16,23){\line(1,0){93}} \put(-16,17){\line(1,0){93}}
\put(-16,11){\line(1,0){93}} \put(-16,5){\line(1,0){93}}

\put(-7,36.5){{\scriptsize $t=2$}} \put(3.3,36.5){{\scriptsize
$t=3$}} \put(18,36.5){{\scriptsize $t=4$}}
\put(36,36.5){{\scriptsize $t=5$}} \put(57.5,36.5){{\scriptsize
$t=6$}}

\put(-15.5,31.5){{\scriptsize $(\gamma_1|\lambda_1)$}}

\put(-15.5,25.5){{\scriptsize $(\rho_2|\sigma_2)$}}

\put(-15.5,19.5){{\scriptsize $(\gamma_2|\lambda_2)$}}

\put(-15.5,13.5){{\scriptsize $(\rho_3|\sigma_3)$}}

\put(-15.5,7.5){{\scriptsize $(\gamma_3|\lambda_3)$}}

\put(-13,1.2){$\vdots$}

\put(73,36.5){$\dots$}

\put(73,32){$\dots$}

\put(73,26){$\dots$}

\put(73,20){$\dots$}

\put(73,14){$\dots$}

\put(73,8){$\dots$}

\put(73,1){$\ddots$}


\put(-9.5,33.5){\line(1,0){3}} \put(-9.5,32.5){\line(1,0){3}}
\put(-9.5,31.5){\line(1,0){3}} \put(-9.5,33.5){\line(0,-1){2}}
\put(-8.5,33.5){\line(0,-1){2}} \put(-7.5,33.5){\line(0,-1){2}}
\put(-6.5,33.5){\line(0,-1){2}}

\put(-6,34){\line(0,-1){4}}

\put(-5.5,33.5){\line(1,0){4}} \put(-5.5,32.5){\line(1,0){4}}
\put(-5.5,31.5){\line(1,0){1}} \put(-5.5,30.5){\line(1,0){1}}
\put(-5.5,33.5){\line(0,-1){3}} \put(-4.5,33.5){\line(0,-1){3}}
\put(-3.5,33.5){\line(0,-1){1}} \put(-2.5,33.5){\line(0,-1){1}}
\put(-1.5,33.5){\line(0,-1){1}}


\put(0.5,33.5){\line(1,0){4}} \put(0.5,32.5){\line(1,0){4}}
\put(0.5,31.5){\line(1,0){4}} \put(0.5,30.5){\line(1,0){1}}
\put(0.5,33.5){\line(0,-1){3}} \put(1.5,33.5){\line(0,-1){3}}
\put(2.5,33.5){\line(0,-1){2}} \put(3.5,33.5){\line(0,-1){2}}
\put(4.5,33.5){\line(0,-1){2}}

\put(5,34){\line(0,-1){4}}

\put(5.5,33.5){\line(1,0){5}} \put(5.5,32.5){\line(1,0){5}}
\put(5.5,31.5){\line(1,0){2}} \put(5.5,30.5){\line(1,0){2}}
\put(5.5,33.5){\line(0,-1){3}} \put(6.5,33.5){\line(0,-1){3}}
\put(7.5,33.5){\line(0,-1){3}} \put(8.5,33.5){\line(0,-1){1}}
\put(9.5,33.5){\line(0,-1){1}} \put(10.5,33.5){\line(0,-1){1}}


\put(-0.5,27.5){\line(1,0){5}} \put(-0.5,26.5){\line(1,0){5}}
\put(-0.5,25.5){\line(1,0){4}} \put(-0.5,27.5){\line(0,-1){2}}
\put(0.5,27.5){\line(0,-1){2}} \put(1.5,27.5){\line(0,-1){2}}
\put(2.5,27.5){\line(0,-1){2}} \put(3.5,27.5){\line(0,-1){2}}
\put(4.5,27.5){\line(0,-1){1}}

\put(5,28){\line(0,-1){4}}

\put(5.5,27.5){\line(1,0){6}} \put(5.5,26.5){\line(1,0){6}}
\put(5.5,25.5){\line(1,0){2}} \put(5.5,24.5){\line(1,0){1}}
\put(5.5,27.5){\line(0,-1){3}} \put(6.5,27.5){\line(0,-1){3}}
\put(7.5,27.5){\line(0,-1){2}} \put(8.5,27.5){\line(0,-1){1}}
\put(9.5,27.5){\line(0,-1){1}} \put(10.5,27.5){\line(0,-1){1}}
\put(11.5,27.5){\line(0,-1){1}}


\put(14,33.5){\line(1,0){5}} \put(14,32.5){\line(1,0){5}}
\put(14,31.5){\line(1,0){5}} \put(14,30.5){\line(1,0){2}}
\put(14,33.5){\line(0,-1){3}} \put(15,33.5){\line(0,-1){3}}
\put(16,33.5){\line(0,-1){3}} \put(17,33.5){\line(0,-1){2}}
\put(18,33.5){\line(0,-1){2}} \put(19,33.5){\line(0,-1){2}}

\put(19.5,34){\line(0,-1){4}}

\put(20,33.5){\line(1,0){6}} \put(20,32.5){\line(1,0){6}}
\put(20,31.5){\line(1,0){3}} \put(20,30.5){\line(1,0){3}}
\put(20,33.5){\line(0,-1){3}} \put(21,33.5){\line(0,-1){3}}
\put(22,33.5){\line(0,-1){3}} \put(23,33.5){\line(0,-1){3}}
\put(24,33.5){\line(0,-1){1}} \put(25,33.5){\line(0,-1){1}}
\put(26,33.5){\line(0,-1){1}}


\put(13,27.5){\line(1,0){6}} \put(13,26.5){\line(1,0){6}}
\put(13,25.5){\line(1,0){5}} \put(13,24.5){\line(1,0){1}}
\put(13,27.5){\line(0,-1){3}} \put(14,27.5){\line(0,-1){3}}
\put(15,27.5){\line(0,-1){2}} \put(16,27.5){\line(0,-1){2}}
\put(17,27.5){\line(0,-1){2}}\put(18,27.5){\line(0,-1){2}}
\put(19,27.5){\line(0,-1){1}}

\put(19.5,28){\line(0,-1){4}}

\put(20,27.5){\line(1,0){7}} \put(20,26.5){\line(1,0){7}}
\put(20,25.5){\line(1,0){3}} \put(20,24.5){\line(1,0){2}}
\put(20,27.5){\line(0,-1){3}} \put(21,27.5){\line(0,-1){3}}
\put(22,27.5){\line(0,-1){3}} \put(23,27.5){\line(0,-1){2}}
\put(24,27.5){\line(0,-1){1}} \put(25,27.5){\line(0,-1){1}}
\put(26,27.5){\line(0,-1){1}}\put(27,27.5){\line(0,-1){1}}


\put(12.5,21.5){\line(1,0){6}} \put(12.5,20.5){\line(1,0){6}}
\put(12.5,19.5){\line(1,0){6}} \put(12.5,21.5){\line(0,-1){2}}
\put(13.5,21.5){\line(0,-1){2}} \put(14.5,21.5){\line(0,-1){2}}
\put(15.5,21.5){\line(0,-1){2}} \put(16.5,21.5){\line(0,-1){2}}
\put(17.5,21.5){\line(0,-1){2}} \put(18.5,21.5){\line(0,-1){2}}

\put(19,22){\line(0,-1){4}}

\put(19.5,21.5){\line(1,0){8}} \put(19.5,20.5){\line(1,0){8}}
\put(19.5,19.5){\line(1,0){2}} \put(19.5,18.5){\line(1,0){2}}
\put(19.5,21.5){\line(0,-1){3}} \put(20.5,21.5){\line(0,-1){3}}
\put(21.5,21.5){\line(0,-1){3}} \put(22.5,21.5){\line(0,-1){1}}
\put(23.5,21.5){\line(0,-1){1}} \put(24.5,21.5){\line(0,-1){1}}
\put(25.5,21.5){\line(0,-1){1}} \put(26.5,21.5){\line(0,-1){1}}
\put(27.5,21.5){\line(0,-1){1}}


\put(31,33.5){\line(1,0){6}} \put(31,32.5){\line(1,0){6}}
\put(31,31.5){\line(1,0){6}} \put(31,30.5){\line(1,0){3}}
\put(31,33.5){\line(0,-1){3}} \put(32,33.5){\line(0,-1){3}}
\put(33,33.5){\line(0,-1){3}} \put(34,33.5){\line(0,-1){3}}
\put(35,33.5){\line(0,-1){2}} \put(36,33.5){\line(0,-1){2}}
\put(37,33.5){\line(0,-1){2}}

\put(37.5,34){\line(0,-1){4}}

\put(38,33.5){\line(1,0){7}} \put(38,32.5){\line(1,0){7}}
\put(38,31.5){\line(1,0){4}} \put(38,30.5){\line(1,0){4}}
\put(38,33.5){\line(0,-1){3}} \put(39,33.5){\line(0,-1){3}}
\put(40,33.5){\line(0,-1){3}} \put(41,33.5){\line(0,-1){3}}
\put(42,33.5){\line(0,-1){3}} \put(43,33.5){\line(0,-1){1}}
\put(44,33.5){\line(0,-1){1}} \put(45,33.5){\line(0,-1){1}}


\put(30,27.5){\line(1,0){7}} \put(30,26.5){\line(1,0){7}}
\put(30,25.5){\line(1,0){6}} \put(30,24.5){\line(1,0){2}}
\put(30,27.5){\line(0,-1){3}} \put(31,27.5){\line(0,-1){3}}
\put(32,27.5){\line(0,-1){3}} \put(33,27.5){\line(0,-1){2}}
\put(34,27.5){\line(0,-1){2}} \put(35,27.5){\line(0,-1){2}}
\put(36,27.5){\line(0,-1){2}}\put(37,27.5){\line(0,-1){1}}

\put(37.5,28){\line(0,-1){4}}

\put(38,27.5){\line(1,0){8}} \put(38,26.5){\line(1,0){8}}
\put(38,25.5){\line(1,0){4}} \put(38,24.5){\line(1,0){3}}
\put(38,27.5){\line(0,-1){3}} \put(39,27.5){\line(0,-1){3}}
\put(40,27.5){\line(0,-1){3}} \put(41,27.5){\line(0,-1){3}}
\put(42,27.5){\line(0,-1){2}} \put(43,27.5){\line(0,-1){1}}
\put(44,27.5){\line(0,-1){1}} \put(45,27.5){\line(0,-1){1}}
\put(46,27.5){\line(0,-1){1}}


\put(29.5,21.5){\line(1,0){7}} \put(29.5,20.5){\line(1,0){7}}
\put(29.5,19.5){\line(1,0){7}} \put(29.5,18.5){\line(1,0){1}}
\put(29.5,21.5){\line(0,-1){3}} \put(30.5,21.5){\line(0,-1){3}}
\put(31.5,21.5){\line(0,-1){2}} \put(32.5,21.5){\line(0,-1){2}}
\put(33.5,21.5){\line(0,-1){2}} \put(34.5,21.5){\line(0,-1){2}}
\put(35.5,21.5){\line(0,-1){2}} \put(36.5,21.5){\line(0,-1){2}}

\put(37,22){\line(0,-1){4}}

\put(37.5,21.5){\line(1,0){9}} \put(37.5,20.5){\line(1,0){9}}
\put(37.5,19.5){\line(1,0){3}} \put(37.5,18.5){\line(1,0){3}}
\put(37.5,21.5){\line(0,-1){3}} \put(38.5,21.5){\line(0,-1){3}}
\put(39.5,21.5){\line(0,-1){3}} \put(40.5,21.5){\line(0,-1){3}}
\put(41.5,21.5){\line(0,-1){1}} \put(42.5,21.5){\line(0,-1){1}}
\put(43.5,21.5){\line(0,-1){1}} \put(44.5,21.5){\line(0,-1){1}}
\put(45.5,21.5){\line(0,-1){1}} \put(46.5,21.5){\line(0,-1){1}}


\put(28.5,15.5){\line(1,0){8}} \put(28.5,14.5){\line(1,0){8}}
\put(28.5,13.5){\line(1,0){7}} \put(28.5,15.5){\line(0,-1){2}}
\put(29.5,15.5){\line(0,-1){2}} \put(30.5,15.5){\line(0,-1){2}}
\put(31.5,15.5){\line(0,-1){2}} \put(32.5,15.5){\line(0,-1){2}}
\put(33.5,15.5){\line(0,-1){2}} \put(34.5,15.5){\line(0,-1){2}}
\put(35.5,15.5){\line(0,-1){2}} \put(36.5,15.5){\line(0,-1){1}}

\put(37,16){\line(0,-1){4}}

\put(37.5,15.5){\line(1,0){10}} \put(37.5,14.5){\line(1,0){10}}
\put(37.5,13.5){\line(1,0){3}} \put(37.5,12.5){\line(1,0){2}}
\put(37.5,15.5){\line(0,-1){3}} \put(38.5,15.5){\line(0,-1){3}}
\put(39.5,15.5){\line(0,-1){3}} \put(40.5,15.5){\line(0,-1){2}}
\put(41.5,15.5){\line(0,-1){1}} \put(42.5,15.5){\line(0,-1){1}}
\put(43.5,15.5){\line(0,-1){1}} \put(44.5,15.5){\line(0,-1){1}}
\put(45.5,15.5){\line(0,-1){1}}\put(46.5,15.5){\line(0,-1){1}}
\put(47.5,15.5){\line(0,-1){1}}


\put(51.5,33.5){\line(1,0){7}} \put(51.5,32.5){\line(1,0){7}}
\put(51.5,31.5){\line(1,0){7}} \put(51.5,30.5){\line(1,0){4}}
\put(51.5,33.5){\line(0,-1){3}} \put(52.5,33.5){\line(0,-1){3}}
\put(53.5,33.5){\line(0,-1){3}} \put(54.5,33.5){\line(0,-1){3}}
\put(55.5,33.5){\line(0,-1){3}} \put(56.5,33.5){\line(0,-1){2}}
\put(57.5,33.5){\line(0,-1){2}} \put(58.5,33.5){\line(0,-1){2}}

\put(59,34){\line(0,-1){4}}

\put(59.5,33.5){\line(1,0){8}} \put(59.5,32.5){\line(1,0){8}}
\put(59.5,31.5){\line(1,0){5}} \put(59.5,30.5){\line(1,0){5}}
\put(59.5,33.5){\line(0,-1){3}} \put(60.5,33.5){\line(0,-1){3}}
\put(61.5,33.5){\line(0,-1){3}} \put(62.5,33.5){\line(0,-1){3}}
\put(63.5,33.5){\line(0,-1){3}} \put(64.5,33.5){\line(0,-1){3}}
\put(65.5,33.5){\line(0,-1){1}} \put(66.5,33.5){\line(0,-1){1}}
\put(67.5,33.5){\line(0,-1){1}}


\put(50.5,27.5){\line(1,0){8}} \put(50.5,26.5){\line(1,0){8}}
\put(50.5,25.5){\line(1,0){7}} \put(50.5,24.5){\line(1,0){3}}
\put(50.5,27.5){\line(0,-1){3}} \put(51.5,27.5){\line(0,-1){3}}
\put(52.5,27.5){\line(0,-1){3}} \put(53.5,27.5){\line(0,-1){3}}
\put(54.5,27.5){\line(0,-1){2}} \put(55.5,27.5){\line(0,-1){2}}
\put(56.5,27.5){\line(0,-1){2}}\put(57.5,27.5){\line(0,-1){2}}
\put(58.5,27.5){\line(0,-1){1}}

\put(59,28){\line(0,-1){4}}

\put(59.5,27.5){\line(1,0){9}} \put(59.5,26.5){\line(1,0){9}}
\put(59.5,25.5){\line(1,0){5}} \put(59.5,24.5){\line(1,0){4}}
\put(59.5,27.5){\line(0,-1){3}} \put(60.5,27.5){\line(0,-1){3}}
\put(61.5,27.5){\line(0,-1){3}} \put(62.5,27.5){\line(0,-1){3}}
\put(63.5,27.5){\line(0,-1){3}} \put(64.5,27.5){\line(0,-1){2}}
\put(65.5,27.5){\line(0,-1){1}} \put(66.5,27.5){\line(0,-1){1}}
\put(67.5,27.5){\line(0,-1){1}} \put(68.5,27.5){\line(0,-1){1}}


\put(50,21.5){\line(1,0){8}} \put(50,20.5){\line(1,0){8}}
\put(50,19.5){\line(1,0){8}} \put(50,18.5){\line(1,0){2}}
\put(50,21.5){\line(0,-1){3}} \put(51,21.5){\line(0,-1){3}}
\put(52,21.5){\line(0,-1){3}} \put(53,21.5){\line(0,-1){2}}
\put(54,21.5){\line(0,-1){2}} \put(55,21.5){\line(0,-1){2}}
\put(56,21.5){\line(0,-1){2}} \put(57,21.5){\line(0,-1){2}}
\put(58,21.5){\line(0,-1){2}}

\put(58.5,22){\line(0,-1){4}}

\put(59,21.5){\line(1,0){10}} \put(59,20.5){\line(1,0){10}}
\put(59,19.5){\line(1,0){4}} \put(59,18.5){\line(1,0){4}}
\put(59,21.5){\line(0,-1){3}} \put(60,21.5){\line(0,-1){3}}
\put(61,21.5){\line(0,-1){3}} \put(62,21.5){\line(0,-1){3}}
\put(63,21.5){\line(0,-1){3}} \put(64,21.5){\line(0,-1){1}}
\put(65,21.5){\line(0,-1){1}} \put(66,21.5){\line(0,-1){1}}
\put(67,21.5){\line(0,-1){1}} \put(68,21.5){\line(0,-1){1}}
\put(69,21.5){\line(0,-1){1}}


\put(49,15.5){\line(1,0){9}} \put(49,14.5){\line(1,0){9}}
\put(49,13.5){\line(1,0){8}} \put(49,12.5){\line(1,0){1}}
\put(49,15.5){\line(0,-1){3}} \put(50,15.5){\line(0,-1){3}}
\put(51,15.5){\line(0,-1){2}} \put(52,15.5){\line(0,-1){2}}
\put(53,15.5){\line(0,-1){2}} \put(54,15.5){\line(0,-1){2}}
\put(55,15.5){\line(0,-1){2}} \put(56,15.5){\line(0,-1){2}}
\put(57,15.5){\line(0,-1){2}} \put(58,15.5){\line(0,-1){1}}

\put(58.5,16){\line(0,-1){4}}

\put(59,15.5){\line(1,0){11}} \put(59,14.5){\line(1,0){11}}
\put(59,13.5){\line(1,0){4}} \put(59,12.5){\line(1,0){3}}
\put(59,15.5){\line(0,-1){3}} \put(60,15.5){\line(0,-1){3}}
\put(61,15.5){\line(0,-1){3}} \put(62,15.5){\line(0,-1){3}}
\put(63,15.5){\line(0,-1){2}} \put(64,15.5){\line(0,-1){1}}
\put(65,15.5){\line(0,-1){1}} \put(66,15.5){\line(0,-1){1}}
\put(67,15.5){\line(0,-1){1}} \put(68,15.5){\line(0,-1){1}}
\put(69,15.5){\line(0,-1){1}} \put(70,15.5){\line(0,-1){1}}


\put(48.5,9.5){\line(1,0){9}} \put(48.5,8.5){\line(1,0){9}}
\put(48.5,7.5){\line(1,0){9}} \put(48.5,9.5){\line(0,-1){2}}
\put(49.5,9.5){\line(0,-1){2}} \put(50.5,9.5){\line(0,-1){2}}
\put(51.5,9.5){\line(0,-1){2}} \put(52.5,9.5){\line(0,-1){2}}
\put(53.5,9.5){\line(0,-1){2}} \put(54.5,9.5){\line(0,-1){2}}
\put(55.5,9.5){\line(0,-1){2}} \put(56.5,9.5){\line(0,-1){2}}
\put(57.5,9.5){\line(0,-1){2}}

\put(58,10){\line(0,-1){4}}

\put(58.5,9.5){\line(1,0){12}} \put(58.5,8.5){\line(1,0){12}}
\put(58.5,7.5){\line(1,0){3}} \put(58.5,6.5){\line(1,0){3}}
\put(58.5,9.5){\line(0,-1){3}} \put(59.5,9.5){\line(0,-1){3}}
\put(60.5,9.5){\line(0,-1){3}} \put(61.5,9.5){\line(0,-1){3}}
\put(62.5,9.5){\line(0,-1){1}} \put(63.5,9.5){\line(0,-1){1}}
\put(64.5,9.5){\line(0,-1){1}} \put(65.5,9.5){\line(0,-1){1}}
\put(66.5,9.5){\line(0,-1){1}} \put(67.5,9.5){\line(0,-1){1}}
\put(68.5,9.5){\line(0,-1){1}} \put(69.5,9.5){\line(0,-1){1}}
\put(70.5,9.5){\line(0,-1){1}}
\end{picture}}
\caption{Bi-diagrams of degree $3$ minimal relations}\label{bidi3}

\end{figure}
Notice that, if $t>2u$, then the bi-shape
$(\gamma_u^\st|\lambda_u^t)$ is a trivial extension of
$(\gamma_u^{\,2u}|\lambda_u^{2u})$ (Proposition \ref{propretract}).
In the same vein, if $t>2u-1$, then the bi-shape
$(\rho_u^t|\sigma_u^t)$ is a trivial extension of
$(\rho_u^{2u-1}|\sigma_u^{2u-1})$. In other words, every time that
the size of minors $t$ increases by $1$, a new type of minimal cubic
relations between $t$-minors comes up:
\begin{itemize}
\item[(i)] If $t$ is even, then
    $(\gamma_{t/2}^\st|\lambda_{t/2}^t)$ starts a new series of
     minimal cubic relations between $t'$-minors, $t'\geq t$.
\item[(ii)] If $t$ is odd, then
    $(\rho_{(t+1)/2}^t|\sigma_{(t+1)/2}^t)$ starts a new series of
    new minimal cubic relations between $t'$-minors,
    $t'\geq t$.
\end{itemize}

\begin{remark}
We have the coarse decomposition:
\[
\Sym^3(E\otimes F^*) =\Bl(\Sym^3E\otimes
\Sym^3F^*\Br)  \oplus
\Bl(L_{(2,1)}E\otimes L_{(2,1)}F^*\Br)  \oplus
\Bl(\Bw^3E\otimes \Bw^3 F^*\Br).
\]
Therefore, as in Remark \ref{place of quadratics}, we would
like to place each $(\gamma_u|\lambda_u)$ and
$(\rho_u|\sigma_u)$ in an irreducible $H$-module:
\begin{itemize}
\item[(i)] $(\rho_u|\sigma_u)$ is in $L_{(2,1)}E
    \otimes L_{(2,1)}F^*$;
\item[(ii)] $(\gamma_u|\lambda_u)$ is in $\Sym^3E
    \otimes \Sym^3 F^*$ if $u$ is even;
\item[(iii)] $(\gamma_u|\lambda_u)$ is in $\Bw^3E
    \otimes \Bw^3F^*$ if $u$ is odd.
\end{itemize}
For $(\rho_u|\sigma_u)$ the $H$-irreducible has been explicitly
determined in the proof of \ref{shapeink3odd}.  For each the
remaining  two cases one inspects the unique predecessor.
\end{remark}

\subsection{Highest bi-weight vectors of the cubic minimal relations}
\label{sec_cubic_weight}

For completeness, in this subsection we will describe the polynomial
corresponding to the highest bi-weight vector of any cubic relation
we found up to now.

\subsubsection{Higehst bi-weight vectors of even cubics}

We need the following lemma:

\begin{lemma}
For all $u=1,\dots,\lfloor t/2\rfloor$ set
$K=\{t-2u+1,\dots,t+u\}\subseteq \NN$. The highest weight
vector of $L_{\gamma_u}V\subseteq \Tensor^3\bigl(\Bw^tV\bigr)$
is: {\small
\begin{equation}\label{rowpart}
\sum_{A,B,C}(-1)^{A,B,C}(e_1\wedge \dots \wedge e_{t-2u}\wedge
e_{K\setminus A})\tensor (e_1\wedge \dots \wedge
e_{t-2u}\wedge e_{K\setminus B})\tensor (e_1\wedge \dots
\wedge e_{t-2u}\wedge e_{K\setminus C})
\end{equation}}
where the sum runs over the $3$-partitions $A,B,C$ of $K$ such
that $|A|=|B|=|C|=u$.
\end{lemma}
\begin{proof}
Set $v=3u$ and consider a $\kk$-vector space $V_0$ of
dimension $v$ with the $\operatorname{SL}(V_0)$-action. Let us
look at
\[
\Bw^v V_0^*\xrightarrow{\alpha}\Tensor^v V_0^*\xrightarrow{\beta}
\Bw^uV_0^*\tensor \Bw^uV_0^*\tensor \Bw^uV_0^*\xrightarrow{\delta}
\Bw^{2u} V_0\tensor \Bw^{2u} V_0\tensor \Bw^{2u} V_0
\]
Here $\alpha$ is antisymmetrization, namely:
\[
x_1\wedge\dots\wedge x_v\mapsto
\sum_\pi (-1)^\pi x_{\pi(1)}\tensor\dots\tensor x_{\pi(v)}.
\]
In particular $\aa$ is $\operatorname{SL}(V_0)$-equivariant.
The map $\beta$ cuts $x_1\tensor\dots\tensor x_v$ into blocks
and maps tensor power to exterior power, so it is also
$\operatorname{SL}(V_0)$-equivariant:
\[
x_1\tensor\dots\tensor x_v\mapsto (x_1\wedge\dots\wedge x_u)\tensor
(x_{u+1}\wedge\dots\wedge x_{2u})\tensor (x_{2u+1}\wedge\dots\wedge x_{v})
\]
The map $\delta$ is the one that gives the isomorphism as
$\operatorname{SL}(V_0)$-modules of $\Bw^uV_0^*$ and $\Bw ^{2u}
V_0$. It is defined, with respect to a fixed basis
$e_1,\dots,e_v$ of $V_0$, as follows. Let $e_1^*,\dots,e_v^*$
be the dual basis of $V_0^*$. Then
\[
e_{i_1}^*\wedge\dots\wedge
e_{i_u}^*=(-1)^{i_1,\dots,i_u}e_I^*\quad\mapsto\quad
(-1)^{i_1,\dots,i_u}(-1)^{u(u-1)/2}(-1)^{I,J\setminus
I}e_{J\setminus I}.
\]
where $I=\{i_1,\dots,i_u\}$ and $J=\{1,\dots,v\}$. The
constant sign $(-1)^{u(u-1)/2}$ is irrelevant for our purpose,
and we will omit it. We can combine the two other signs as
\[
(-1)^{i_1,\dots,i_u}(-1)^{I,J\setminus I}=(-1)^{i_1,\dots,i_u,J\setminus I}.
\]
Now we can start from the $\operatorname{SL}(V_0)$-invariant
$e_1^*\wedge\dots\wedge e_v^*\in \Bw^v V_0^*$ and apply our
maps. Because all the maps involved are
$\operatorname{SL}(V_0)$-equivariant we end with an
$\operatorname{SL}(V_0)$-invariant in $\Tensor^3\Bw^{2u} V_0$. We
can assume that the permutations are increasing in the three
blocks since the sign $(-1)^{i_1,\dots,i_u}$ ``corrects'' the
order. Thus we get
\[
e_1^*\wedge\dots\wedge e_v^*\mapsto (u!)^3\sum_{F,G,H} (-1)^{F,G,H}
(-1)^{F,G\cup H}(-1)^{G,F\cup H}(-1)^{H,F\cup G}e_{J\setminus F}
\tensor e_{J\setminus G}\tensor e_{J\setminus H}.
\]
where the sum is extended over all the $3$-partitions $F,G,H$
of $J$ such that $|F|=|G|=|H|=u$. But $ (-1)^{F,G\cup
H}(-1)^{G,F\cup H}(-1)^{H,F\cup G}$ is constant, namely equal
to $(-1)^{3u^2}$. Removing the constant sign and dividing by
$(u!)^3$ yields
\[
\sum_{F,G,H} (-1)^{F,G,H}e_{J\setminus F}\tensor e_{J\setminus G}
\tensor e_{J\setminus H}.
\]
Since the above element is $\operatorname{SL}(V_0)$-invariant,
the element of $\Tensor^3\bigl(\Bw^tV\bigr)$ of the statement,
namely
\[
\sum_{A,B,C}(-1)^{A,B,C}(e_1\wedge \dots \wedge e_{t-2u}\wedge e_{K\setminus A})
\tensor (e_1\wedge \dots \wedge e_{t-2u}\wedge e_{K\setminus B})
\tensor (e_1\wedge \dots \wedge e_{t-2u}\wedge e_{K\setminus C}),
\]
is $\UU_-(V)$-invariant. Moreover, its weight is $\tl
\gamma_u$, therefore it is the highest weight vector of
$L_{\gamma_u}V$.
\end{proof}

By a similar and simpler construction (we need not to dualize) we
can compute also the highest weight vector of
$L_{\lambda_u}W^*\subseteq \Tensor^3\Bl(\Bw^tW^*\Br)$:
\begin{equation}\label{columnpart}
\sum_{L,M,N} (-1)^{L,M,N}(f_1^*\wedge \dots \wedge f_{t-u}^*\wedge f_L^*)\tensor (f_1^*\wedge \dots
\wedge f_{t-u}^*\wedge f_M^*)\tensor (f_1^*\wedge \dots \wedge f_{t-u}^*\wedge f_N^*)
\end{equation}
where the sum is extended over the $3$-partitions $L,M,N$ of
$\{t-u+1,\dots,t+2u\}$ such that $|L|=|M|=|N|=u$.

Now we tensor the row part \eqref{rowpart} and the column part
\eqref{columnpart} together and pass to the symmetric power
$(S_t)_3$. Then each monomial appears $6$ times since the
monomials only depend on the set of pairs $(A,L),(B,M)$ and
$(C,N)$, but not on their order anymore. Permuting these sets
does not change the sign, since both row and column factor
change by the same sign. So, dividing by $6$, we can assume
that $A,B,C$ is ordered lexicographically. The element we get
is the highest bi-weight vector of $L_{\gamma_u}V\tensor
L_{\lambda_u}W^*\subseteq (J_t)_3$. In particular it is a
minimal relation between $t$-minors of degree $3$, and all the
cubic shape relations of type $(\gamma_u|\lambda_u)$ are in the
$G$-space generated by it. Explicitly, such a relation is:
\begin{equation}\label{expldeg3}
\gg_u=\sum_{{A,B,C}\atop{L,M,N}} (-1)^{A,B,C} (-1)^{L,M,N}
[P,K\setminus A|Q,L][P,K\setminus B|Q,M][P,K\setminus C|Q,N],
\end{equation}
where the sum runs over the $3$-partitions $A,B,C$ of
$K=\{t-2u+1,\dots,t+u\}$ and $L,M,N$ of $\{t-u+1,\dots,t+2u\}$
such that $|A|=|B|=|C|=|L|=|M|=|N|=u$ and $A,B,C$ are ordered
lexicographically. Moreover $P=\{1,\dots,t-2u\}$ and
$Q=\{1,\dots,t-u\}$. Of course there are also the mirror
relations of \eqref{expldeg3}, namely the ones obtained
switching columns by rows. We will denote  them by $\gg_u'$.

\begin{remark}
As already noticed, the highest bi-weight vector of
$(\gamma_u^\st|\lambda_u^t)$ is a trivial extension of the highest
bi-weight vector of the same irreducible $G$-representation relative
to $2u$-minors, namely $(\gamma_u^{\,2u}|\lambda_u^{2u})$. In this
case $\gg_u$ assumes the following simpler form:
\[\gg_u=\sum_{{A,B,C}\atop{L,M,N}} (-1)^{A,B,C} (-1)^{L,M,N}
[K\setminus A|1,\ldots ,u,L][K\setminus B|1,\ldots ,u,M][K\setminus C|1,\ldots ,u,N],\]
where the sum runs over the $3$-partitions $A,B,C$ of
$K=\{1,\dots,u\}$ and $L,M,N$ of $\{u+1,\dots,4u\}$
such that $|A|=|B|=|C|=|L|=|M|=|N|=u$ and $A,B,C$ are ordered
lexicographically.
\end{remark}


\subsubsection{Highest bi-weight vectors of odd cubics}

Let $u$ be a positive integer in $\{2,\ldots ,\allowbreak
\lceil t/2 \rceil\}$. We are going to describe the highest
weight vector of one of the copies of
\[
L_{\rho_u}V\subseteq \Tensor^3\Bl(\Bw^tV\Br).
\]
To this aim, let us set
\[
v_1=\sum_{A,B,C}(-1)^{A,B,C}(e_P\wedge e_{K\setminus A})\otimes
(e_P\wedge e_{K\setminus B}\wedge e_{t+u})\otimes (e_P\wedge e_{K\setminus C})
\]
and
\[
v_2=\sum_{A,B,C}(-1)^{A,B,C}(e_P\wedge e_{K\setminus B}\wedge e_{t+u})
\otimes (e_P\wedge e_{K\setminus A}) \otimes (e_P\wedge e_{K\setminus C}),
\]
where the sums run over the partitions $A,B,C$ of
$K=\{t-2u+2,\ldots ,t+u-1\}$ such that $|A|=|C|=u-1$ and
$|B|=u$. Moreover, $P=\{1,\ldots ,t-2u+1\}$.
It is not difficult to show that the element
\begin{equation}\label{hvodd1}
v=v_1-v_2 \in\bigotimes^3 \Bl(\Bw^tV\Br)
\end{equation}
is a nonzero $\UU_-(V)$-invariant. Moreover, since $v$ has weight
$\tl\rho_u$, it is the highest weight vector of one of the copies of $L_{\rho_u}V$.

In the same vein, let $u\in \{2,\dots,\lceil t/2 \rceil\}$.
Analogously to above, we set
\[
w_1=\sum_{L,M,N}(-1)^{L,M,N}(f_Q^*\wedge f_L^*\wedge f_{t-u+1}^*)\otimes
(f_Q^*\wedge f_M^*) \otimes (f_Q^*\wedge f_N^*\wedge f_{t-u+1}^*)\]
and
\[
w_2=\sum_{L,M,N}(-1)^{L,M,N}(f_Q^*\wedge f_M^*)\otimes (f_Q^*\wedge f_L^*\wedge
f_{t-u+1}^*)\otimes (f_Q^*\wedge f_N^*\wedge f_{t-u+1}^*),
\]
where the sums run over the partitions $L,M,N$ of
$\{t-u+2,\ldots ,t+2u-1\}$ such that $|L|=|N|=u-1$ and $|M|=u$.
Furthermore, $Q=\{1,\ldots ,t-u\}$. Once again, it is not difficult to show that the element
\begin{equation}\label{hvodd2}
w=w_1-w_2 \in \Tensor^3\Bl(\Bw^tW^*\Br)
\end{equation}
is a nonzero $\UU_+(W)$-invariant. Moreover, since $w$ has weight
$\tl\sigma_u$, it is the highest weight vector of one of the copies of $L_{\sigma_u}W^*$.

Now, as for the even relations, we tensor the row part
\eqref{hvodd1} and the column part \eqref{hvodd2} together and
pass to the symmetric power $(S_t)_3$. After some manipulations, we get:
{\small
\begin{multline}\label{expldeg3odd}
\hh_u=\\\sum_{{A,B,C}\atop{L,M,N}} (-1)^{A,B,C} (-1)^{L,M,N}(
[P,K\setminus A|Q,L,t-u+1][P,K\setminus B,t+u|Q,M][P,K\setminus C|Q,N,t-u+1] \\
-[P,K\setminus A|Q,M][P,K\setminus
B,t+u|Q,L,t-u+1][P,K\setminus C|Q,N,t-u+1])
\end{multline}
}where the sum runs over the $3$-partitions $A,B,C$ of
$K=\{t-2u+2,\dots,t+u-1\}$ and $L,M,N$ of $\{t-u+2,\dots,t+2u-1\}$
such that $|A|=|C|=|L|=|N|=u-1$, $|B|=|M|=u$ and $A$ is less than $C$
lexicographically. Moreover $P=\{1,\dots,t-2u+1\}$ and
$Q=\{1,\dots,t-u\}$. Of course there are also the mirror
relations of \eqref{expldeg3odd}, namely the ones obtained
switching columns by  rows. We will denote them by $\hh_u'$.

We believe that the relations found so far generate $J_t$.
Despite of the rather limited evidence for this belief we
formulate it as a conjecture:

\begin{conjecture}\label{conjrel}
For all $t,m,n$ the polynomials $\ff_{u,v}$ of degree $2$
and $\gg_u,\gg_u',\hh_u,\hh_u'$ of degree $3$ (as far as they are
defined in $S_t(m,n)$) generate $J_t(m,n)$ as a
$G$-ideal. Equivalently,
\begin{align*}
J_t/(S_t)_1J_t\cong\Dirsum_{{u,v\in\{0,\ldots ,t\}}\atop{{u+v \text{ even}}\atop{u\neq v}}}
 L_{\tau_u}V\tensor L_{\tau_v}W^*&\dirsum
 \Dirsum_{u\leq m-t\atop 2u\leq n-t} L_{\gamma_u}V\tensor L_{\lambda_u} W^*\dirsum
 \Dirsum_{u\leq m-t\atop2u\leq n-t+1} L_{\rho_u}V\tensor L_{\sigma_u} W^*\\
 &\dirsum
 \Dirsum_{u\leq n-t\atop 2u\leq m-t} L_{\lambda_u}V\tensor L_{\gamma_u} W^*\dirsum
 \Dirsum_{u\leq n-t\atop2u\leq m-t+1} L_{\sigma_u}V\tensor L_{\rho_u} W^*.
\end{align*}
\end{conjecture}

It is remarkable that all the minimal relations we have found,
are not only shape relations, but even of single $\Bw^t$-type.
If one could prove that all minimal relations were of single
$\Bw^t$-type, then the conjecture would be proved as well: as
shown in \cite{BrVa}, the conjecture indeed lists all (minimal)
relations $(\gamma|\lambda)$ in which both $\gamma$ and
$\lambda$ are of single $\Bw^t$-type.

\begin{remark}
(a) How far $J_t(m,n)$ is from the ideal generated by the
degree $2$ relations can be easily analyzed in the case $t=2$,
$m=3$, $n=4$. In this case the ideal $Q$ generated by the
Plücker relations is a complete intersection ideal of height
$6$ and $Q=J_2(3,4)\cap P$ where $P$ is a prime
ideal  generated by $Q$ and $(S_2)_{(3,3|3,3)}$. In fact,
there is an automorphism of $S_2(3,4)$ carrying $J_2(3,4)$ into
$P$ so that $S_2(3,4)/P\cong A_2(3,4)$. Furthermore for  $t=2$,
$m=3$, $n=5$ the ideal of quadrics in $J_2(3,5)$ generate an ideal whose codimension is smaller than
that of $J_2(3,5)$ itself.

(b) It was shown in \cite{BC5} that the ideal $I$ generated by
the Plücker relations and the degree $3$ relations in the
irreducible representation of the bi-shapes
$(\gamma_1|\lambda_1)$ and $(\lambda_1|\gamma_1)$ satisfy the
following property: $J_t(m,n)_P=I_P$ for all prime ideals
$P\supset J_t(m,n)$ for which $(A_t)_P$ is non-singular. (The
singular locus of $A_t$ was also determined in \cite{BC5}.)
This supports Conjecture \ref{conjrel} to some extent.

(c) Using the methods of Section \ref{sec4xn}, we have computed the relations of the algebra of $2$-minors of a
symmetric $n\times n$ matrix with $n\le 5$ rows. Surprisingly the
ideal is generated in degree $2$.

(d) On the other hand, De Negri \cite[Theorem 1.4]{DN} proved that there are no
degree $2$ relations between $2t$-pfaffians of an alternating
$n\times n$ matrix for arbitrary $t$ and $n$ in characteristic $0$.
\end{remark}

\subsection{Determinantal relations}\label{DetREl}
It turns out that the relations $\gg_1$
are of determinantal type. In the following we want to indicate
how to construct more such determinantal highest bi-weight vectors
in $J_t$. They are closely related to the structure of
\[
\Bw^dE\tensor \Bw^dF^*.
\]
As usual by now, we (have) set $E=\Bw^t V$, $F=\Bw^t W$, and
$H=\GL(E)\times \GL(F)$. The $H$-bi-shape associated with the above $H$-module is $(d|d)$.

If we order the canonical bases of $E$ and $F$ in such a way
that this linear order extends the componentwise partial order
on $t$-uples of the canonical bases in $V$ and $W$,
respectively, then the unipotent subgroup of $G$ that we used
to define $U$-invariants embeds naturally into the unipotent
subgroup of $H$ defined by the order of the base elements.
Therefore $H$-$U$-invariants are in particular
$G$-$U$-invariants (in self explaining notation). The
$H$-$U$-invariant of shape $(d|d)$ is simply the $d$-minor of the
matrix whose entries represent the pairs of the first $d$ base
vectors in $E$ and $F$, respectively. It remains to fill the
rows and columns of this $d$-minor in such a way that one
obtains an element in $J_t$.

The crucial point is that the linear extension of the partial order
is not unique (apart from trivial cases). Therefore we can choose
different orders in $E$ and $F$ to produce asymmetric $G$-shapes in
$S_t$, and these belong automatically to the ideal $J_t$ of
relations. In particular, the third largest element of a basis of
$E$ can be chosen in two ways,
and this
fact leads to the cubic relation $\gg_1$.

We discuss the case $t=2$ in detail. In each triangle of Example
\ref{triangle} below we take an initial subsequence of each
row, and if no such subsequence sticks out further to the right
than the one above it, the total sequence formed by
concatenation represents an initial sequence in a suitable
linear extension of the partial order. The entries of each
subsequence represent a hook of type $(u+1,1,1,\dots,1)\vdash
2u$. The concatenated sequence represents a shape that is
obtained by nesting these hooks, and thus we obtain
$\GL(V)$-shapes in $\Bw^d E$.

\begin{example}\label{triangle}
Let us consider the following two initial segments
corresponding to two different linear extensions of the
componentwise order:
\[
\begin{matrix}
\diamond&{\bf 12}&{\bf 13}&{\bf 14}&{\bf 15}&16&\cdots\\
\bullet&&{\bf 23}&{\bf 24}&{\bf 25}&26&\cdots\\
\ast&&&{\bf 34}&{\bf 35}&36&\cdots\\
&&&&45&46&\cdots\\
&&&&&&\cdots
\end{matrix}
\hspace{10mm}
\begin{matrix}
\diamond&{\bf 12}&{\bf 13}&{\bf 14}&{\bf 15}&{\bf 16}&\cdots\\
\bullet&&{\bf 23}&{\bf 24}&{\bf 25}&26&\cdots\\
\ast&&&{\bf 34}&35&36&\cdots\\
&&&&45&46&\cdots\\
&&&&&&\cdots
\end{matrix}
\]
The elements of the initial segments are written in bold. The
symbols at the beginning of the rows should help to understand
how to get the following bi-shape from the two above initial
segments:
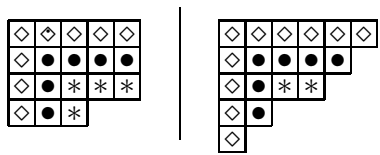
\[
{\setlength{\unitlength}{0.7mm}
\begin{picture}(30,30)(-5,-5)

\put(-18.3,16.5){$\cdot$}

\put(-24,16){$\diamond$} \put(-19,16){$\diamond$}
\put(-14,16){$\diamond$} \put(-9,16){$\diamond$}
\put(-4,16){$\diamond$} \put(-24,11){$\diamond$}
\put(-24,6){$\diamond$} \put(-24,1){$\diamond$}

\put(-19,11){$\bullet$} \put(-14,11){$\bullet$}
\put(-9,11){$\bullet$} \put(-4,11){$\bullet$}
\put(-19,6){$\bullet$} \put(-19,1){$\bullet$}

\put(-14,6){$\ast$} \put(-9,6){$\ast$} \put(-4,6){$\ast$}
\put(-14,1){$\ast$}

\put(-25,20){\line(1,0){25}} \put(-25,15){\line(1,0){25}}
\put(-25,10){\line(1,0){25}} \put(-25,5){\line(1,0){25}}
\put(-25,0){\line(1,0){15}}

\put(-25,20){\line(0,-1){20}} \put(-20,20){\line(0,-1){20}}
\put(-15,20){\line(0,-1){20}} \put(-10,20){\line(0,-1){20}}
\put(-5,20){\line(0,-1){15}} \put(0,20){\line(0,-1){15}}

\put(7.5,22.5){\line(0,-1){25}}

\put(16,16){$\diamond$} \put(21,16){$\diamond$}
\put(26,16){$\diamond$} \put(31,16){$\diamond$}
\put(36,16){$\diamond$} \put(41,16){$\diamond$}
\put(16,11){$\diamond$} \put(16,6){$\diamond$}
\put(16,1){$\diamond$} \put(16,-4){$\diamond$}

\put(21,11){$\bullet$} \put(26,11){$\bullet$}
\put(31,11){$\bullet$} \put(36,11){$\bullet$}
\put(21,6){$\bullet$} \put(21,1){$\bullet$}

\put(26,6){$\ast$} \put(31,6){$\ast$}

\put(15,20){\line(1,0){30}} \put(15,15){\line(1,0){30}}
\put(15,10){\line(1,0){25}} \put(15,5){\line(1,0){20}}
\put(15,0){\line(1,0){10}} \put(15,-5){\line(1,0){5}}

\put(15,20){\line(0,-1){25}} \put(20,20){\line(0,-1){25}}
\put(25,20){\line(0,-1){20}} \put(30,20){\line(0,-1){15}}
\put(35,20){\line(0,-1){15}} \put(40,20){\line(0,-1){10}}
\put(45,20){\line(0,-1){5}}

\end{picture}}
\]
The $G$-$U$-invariant of the above bi-shape is the determinant
of the $9\times 9$-matrix in Figure \ref{MatRel}.
\begin{figure}[htb]
$$\begin{pmatrix}
[12|12]&[12|13]&[12|14]&[12|15]&[12|16]&[12|23]&[12|24]&[12|25]&[12|34]\\
[13|12]&[13|13]&[13|14]&[13|15]&[13|16]&[13|23]&[13|24]&[13|25]&[13|34]\\
[14|12]&[14|13]&[14|14]&[14|15]&[14|16]&[14|23]&[14|24]&[14|25]&[14|34]\\
[15|12]&[15|13]&[15|14]&[15|15]&[15|16]&[15|23]&[15|24]&[15|25]&[15|34]\\
[23|12]&[23|13]&[23|14]&[23|15]&[23|16]&[23|23]&[23|24]&[23|25]&[23|34]\\
[24|12]&[24|13]&[24|14]&[24|15]&[24|16]&[24|23]&[24|24]&[24|25]&[24|34]\\
[25|12]&[25|13]&[25|14]&[25|15]&[25|16]&[25|23]&[25|24]&[25|25]&[25|34]\\
[34|12]&[34|13]&[34|14]&[34|15]&[34|16]&[34|23]&[34|24]&[34|25]&[34|34]\\
[35|12]&[35|13]&[35|14]&[35|15]&[35|16]&[35|23]&[35|24]&[35|25]&[35|34]
\end{pmatrix}$$
\caption{A matrix representing a determinantal relation}\label{MatRel}
\end{figure}
Such a determinant is a degree $9$ relation between $2$-minors.
\end{example}

Surprisingly, we have found the complete $\GL(V)$-decomposition
of $\Bw^d E$ for $t=2$: see \cite[p.~65]{We} for this classical
plethysm.

\section{Upper bounds on the degree of minimal
relations}\label{Upper}

In this section we will give some evidence for the truth of
Conjecture \ref{conjrel}. For $t=2$ we have the strongest support:
(i) the conjecture holds for $m\times n$-matrices with $m\leq 4$ and
$m=n=5$; (ii)  the only minimal relations of degree $3$ are those
described in the conjecture; (iii) there are no minimal relations in
degree $4$. For $t=3$ we have verified that there are no other
minimal relations in degree $3$. For arbitrary $t$, we can give some
combinatorial support for the conjecture.

The results for $t=2$ and $t=3$ depend on computer
calculations. For them an a priori bound on the degree of a
minimal generator of $J_t$ is very useful, and we will derive
from the Castelnuovo-Mumford regularity of $A_t$.

\subsection{Castelnuovo-Mumford regularity of $A_t$}

For the computation of the Castelnuovo-Mumford regularity we will use  the initial algebra $\init(A_t)$ of $A_t$ with respect to a
diagonal term order $\prec$ on $R$, i.e.\ a term order such that
$\init([i_1\dots i_p|j_1\dots j_p])=x_{i_1j_1}\cdots x_{i_pj_p}$.

\begin{thm}\label{regularity}
Apart from the cases discussed in Remark \ref{partcases},
we have:
\begin{itemize}
\item[(i)] If $m+n-1<\lfloor mn/t \rfloor$, then
\[
\reg(A_t) =  mn-\lceil mn/t \rceil .
\]
\item[(ii)] if $m+n-1\geq \lfloor mn/t  \rfloor$, then  \[
\reg(A_t) = mn-\lfloor m(n+k_0)/t \rfloor .
\]
where $k_0 = \lceil (tm+tn-mn)/(m-t)
    \rceil$. \end{itemize}
\end{thm}
\begin{proof}
We know that  $A_t$ is  Cohen-Macaulay  by \cite[Theorem 7.10]{BC2} and has dimension $mn$ by \cite[Proposition 10.16]{BV} because we have excluded the cases listed in Remark \ref{partcases}.
Therefore we have  $\reg(A_t)=\dim A_t+a(A_t)=mn+a(A_t)$.
Here $a(A_t)$ is the $a$-invariant of $A_t$, i.e. the opposite of the least degree of a non-zero element of the graded canonical module of $A_t$.
 Since by \cite[Theorem 7.10]{BC2}   $\init(A_t)$ is Cohen-Macaulay as well, we have $a(A_t)=a(\init(A_t))$.  Hence it is enough to compute $a(\init(A_t))$. Denote by $\omega$ the canonical module of $\init(A_t)$.
By \cite[Lemma 3.3]{BC1} $\omega$ is  generated by the monomials of the form $\init(\Delta)$,  where
$\Delta$ is a product of minors of $X$ of shape
$\gamma=(\gamma_1,\dots,\gamma_h)$ where $|\gamma|=td$, $h<d$
and such that $\xx = \prod x_{ij}$ divides $\init(\Delta)$.
Therefore, if $d$  is the least number for which such a
$\Delta$ exists, then  $\reg(A_t)=mn-d$.

First let us consider case (i). Set $d_0=\lceil mn/t\rceil$. Of
course $\omega_d=0$ if $d<d_0$. We have to show that
$\omega_{d_0}\neq 0$.  Let us pick the
unique integer $r_0$ with $0\leq r_0<t$ and $mn+r_0=d_0t$. Of course
we can consider a product $\Delta\in R$ of minors of shape $\gamma =
(m^{n-m+1},(m-1)^2,(m-2)^2,\dots,1^2,r_0)\vdash d_0t$ (possibly the
partition has to be reordered but this does not matter) such that
$\xx$ divides $\init(\Delta)$. (a) If $r_0=0$, then $\gamma$ is a
partition of $m+n-1$ rows: since $m+n-1<d_0$ by hypothesis, we have
$\init(\Delta)\in \omega$. (b) If $r_0>0$, the partition $\gamma$
consists of $m+n$ rows. Then $d_0=\lceil mn/t \rceil = \lfloor mn/t
\rfloor + 1$, so the hypothesis implies $m+n < d_0$. Therefore also
$\init(\Delta)\in \omega$ if $r_0>0$. We are done in case (i).

Now let us discuss case (ii). Notice that the integer $k_0$
introduced in (ii) is larger than $0$. Let $p_0$ be the unique
integer such that $0\leq p_0<t$ and $m(n+k_0)=d_0t+p_0$. We can
consider a product $\Delta\in R$ of minors of shape $\gamma =
(m^{k_0+n-m},(m-1)^2,\dots,1^2,m-p_0)$ such that $\xx$ divides
$\Delta$. This is a partition of $d_0t$ with $k_0+n+m-1$ parts.
By the choice of $k_0$, one can verify that $k_0+n+m-1<d_0$. So
$\init(\Delta)\in \omega$, which implies $\omega_{d_0}\neq 0$.

To complete the proof showing that $\omega_d=0$ whenever
$d<d_0$, we need the following easy lemma.
\begin{lemma}\label{decofx}
With a little abuse of notation set $X=\{x_{ij} \ : \
i=1,\dots, m, \ j=1,\dots,n \}$. Define a poset structure on
$X$ in the following way:
\[
x_{ij}\leq x_{hk} \ \ \text{ if \ \ \ }i=h
\text{ and $j=k$ \ \ or \ \ $i<h$ and $j<k$}.
\]
Suppose that $X=X_1 \cup \dots \cup X_h$ where each $X_i$ is a
chain, i.e. any two elements of $X_i$ are comparable, and set
$N=\sum_{i=1}^h|X_i|$. Then
\[
h\geq N/m+m-1.
\]
\end{lemma}

Let us take a product of minors $\Delta=\delta_1 \cdots
\delta_h$ such that $\init(\Delta)\in \omega$. Let $\lambda$ be
the shape of $\Delta$ and suppose by contradiction that
$|\lambda|=td$ with $d<d_0$. For $i=1,\dots,h$ set
\[
X_i=\{x_{pr}: x_{pr}|\init(\delta_i)\}.
\]
Since $\xx$ divides $\init(\Delta)$, with the notation of Lemma
\ref{decofx} we have that $X=\cup_{i=1}^h X_i$ where each $X_i$
is a chain with respect to the order defined on $X$. So, by
Lemma \ref{decofx},
\[
h\geq dt/m+m-1.
\]
We recall that $d_0t=mn+mk_0-p_0$, where $0\leq p_0 < t $. Of
course we can write $dt=mn+ms-q$ in a unique way, where $0\leq
q < m$. Before going on, notice that $k_0$ is the smallest
natural number $k$ satisfying the inequality
\[
 m+n+k-1 < \Bl\lfloor \frac{m(n+k)}t\Br\rfloor.
\]
Of course $s\leq k_0$. There are two cases:
\begin{itemize}
\item[(i)] If $s=k_0$, consider the inequalities
\[
m+n+(s-1)-1 = \frac{dt+q}m + m -2 < \frac{dt}m+m-1\leq h \leq d-1.
\]
Notice that, since $d<d_0$, we have that $q\geq p_0+t$.
Moreover $m<2t$, otherwise we would be in case (i) of the
theorem. Thus
\[
 d-1 = \frac{m(n+s)-q-t}t \leq \Bl\lfloor \frac{m(n+(s-1))}t \Br\rfloor.
 \]
The inequalities above contradicts the minimality of $k_0$.
\item[(ii)] If $s<k_0$, then
\[
n+s+m -1 = \frac{dt+q}m + m -1 \leq h < d  = \frac{m(n+s)-q}t
\leq \Bl\lfloor \frac{m(n+s)}t\Br\rfloor.
\]
Once again, this yields a contradiction to the minimality of
$k_0$.
\end{itemize}
To sum up, we deduce that $\omega_d=0$ whenever $d<d_0$, and
this completes the proof.
\end{proof}

\begin{remark}
Let us look at the cases in Theorem \ref{regularity}.
\begin{itemize}
\item[(i)] If $X$ is a square matrix, that is $m=n$, one can
    easily check that we are in case (i) of Theorem
    \ref{regularity} if and only if  $m\geq 2t$.
\item[(ii)] The natural number $k_0$ of Theorem
    \ref{regularity} may be very large. For instance,
    consider the case $t=m-1$ and $n=m+1$ with $m\geq 3$.  One can easily check that
    we are in the case (ii) of Theorem \ref{regularity}. In
    this case we have $k_0=m^2-2m-1$. Therefore Theorem
    \ref{regularity} yields
    \[\reg(A_{m-1}(m,m+1))=m.\]
\end{itemize}
\end{remark}

Since $\reg(J_t)=\reg(A_t)+1$ bounds  the degree of a minimal
generator of $J_t$ from above, Theorem \ref{regularity} yields
an upper bound for the degree of a minimal relation between
$t$-minors.

\subsection{Minimal relations between $2$-minors of a $4\times n$-matrix}
\label{sec4xn} In this subsection we will indicate how to
verify Conjecture \ref{conjrel} for $J_2(m,n)$ with $m\leq 4$
and $m=n=5$. The following result enables us to succeed in this
case by machine computation. It says that a minimal relation
between $t$-minors of a $m\times n$-matrix must already
``live'' in a $m\times (m+t)$-matrix.

\begin{thm}\label{independencefromn}
Let $(\gamma |\lambda)$ be a minimal representation in
$J_t(m,n)$. Then $(\gamma|\lambda)$ is a minimal representation
already in $J_t(m,m+t)$. In particular, if we denote the
highest degree of a minimal generator of $J_t(m,n)$ by
$d(t,m,n)$ , then
\[
d(t,m,n)\leq d(t,m,m+t).
\]
\end{thm}

\begin{proof}
Suppose that $(\gamma|\lambda)$ is a minimal irreducible
representation of $J_t(m,n)$. Then it is impossible that
$(\gamma|\lambda)$ has only asymmetric bi-predecessors by Theorem
\ref{GandH}. Since $\gamma_1\le m$, we must have $\lambda_1\le
m+t$. Therefore it is a minimal irreducible representation in
$J_t(m,m+t)$.
\end{proof}

The above theorem, together with Theorem \ref{regularity},
gives the following upper bound (far from what we have
suggested in \ref{conjrel}) for the degree of a minimal
relation between $t$-minors.

\begin{corollary}\label{generalupper}
The degree of a minimal generator of $J_t(m,n)$ is bounded
above by
\[
m(m+t) - m - \Bl\lfloor \frac{m^2}{t} \Br\rfloor +1 \quad ( \ \leq m^2 + (t-2)m  \ ).
\]
\end{corollary}

However, Theorem \ref{independencefromn} means that the
validity of Conjecture \ref{conjrel} for $2$-minors of a
$3\times 5$-matrix implies it for $2$-minors of any $3\times n$
matrix etc. In particular, Theorem \ref{independencefromn}
implies $d(2,3,n)\leq d(2,3,5)$ and $d(2,4,n)\leq d(2,4,6)$.
Actually we can show that $d(2,3,5)\leq 3$ and $d(2,4,6)\leq 3$
by computer.

For \textsf{Singular} \cite{singular} the computation of $J_2(3,5)$
is a matter of seconds, but for $J_2(4,6)$ it is already a matter of
days, and we succeeded only because of the following strategy that
uses a priori informations on the Hilbert function of $A_2(m,n)$.
Since the decomposition of the graded pieces of $A_t(m,n)$ can be
computed easily via \eqref{decat}, an evaluation of the hook formula
then yields its $\kk$-dimension. (A tool for this computation had
already been developed for \cite{BC1}.)
\begin{itemize}
\item[(1)] Set $J=J_2(4,6)$, $S=S_2(4,6)$ and, for any $d\in
    \NN$, let $J_{\leq d}\subseteq J$ denote the ideal
    generated by the polynomials in $J$ of degree at most $d$.
    Corollary \ref{generalupper} implies that    $J=J_{\leq 13}$.
\item[(2)] By elimination (for instance see Eisenbud
    \cite[15.10.4]{eisenbud}), \textsf{Singular} computes a set of generators
    of  $J_{\leq 3}$.
\item[(3)] For the degree reverse lexicographical term order, we
    compute a Gr\"obner basis of $J_{\leq 3}$ up to degree
    $13$. So we get $B=\init(J_{\leq 3})_{\leq 13}$.
\item[(4)] The Hilbert function of $S/B$ is
    easily computable, and we have
\[\Hf_{S/J_{\leq 3}}(d)\leq \Hf_{S/B}(d),
\]
where equality holds for $d\leq 13$.
\item[(5)] Since $J_{\leq 3}\subseteq J$, we have
    $\Hf_{S/B}(d)\geq \Hf_{S/J}(d)$. However, comparing $\Hf_{S/B}(d)$
    with the precomputed $\Hf_{S/J}(d)$ shows equality for $d\leq 13$.
    This implies  $J_{\leq 3}=J_{\leq 13}$, and we are done.
\end{itemize}
The verification of $d(2,5,5)=3$ is of similar complexity as
that of $d(2,4,6)=3$. However, already $d(2,5,6)$ or $d(3,4,7)$
seem to be out of reach for present day machines.

\begin{thm}\label{settled cases}
Conjecture \ref{conjrel} is true for $2$-minors of a $4\times
n$-matrix and a $5\times 5$-matrix. In particular, the only
minimal relations between $2$-minors of a $4\times n$-matrix
and a $5\times 5$-matrix, respectively, are quadratics and
cubics.

The conjecture also holds for $3$-minors of a $5\times
5$-matrix.
\end{thm}
\begin{proof}
Subsection \ref{subdeg2} implies that the only degree $2$
minimal generators of $J_t(m,n)$ are those listed in
\ref{conjrel}. The discussion above shows that there are no
minimal generators of degree larger than $3$ in $J_2(4,n)$, as
predicted by Conjecture \ref{conjrel}. It remains to show that
the only degree $3$ minimal generators are in the $G$-module
generated by $\gg_1$ and $\gg_1'$. This will follow by a result
of the next subsection, in which we prove this fact without
restriction on $m$.

The statement on $3$-minors of a $5\times5$-matrix follows from
Proposition \ref{dual}.
\end{proof}

\subsection{Cubic minimal relations between
$2$-minors}\label{sec_min_cub}

In this subsection we are going to show that the only cubic
minimal relations between $2$-minors are those predicted in
Conjecture \ref{conjrel}, i.e.~those in the $G$-space generated
by $\gg_1$ and by $\gg_1'$. So we want to show that among the
bi-diagrams $(\gamma|\lambda)$ in $\Sym^3\bigl(\Bw^2V\tensor
\Bw^2W^*\bigr)$only $(\gamma_1|\lambda_1)$ (see
\eqref{kerdiage} is minimal in $J_2(m,n)$. Since $\gamma$ and
$\lambda$ are partitions of $6$, the $U$-invariant of
$(\gamma|\lambda)$ is in $S_2(6,6)$. This means that, for our
task, it suffices to consider a $6\times 6$-matrix. Since this
format is presently unreachable by machine calculation, we must
reduce it further.

\begin{prop}\label{extreme}
Let $t=2$. Then the following hold:
\begin{enumerate}
\item The bi-shapes $(2d | 2d)$, $(2d -1 ,1|2d-1, 1)$ and
    $(2d|2d-2, 2)$ have multiplicity $1$ in $S_2$ (provided
    the vector space dimensions are sufficiently large).
\item the bi-shape $(2d| 2d-1, 1)$ does not appear in
    $S_2$.
\end{enumerate}
\end{prop}

\begin{proof}
In the following we use the plethysm \eqref{plethsym}. Let
$E=\Bw^2 V$. Evidently $(2d)$ has multiplicity $1$ in
$\Tensor^d E$, and since $(2d|2d)$ has multiplicity $1$ in
$A_2$, it must have multiplicity $1$ in the intermediate $S_2$.
Since $(2d)$ appears only in $\Sym^dE$ and $(2d-2,2)$ has
multiplicity $1$ in the latter, the multiplicity of $(2d|2d-2,
2)$ in $S_t$ must also be $1$.

We claim that $(2d -1, 1)$ is of single $\Bw^2$-type
$\mu=(2,1,\dots,1)\vdash d$. In fact, $(2d -1, 1)$ has
multiplicity $d-1$ in $\Tensor^d E$ by Pieri's rule, and this
is also the multiplicity of $\mu$ in the
$\GL(E)$-decomposition. Therefore it is enough that $(2d-1, 1)$
appears in $L_\mu E$. Note that
\[
\Sym^{d-1} E\tensor E=
\Sym^{d} E\dirsum L_\mu E,
\]
the non-even successor $(2d-1, 1)$ of $(2(d-1))$ must land in
$L_\mu E$. Proposition \ref{single_char} finishes the argument.
\end{proof}

Proposition \ref{extreme} allows us to reduce the problem to
size $4\times 5$. The symmetric bi-shapes $(6|6)$ and
$(5,1|5,1)$ have multiplicity $1$ in $S_2$, occur in $A_2$ and
so do not belong to $J_2$. The asymmetric shape $(6|5,1)$ is
not represented in $S_2$ at all, and for the reduction to size
$4\times 5$ it remains to rule out the bi-shape $(6|4,2)$ of
multiplicity $1$, since the other bi-shapes involving $(6)$ do
not have symmetric bi-predecessors and $(5,1|5,1)$ has
multiplicity $1$.

We claim that $(S_t)_{(6|4,2)}$ is
contained in the ideal generated by $(S_t)_{(4|2,2)}$. Because of Proposition \ref{Segre}
it is enough to prove this in $\Sym\bigl(\Bw^2V\bigr)\ \sharp\
\Sym\bigl(\Bw^2W^*\bigr)$. But in the Segre product it is
enough to consider the single factors, and the algebra
$\Sym\bigl(\Bw^2V\bigr)$ is well-understood; see Abeasis and
Del Fra \cite{AF}.

For a $4\times 5$-matrix it is not hard to check by machine
computation that
\[
\dim_{\kk}(J_2)_3=\dim_{\kk}(((J_2)_{\le 2})_3)+\dim_{\kk}
\bigl((L_{\gamma_1}V\tensor
L_{\lambda_1}W^*)\oplus (L_{\lambda_1}V\tensor
L_{\gamma_1}W^*)\bigr)
\]
where $\gamma_1$ and $\lambda_1$ are defined in
\eqref{kerdiage}. Thus the only subspace missing from
$(((J_2)_{\le 2})_3$ is indeed the one predicted by Conjecture
\ref{conjrel}.

\subsection{No minimal degree $4$ relations for
$2$-minors}\label{nogeg4t=2}

In this subsection we explain how to verify that there are no
degree $4$ minimal relations between $2$-minors. The same
method has been applied to exclude any further degree $3$
minimal relations for $3$ than those listed in Conjecture
\ref{conjrel}.

The first step is the computation of the $\GL(V)$-decomposition
of $(S_2)_4$ by \textsf{Lie}. (The reader can reconstruct the
decomposition from Table \ref{Plt24} and \eqref{CauchyH}.) As
documented above, it is already known that $J_t$ is generated
in degree $2$ and $3$ if $m=n=5$ or $m=4$. This excludes all
bi-shapes from being minimal relations that fit into matrices
of these sizes. After their exclusion and the exclusion of the
cases covered by Theorem \ref{independencefromn} and
Proposition \ref{extreme}, there remain $6$ critical bi-shapes
of multiplicity $1$ in $J_2$, and $2$ other critical bi-shapes
of multiplicity $2$. (A further reduction would be possible via
Theorem \ref{GandH}(ii).)

We want to show that they are not minimal relations. For
multiplicity $1$ it is enough to find a $U$-invariant of the
given shape in $(S_2)_1\cdot  (J_t)_3$. For example, let
$(\gamma|\lambda)=(6,2|7,1)$. We try to ``derive'' it from
$(\alpha|\beta)=(4,2|6)$. To this end we first compute $g_{\alpha
\leadsto \gamma}$ and $g_{\beta \leadsto \lambda}$ by
\eqref{fromltog}. Then we consider $g_{\alpha \leadsto
\gamma}\tensor g_{\beta \leadsto \lambda}$ as an element of
$\Tensor^4 (E\tensor F^*)$ (by reordering the factors) and pass to
$\Sym^4 (E\tensor F^*)$ by identifying summands that differ only by
a simultaneous permutation of the $E$- and $F^*$-factors. The
result, unless it is $0$, is the desired $U$-invariant, and it
could be found for all critical shapes of multiplicity $1$. (Note
that the computations depend on tableaux, not just diagrams, and not
every choice of tableaux may work.)

If the critical shape has multiplicity $2$, then we must derive
two linearly independent $U$-invariants from asymmetric
bi-shapes in degree $3$. Again, this has turned out successful.
The algorithm has been implemented by the authors in
\textsf{Singular}. It is available with all input and output
files from \cite{YS}.

To justify the claim that $g_{\alpha \leadsto \gamma}\tensor
g_{\beta \leadsto \lambda}$ indeed gives an element in $(S_2)_1
J_t$, note that we take a sum of tensors $(\YY_A(a)\tensor
a')\tensor(\YY_B(b)\tensor b')$ where $A$ and $B$ are tableaus of
shapes $\alpha$ and $\beta$, respectively. Therefore
$\YY_A(a)\tensor\YY_B(b)$ represents an element of $J_2$, and
$a'\tensor b'$ represents an element of $(S_2)_1$.

A similar computation has been carried out for $t=3$ in order to
exclude any further minimal degree $3$ relations. It would certainly
be possible to reach degree $6$ for $t=2$ or degree $4$ for $t=3$.
However, then the algorithm must be re-implemented in a faster
programming language, and its use must be further automatized.

As said in the introduction, we do not expect that relations are
minimal because the algebra structure of $S_t$ is too weak to
exclude them from the ideal generated by the bi-predecessors that
represent relations. If the following conjecture had a positive
answer, then one would be a good deal closer to proving Conjecture
\ref{conjrel}. It reflects the computational experience described
above.

\begin{conjecture}\label{StrongAlg}
Let $(\gamma|\lambda)$ be a bi-shape occurring in $(S_t)_\mu$, and
suppose that there exists a $1$-predecessor $\mu'$ of $\mu$ that
contains a $t$-bi-predecessor $(\alpha|\beta)$ of
$(\gamma|\lambda)$. Then  $(\gamma|\lambda)$ does occur in
$(S_t)_1(S_t)_{(\alpha|\beta)}$.
\end{conjecture}

\subsection{$T$-shape relations}\label{T-shape_rel}
In Theorem \ref{shapeink3} we have identified cubic minimal
relations in $J_t$ that are even $T$-shape relations. In this
subsection we want to show that these cubic relations and the
degree $2$ relations are the only $T$-shape relations in $J_t$.
We recall that an asymmetric bi-shape $(\gamma|\lambda)$ is
called a $T$-shape relation if it has only symmetric
bi-predecessors of multiplicity $1$ in $T_t$. This is a very
strong condition:

\begin{prop}\label{SiTsha} Let $\gamma,\lambda$ be
$(t,d)$-admissible partitions. Then the following are equivalent:
\begin{itemize}
\item[(i)] $(\gamma|\lambda)$ is a $T$-shape relation;
\item[(ii)] $(\gamma|\lambda)$ has a unique bi-predecessor;
\item[(iii)] $\gamma$ and $\lambda$ are both of
    multiplicity $1$ in $\Tensor^d \Bw V$ and,
    respectively, in $\Tensor^d \Bw W^*$ and have the same
    predecessor.
\end{itemize}
\end{prop}

\begin{proof}
Let us just mention the main fact on which the easy proof
relies. If $\gamma$ or $\lambda$ has more than one predecessor,
then $(\gamma|\lambda)$ must have an asymmetric bi-predecessor in
$T_t$, simply because we can pair any predecessors $\alpha$ and
$\beta$ of $\gamma$ and $\lambda$, respectively, to a
bi-predecessor $(\alpha|\beta)$ in $T_t$ (but not necessarily in
$S_t$!). This argument has already been used in the proof of
Proposition \ref{tensordec}.
\end{proof}

In view of Proposition \ref{SiTsha} we must first classify the
shapes of multiplicity $1$ in $\Tensor^d \Bw V$. To this end, we
need the following lemma, whose proof is easy.


\begin{lemma}\label{onepredecessor}
Let $\lambda=(\lambda_1,\dots, \lambda_k)$ be a diagram of
$\Tensor^d \Bw^t  V$. Then $\lambda$ has a unique predecessor
if and only if either $\lambda_1=\dots =\lambda_k$ ($\lambda$
is a rectangle) or there exist $i$ such that $\lambda_1=\dots
=\lambda_i>\lambda_{i+1}=\dots =\lambda_k$ and $k=d$ ($\lambda$
is called a \emph{fat hook}).
\end{lemma}

\begin{corollary}\label{mult1}
For a diagram $\lambda=(\lambda_1,\dots, \lambda_k)$ of
$\Tensor^d \Bw^t  V$, $d\ge2$, the following are equivalent:
\begin{itemize}
\item[(i)]  $\lambda$ has multiplicity $1$ in $\Tensor \Bw^t
    V$;
\item[(ii)] $\lambda$ has  a single predecessor $\lambda'$, and
    $\lambda'$ has again a single predecessor;
\item[(iii)] $\lambda$ is a rectangle or fat hook of type (a)
    $\lambda_2=\dots = \lambda_d$ or (b) $\lambda_1=\dots =\lambda_{d-1}$.
\end{itemize}
\end{corollary}

\begin{remark}
Diagrams $\lambda$ of multiplicity $1$ in $\Tensor^d \Bw^t V$
are clearly of single $\Bw^t$-type $\mu$ where $\mu$ itself has
multiplicity $1$, and therefore represents either $\Bw^d(\Bw^tV)$ or
$\Sym^d(\Bw^tV)$. We leave it to the reader to locate the diagrams in
\ref{mult1}(iii).
\end{remark}

The following theorem shows that we have found all $T$-shape
relations. We suppress the case $d=2$ since all asymmetric
shapes of degree $2$ are evidently $T$-shape relations.

\begin{thm}\label{nomoreshaperelations}
The only $T$-shape relations of degree $d\geq 3$ are the cubics
$(\gamma_u|\lambda_u)$ and $(\lambda_u|\gamma_u)$ where $u$
varies in $\{1,\dots,\lfloor t/2\rfloor\}$.
\end{thm}

\begin{proof}
Let $(\gamma|\lambda)$ be a $T$-shape relation. We can assume
that at least one of the two diagrams, say $\gamma$, is not a
trivial extension, in other words has at most $d-1$ rows.

Suppose first that $\gamma_2=\dots=\gamma_d$. Since
$\gamma_d=0$, $\gamma$ is a rectangle with one row of $td$
boxes, and it is evident that we cannot find a second successor
to the predecessor $(t(d-1))$ of $\gamma$ that is different
from $\gamma$ but has itself multiplicity $1$. (The only
exception would be $d=2$ in which case we could pair $\gamma$
with $(2t-u,u)$.)

Now suppose that $\gamma_1=\dots=\gamma_{d-1}$. Since
$\gamma_d=0$ by assumption on $\gamma$, it must be a rectangle
with $d-1\ge 2$ rows. Again we look at the predecessor
$\alpha=(\gamma_1,\dots,\gamma_{d-2},\gamma_{d-1}-t)$. Scanning
the successors of $\alpha$, we see that there is another
successor $\lambda\neq \gamma$ of multiplicity $1$ if and only
if $d=3$, $t$ is even, and $\gamma_{2}=3t/2$. Then
$\lambda=(2t,t/2,t/2)$, as desired.
\end{proof}

\begin{remark}
Let $(\gamma|\lambda)$ a bi-diagram in $T_t$ and let
$(\alpha_1|\beta_1),\dots,(\alpha_N|\beta_N)$ be its
bi-prede\-ce\-ssors counted with multiplicities in $T_t$ (so it
may happen that $(\alpha_i|\beta_i)=(\alpha_j|\beta_j)$ also if
$i\neq j$). Suppose that exactly $k$ of the bi-predecessors of
$(\gamma|\lambda)$, say
$(\alpha_1|\beta_1),\dots,(\alpha_k|\beta_k)$, are in $K_t$: If
one of the copies of $\Wc\tensor \Wl^*$ is in $K_t$ and does
not belong to
\begin{multline*}
\bigl((L_{\alpha_1}V\tensor L_{\beta_1}W^*) \oplus \dots \oplus
(L_{\alpha_k}V\tensor L_{\beta_k}W^*)\bigr)\tensor (T_t)_1\\
\oplus (T_t)_1 \tensor \bigl((L_{\alpha_1}V\tensor
L_{\beta_1}W^*) \oplus \dots \oplus (L_{\alpha_k}V\tensor
L_{\beta_k}W^*)\bigr),
\end{multline*}
then it is actually minimal in $K_t$. In particular, exploiting
\eqref{decat}, a strategy to find minimal generators of $K_t$
could be the following: to track down asymmetric bi-diagrams
$(\gamma|\lambda)$ such that $k<N/2$ or symmetric ones such
that $k<\lfloor N/2\rfloor$. However, one can easily realize
that this situation happens if and only if $(\gamma|\lambda)$
is asymmetric, has multiplicity $1$ in $T_t$ and its unique bi-predecessor
is symmetric. By Theorem \ref{nomoreshaperelations}, such a
bi-diagram has to be among those predicted in Conjecture
\ref{conjrel}.
\end{remark}

\subsection{No other degree $3$ shape relations}\label{nomoredeg3sh}

As usual let $E=\Bw^t V$ and $F=\Bw^t W$. In \ref{secoddrelations},
we could found some minimal cubic relations between $t$-minors
because the asymmetric bi-diagrams $(\rho_u|\sigma_u)$ in
$\Sym^3(E\otimes F^*)$ have no asymmetric bi-predecessors in
$\Sym^2(E\otimes F^*)$. Below we will show that, apart from
$(\gamma_u|\lambda_u)$ and $(\rho_u|\sigma_u)$, no other bi-diagrams
in $\Sym^3(E \otimes F^*)$ have this property. In other words, there
exist no other degree $3$ shape relations than the known ones. For
the proof of this claim we need the following easy remark:

\begin{remark}\label{sym=alt}
Suppose that $\lambda$ is a $(t,3)$-admissible partition with
$k$ predecessors in $\Tensor^2E$, say $a$ of them in
$\Sym^2E$ and the remaining $b=k-a$ in $\Bw^2E$. Then
$a-b\in\{-1,0,1\}$. To check this one has to use Lemma
\ref{decs2at}, noticing that
\[
\tau_{u-1} \mbox{ and }\tau_{u+1} \mbox{ are predecessors of }\lambda
\implies \tau_u \mbox{ is a predecessor of }\lambda.
\]
\end{remark}

Suppose that $(\gamma|\lambda)$ is an asymmetric bi-diagram in
$\Sym^3(E \otimes F^*)$ such that $\gamma$ has $h$ predecessors
and $\lambda$ has $k$ predecessors. We can assume that $1\leq
h\leq k$, because the issue is symmetric.
\begin{itemize}
\item[(i)] Suppose that $h\geq 2$ and $k\geq 3$. Then, by
    Remark \ref{sym=alt}, at least one of
    $\Sym^2E$ and $\Bw^2E$ contains (at least) two
    predecessors of $\lambda$ and one predecessor of
    $\gamma$. So in this case, we can deduce from
    \eqref{CauchyH} that $(\gamma|\lambda)$ has an asymmetric
    bi-predecessor which actually lives in $\Sym^2(E\otimes
    F^*)$.
\item[(ii)] Similar arguments finish the case $h=1$, $k\ge 4$.
\item[(iii)] If $h=k=1$, then we Theorem
    \ref{nomoreshaperelations} implies: either
    $(\gamma|\lambda)=(\gamma_u|\lambda_u)$ for some $u$, or
    $(\gamma|\lambda)$ has an asymmetric bi-predecessor in
    $\Tensor^2(E\otimes F^*)$. Moreover, since
    $(\gamma|\lambda)$ is in $\Sym^3(E\otimes F^*)$,
    such a bi-predecessor actually lives in $\Sym^2(E\otimes
    F^*)$.
\end{itemize}
We still need to deal with the cases $h=1$ and $k=2$, $h=1$ and
$k=3$, $h=2$ and $k=2$. These cases are a bit more tricky:

\begin{prop}
Any asymmetric bi-diagram in $\Sym^3(E \otimes F^*)$, different
from $(\gamma_u|\lambda_u)$, $(\rho_u|\sigma_u)$ and their
mirror images, has an asymmetric bi-predecessor in
$\Sym^2(E\otimes F^*)$.
\end{prop}
\begin{proof}
We keep the previous notation and continue with the remaining
cases.

(i)  $h=1$ and $k=2$. By Proposition \ref{propinthemiddle},
since $L_{(2,1)}F^*$ occurs with multiplicity $2$ in
$\Tensor^3F^*$, the irreducible $L_\lambda W^*$ occurs only
in $L_{(2,1)}F^*$, and neither in $\Sym^3F^*$ nor in
$\Bw^3F^*$. On the other hand, since $h=1$, $L_\gamma V$ has
to be in $\Sym^3E$ or in $\Bw^3E$, but not in
$L_{(2,1)}E$. Therefore $(\gamma|\lambda)$ cannot be in
$\Sym^3(E\otimes F^*)$ by \eqref{CauchyH}.

(ii) $h=1$ and $k=3$. Let us assume that $2$ of the
predecessors of $\lambda$ are in $\Sym^2F^*$ and $1$ in
$\Bw^2 F^*$. The symmetric case is analogous, and there are no
other cases by Remark \ref{sym=alt}. We claim that $L_\lambda
W^*$ is not in $\Bw^3 F^*$. By Pieri's formula, we know that
\[
\Bl(\Bw^2 F^*\Br)\otimes F^*
\cong \Bw^3 F^*\oplus L_{(2,1)}F^*.
\]
Notice that one copy of $L_\lambda W^*$ is in $L_{(2,1)}F^*$ by
Proposition \ref{propinthemiddle}. So, if $L_\lambda W^*$ were
in $\Bw^3 F^*$, then $\lambda$ would have $2$ predecessors in
$\Bw^2 F^*$, a contradiction.

It follows  that $L_\lambda W^*$ does not occur in $\Bw^3 F^*$.
Thus \eqref{CauchyH} implies that the only copy of $\Wc$ has to
be in $\Sym^3E$, and the only predecessor of $\gamma$ is in
$\Sym^2E$. Since $\lambda$ has $2$ predecessors in
$\Sym^2E$, $(\gamma|\lambda)$ has an asymmetric bi-predecessor
which really lives in $\Sym^2(E\otimes F^*)$ by
\eqref{CauchyH}.

If $h=k=2$. We want to show that, in this case, there exist $u$
and $v$ such that $\gamma \in \{\rho_u,\sigma_u\}$ and
$\lambda\in \{\rho_v, \sigma_v\}$. This is an immediate
consequence of the following easy fact: A $(t,3)$-admissible
diagram $\alpha=(\alpha_1,\alpha_2,\alpha_3)$ has $\ell$
predecessors if and only if
$\min\{\alpha_1-\alpha_2,\alpha_2-\alpha_3\}=\ell-1$. At this
point, one can easily check that, apart from the cases in which
$\gamma=\lambda$, $(\gamma|\lambda)=(\rho_u|\sigma_u)$ or
$(\gamma|\lambda)=(\sigma_u|\rho_u)$, the
bi-shape$(\gamma|\lambda)$ has always an asymmetric
bi-predecessor or in $\Sym^2E\otimes \Sym^2F^*$, or in
$\Bw^2E\otimes \Bw^2F^*$, and thus in $\Sym^2(E\otimes F^*)$ by
\eqref{CauchyH}.
\end{proof}

\begin{remark}
Using the plethysms computed by \textsf{Lie} we have checked
that there are no other shape relations than the known degree
$2$ and $3$  ones in the following cases: (i) $t=2,3,$ $d\le 5$ and (ii) $t=4,5,$ $d\le 4$.
\end{remark}

\end{document}